\documentclass[12pt, a4paper]{amsart}
\usepackage[utf8]{inputenc}
\usepackage{hyperref}
\usepackage{xcolor}
\usepackage{multirow}
\usepackage{amsmath,amsfonts,amssymb,amsthm,mathrsfs,mathtools} 
\newtheorem{lemma}{Lemma} 
\newtheorem{corollary}{Corollary}
\newtheorem{theorem}{Theorem}
\newtheorem*{theorem*}{Theorem}

\newcommand{\R}{\mathbb R}	
\newcommand{\C}{\mathbb C}		
\newcommand{\Z}{\mathbb Z}		
\newcommand{\N}{\mathbb N}		
\newcommand{\Q}{\mathbb Q}
\DeclareMathOperator{\sinc}{Sinc}
\DeclareMathOperator{\supp}{supp}
\DeclareMathOperator{\proj}{Proj}
\DeclareMathOperator{\PW}{PW}
\DeclareMathOperator{\Hardy}{H}

\usepackage{geometry}
\title{On pointwise convergence of the non-linear Fourier transform}
\author{Lukas Mauth}
\address{Mathematisches Institut, Endenicher Allee 60, 53115 Bonn}
\email{s6lumaut@uni-bonn.de}
\date{}

\begin{document}
	\maketitle

		\begin{abstract}
			 In  \cite{Poltoratski} A. Poltotaski proved an analog of Carleson's Theorem on almost everywhere convergence of Fourier series for a version of the non-linear Fourier transform. We aim to present his proof in full detail and elaborate on the ideas behind each step.		\end{abstract}

	\tableofcontents
	
	\section{Introduction}	
	We aim to present a proof for an analog of Carleson's Theorem for the non-linear Fourier transform. This result has recently been established by Alexei Poltoratski \cite{Poltoratski} and we will present his proof in full detail. We are going to follow his proof closely by not modifying his argument, but changing the order of steps and adding more intermediate results, making it easier to understand the proof. In particular, we aim to present the proof in such a way that anyone with a basic background in complex analysis can understand it. In this section, we are going to introduce the non-linear Fourier transform and draw the analogy to Carleson's Theorem on the linear Fourier transform. At the end of this section, we will describe the strategy of Poltoratski's proof.

	The next section aims to introduce the basic concepts needed from the theory of bounded analytic functions in the upper half-plane and Krein-de Branges theory. Besides these few preliminary results there is not much theory involved in the proof. One could go even as far to say that this proof is elementary in a sense that it almost purely relies upon well-known results.

	In the third section, we present Poltoratski's proof with all details included. Poltoratski's paper \cite{Poltoratski} is organized into $12$ sections. In our treatment of his proof, we are following his structure closely but shifting some sections around or dividing them into other sections. The biggest change is probably moving section $8$ from \cite{Poltoratski}, which establishes joint approximations of the Hermite-Biehler functions to the very end of our discussion to make the key elements of his argument stand out. In particular, every Lemma, Theorem, and Corollary in Poltoratski's paper has a corresponding result with the same statement in our version. All of our other results do not appear in Poltoratski's paper explicitly and we added them to simplify the understanding of his ideas. In the following table, there is a dictionary that allows going back and forth between Poltoratski's and our paper.

	In the fourth and final chapter we give a brief outlook on the open question whether $\arg a(t,s)$ and $\arg b(t,s)$ converge as well. In the present version of this proof we only showed that $|a(t,s)|$ and $|b(t,s)|$ converge for almost all $s \in \R.$ We will prove in some sense that with probability $1$ this question can be answered affirmatively and we manage to prove it completely for the special case $f \in L^1(\R).$ Another interesting quesiton is whether we can extract bounds on the maximal operator. This question was studied by G. Mnatsakanyan in \cite{Gevorg}.
	
\begin{center}
	\begin{tabular}{ |p{4cm}| |p{4cm}| } 
		\hline
		\multicolumn{2}{|c|}{Dictionary} \\
		\hline
		Our Version & Poltoratski's Version \cite{Poltoratski} \\
		\hline
		Theorem \ref{non-linear carleson}  & Theorem $1$ \\
		\hline
		Lemma \ref{Lemma_1} & Lemma $1$ \\
		Lemma \ref{Lemma_2} & Lemma $2$ \\
		Corollary \ref{Lemma_3} & Lemma $3$ \\
		Lemma \ref{Lemma_4} & Lemma $4$ \\
		Corollary \ref{corollary_1} & Lemma $5$ \\
		Lemma \ref{Lemma_6} & Lemma $6$ \\
		Lemma \ref{Lemma_7} & Lemma $7$ \\
		Lemma \ref{Lemma_8} & Lemma $8$ \\
		Lemma \ref{Lemma_9} & Lemma $9$ \\
		Theorem \ref{Lemma_10} & Lemma $10$ \\
		Lemma \ref{Lemma_11} & Lemma $11$ \\
		\hline
		Corollary \ref{corollary_2} & Corollary $1$ \\
		Lemma \ref{formulas_mif} & Corollary $2$ \\
		Corollary \ref{corollary_3} & Corollary $3$ \\
		Corollary \ref{corollary_4} & Corollary $4$ \\
		Corollary \ref{corollary_5} & Corollary $5$ \\
		\hline
		Corollary \ref{claim_1} & Claim $1$ \\
		Lemma \ref{claim_2} & Claim $2$ \\
		Corollary \ref{claim_3} & Claim $3$ \\
		\hline
	\end{tabular}
\end{center}

\textbf{Acknowldegements} I wish to thank Christoph Thiele for introducing me to non-linear Fourier Analysis, giving me this extensive project to work on and being a hospital and forecoming supervisor. Furthermore, I wish to thank Gevorg Mnatsakanyan with whom I spent countless hours understanding and discussing Poltoratski's proof. I wish to thank Alexei Poltoratski for his hospitality and patience when explaining us the remaining parts of the proof we did not understand. Moreover, I wish to thank the Number Theory group of Bonn, especially Valentin Blomer, Edgar Assing, Radu Toma and Gilles Felber for the opportunity to learn more about, and discussing topics in Number Theory, when working on this project was a little difficult from time to time. 
	
	\subsection{Carleson's Theorem}
	Let $f:[0,1] \rightarrow \C$ be a smooth function. We define for $n \in \Z$ the Fourier coefficients of $f$ by 
	
	\begin{equation}\label{fourier_coefficients}
	\hat{f}(n)=\int_{0}^{1} f(x)e^{-2\pi i n x} dx.
	\end{equation}
	
	\noindent
	We can use the Fourier coefficients $\hat{f}$ to expand our initial function $f$  by into a Fourier series
	
	\begin{equation}\label{fourier_series}
	f(x)=\sum_{n \in \Z} \hat{f}(n) e^{2 \pi i n x}.
	\end{equation}
	
	\noindent
	Indeed, since we assumed $f$ to be smooth it is easy to verify that this series converges uniformly to $f.$ Looking at \eqref{fourier_coefficients} we see that the Fourier coefficients are well-defined for the much larger class of $L^p-$functions with $1\leq p \leq \infty.$ It was desirable that $(\ref{fourier_series})$ remained true for $f\in L^p[0,1]$, i.e. that the Fourier series still converges to $f$ in some sense. Indeed one shows easily that for $f \in L^p[0,1]$ with $1<p<\infty$
	
	$$S_N f(x) := \sum_{|n| \leq N}  \hat{f}(n) e^{2 \pi i n x} \overset{L^p}{\longrightarrow} f(x),$$
	
	\noindent
	where the convergence is understood in the $L^p[0,1]-$norm. It is a much deeper question to ask if $S_N f$ converges to $f$ pointwise almost everywhere for $f\in L^p[0,1].$ This problem remained open for a long time. The case $p=1$ was settled first in 1923, when Kolmogorov \cite{Kolmogorov} constructed an integrable function $f$, whose Fourier series diverges almost everywhere.
	The first positive answer came from L. Carleson in 1966 \cite{Carleson} and is since celebrated under the name Carleson Theorem.
	
	\begin{theorem}[Carleson Theorem]
		Let $f$ be in $L^2[0,1].$ Then $S_N f$ converges pointwise to $f$ almost everywhere.
	\end{theorem}

	By inclusion of $L^p$-spaces, the above result generalizes immediately to $2\leq p \leq \infty.$ One year later R. Hunt \cite{Hunt} gave an affirmative answer for the remaining cases $1<p<2.$ In conclusion, the Fourier series of $f$ converges pointwise almost everywhere for any $f$ in $L^p[0,1]$ with $1<p\leq \infty.$

	\subsection{The non-linear Fourier transform}
	We are going to study scattering data of certain systems of differential equations. The associated scattering functions behave similarly to the classical Fourier transform, see below. In particular, there is a non-linear version of Carleson's Theorem which remained open for a long time and which we are going to prove.
	
	The scattering data we are going to consider is constructed from real Dirac systems on $\R_+.$
	Let $f\in L^2(\R_+).$ We define matrices 
	
	$$\Omega=\begin{pmatrix}
	0 && 1 \\
	-1 && 0
	\end{pmatrix}\text{, and }
	Q(t)=\begin{pmatrix}
	0 && f(t) \\
	f(t) && 0
	\end{pmatrix}.
	$$
	
	For a spectral parameter $z \in \C$ we consider the real Dirac system on $\R_+$
	
	\begin{equation}\label{dirac_system}
		\Omega \frac{d}{dt} X(t,z) = zX(t,z)-Q(t)X(t,z).
	\end{equation}
	
	A solution to the Dirac system \eqref{dirac_system} is a vector-valued function 
	
	$$X(t,z)=\begin{pmatrix}
	u(t,z) \\
	v(t,z)
	\end{pmatrix}.
	$$
	
	We will consider two solutions of the Dirac system satisfying the following initial conditions respectively
	
	$$X_N(t,z)=\begin{pmatrix}
	A(t,z) \\ B(t,z)
	\end{pmatrix} ,\quad
	X_N(0,z)=\begin{pmatrix}
	1 \\ 0
	\end{pmatrix} \quad \text{(Neumann boundary condition)}
	$$ 

	$$X_D(t,z)=\begin{pmatrix}
	C(t,z) \\ D(t,z)
	\end{pmatrix} ,\quad
	X_D(0,z)=\begin{pmatrix}
	0 \\ 1
	\end{pmatrix} \quad \text{(Dirichlet boundary condition)}
	$$

	From the solutions to Dirichlet and Neumann conditions we construct two functions 
	
	$$E(t,z)=A(t,z)-iC(t,z) \text{ and } \tilde{E}(t,z)=B(t,z)-iD(t,z).$$
	
	Closely related to these functions are the scattering functions corresponding to the Dirac system
	
	$$\mathcal{E}(t,z)=e^{itz}E(t,z) \text{ and } \tilde{\mathcal{E}}(t,z)=e^{itz}\tilde{E}(t,z).$$
	
	Furthermore, define for each $t>0$
	
	$$a(t,z)=\frac{\mathcal{E}(t,z)+ i \tilde{\mathcal{E}}(t,z)}{2}=\frac{e^{itz}}{2}(E(t,z)+i\tilde{E}(t,z)),$$
	
	$$b(t,z)=\frac{\mathcal{E}(t,z)-i\tilde{\mathcal{E}}(t,z)}{2}=\frac{e^{itz}}{2}(E(t,z)-i\tilde{E}(t,z)).$$
	
	We are interested in the quotient $b(t,z)/a(t,z).$ We define the non-linear Fourier transform of $f$ by the pointwise limit
	
	$$f^{\dagger}=\lim_{t\rightarrow \infty}\frac{b(t,z)}{a(t,z)},$$
	
	if it exists at $z.$ It is our main goal to show (see Theorem \ref{non-linear carleson}) that this limit exists pointwise almost everywhere. Analogous to Carleson's Theorem where we investigated the limiting behavior of the partial sums, we consider a partial non-linear Fourier transform. More precisely given $f\in L^2(\R_+)$ we look for any $T>0$ at the restriction of $f$ to the interval $(0,T)$ given by  $f_T:=f\chi_{(0,T)}.$ Then, non-linear Fourier transform is  
	
	$$f_{T}^{\dagger}(z)=\frac{b(T,z)}{a(T,z)}.$$
	
	It is a natural question to ask whether the classical results for the linear Fourier-transform hold as well in the non-linear case. For instance there is an analogue of Parseval's identity. For all $0<t<\infty$ the functions $a(t,\cdot)$ satisfy the non-linear Parseval identity
		
		\begin{equation}\label{non_linear_parseval}
		||\log|a(t,\cdot)| ||_{L^1(\R)}=||f||_{L^2(0,t)}^2.
		\end{equation}
	
	\noindent	
	The non-linear Parseval identity in terms of $f^{\dagger}$ takes the form
	
	$$\frac{1}{2} ||\log(1-|f^{\dagger})|||_{L^1(\R)}=||f||_{L^2(\R)}^2.$$

	\noindent
	This identity together with the fact that the linear term in the series expansion of $f^{\dagger}$ is the Fourier transform $\hat{f}$ are mainly responsible for calling $f^{\dagger}$ the non-linear Fourier transform of $f.$ For more information regarding the non-linear Fourier transform we refer to the lecture notes by T. Tao and C. Thiele \cite{Tao}. We want to prove the following non-linear version of Carleson's Theorem.
	
	\begin{theorem}[Non-linear Carleson theorem]\label{non-linear carleson}
		Let $f$ be in $L^2(\R_+).$ Then, for almost all $s\in \R$ 
		
		$$\lim_{T\rightarrow \infty} f_{T}^{\dagger}(s) = f^{\dagger}(s).$$  
	\end{theorem}

	\subsection{Idea of the proof for non-linear Carleson}
	
	The solutions $E(t,z)$ and $\tilde E(t,z)$ of \eqref{dirac_system} are Hermite-Biehler functions, i.e. for all $z \in \C_+$ they satisfy the inequality
	
	$$|E(t,z)|>|E(t,\overline{z})|.$$
	
	It turns out that to each Hermite-Biehler function $E(t,z)$ one can associate a Hilbert space of entire functions $B(E(t,z)).$ As sets, we have that the spaces $B(E(t,z))$ are the same as the Paley-Wiener spaces $\PW_t,$ but with different norms. One key property that they share with the Paley-Wiener spaces $\PW_t$ is that they are reproducing kernel Hilbert spaces, which means that for all $\lambda \in \C$ there exists a function $K(t,\lambda,\cdot) \in B(E(t,z))$ such that for for all $f\in B(E(t,z))$ we can recover $f(\lambda)$ by the formula
	
	$$f(\lambda)=\langle f, K(t, \lambda, \cdot) \rangle.$$
	
	The reproducing kernel for the Paley-Wiener space $\PW_t$ is by Fourier inversion the function
	
	$$\sinc(t,\lambda,z)=\frac{1}{\pi} \frac{\sin[t(z-\overline{\lambda})]}{z-\overline{\lambda}}$$
	
	\noindent
	and the reproducing kernel of the space $B(E(t,z))$ is given by the expression
	
	$$K(t,\lambda,z)=\frac{1}{2\pi i} \frac{E(z)E^{\sharp}(\overline{\lambda})-E^{\sharp}(\overline{z})E(\overline{\lambda})}{\overline{\lambda}-z}=\frac{1}{\pi} \frac{A(z)C(\overline{\lambda})-C(z)A(\overline{\lambda})}{\overline{\lambda}-z},$$
	
	\noindent
	where $E^{\sharp}(t,z)=\overline{E(t,\overline{z})}$ denotes the Schwarz transform. From these two formulas, one could guess that if we showed that for large times $t$ the two reproducing kernels are close, then $E(t,z)$ should be approximated by linear combinations of trigonometric functions. This is exactly what we are going to prove in sections $3.1$ and $3.2.$ Let $s\in \R$ and $C>0.$ We define for $t>0$ the box 
	
	$$Q(s,C/t):=\{z \in \C \mid |\Re(s-z)|\leq C/t, |\Im z|\leq C/t \}.$$
	
	We will show in Corollary \ref{corollary_1} that for almost all $s \in \R$ and all $C>0$ we have 
	
	$$\sup_{\lambda, z \in Q(s,C/t)} \bigg|[A(t,z)C(t,\overline{\lambda})-C(t,z)A(t,\overline{\lambda})]-\frac{1}{w(s)}\sin[t(\overline{\lambda}-z)]\bigg|=o(1)$$
	
	\noindent
	as $t\rightarrow \infty.$ Based on this equation we find that there are real functions $x_1(t),x_2(t),c_1(t),c_2(t)$ such that for all $z \in Q(s,C/t)$ we have uniformly
	
	$$E(t,z)=-\frac{1}{c_1(t)}\cos[t(z-x_1(t))]-i\frac{1}{c_2(t)}\sin[t(z-x_2(t))]+o(1).$$
	
	If we now had that $c_1(t)=c_2(t)+o(1)$ and $t(x_2(t)-x_1(t))=o(1)$ mod $2\pi,$ then by Euler's formula 
	
	$$\sup_{z \in Q(s,C/t)} |E(t,z)-D\alpha(t)e^{itz}|=o(1)$$
	
	\noindent
	for some complex constant $D$ and a complex function $\alpha(s,t)$ satisfying $|\alpha(s,t)|=1$ for all $t>0.$ If we had this approximation unconditionally we could derive the pointwise convergence of the non-linear Fourier transform easily, see section 3.8. 
	
	Unfortunately we only have $c_1(t)=c_2(t)+o(1)$ and $t(x_2(t)-x_1(t))=o(1)$ mod $2\pi$, when there is no zero of $E(t,z)$ inside the box $Q(s,C/t)$ for all large enough $t.$ If there is a zero $z(t)$ of $E(t,z)$ inside $Q(s,C/t)$ the approximation is pertubated to
	
	$$\sup_{z \in Q(s,C/t)} |E(t,z)-\alpha(t)\gamma(t,z(t))\sin[t(z-z(t))]|=o(1),$$
	
	\noindent
	compare to Lemma \ref{Lemma_6} and Corollary \ref{corollary_4} for precise statements. This is the first big milestone in the proof. We have shown that there is a good approximation to $E(t,z)$ when there is no zero in the box $Q(s,C/t)$ for all large enough $t$ and that the approximation to $E(t,z)$ is bad when there is a zero inside $Q(s,C/t)$ for some time $t.$ Thus, for the remainder of the proof we will show that the set 
	
	$$T_0(s,C)=\{t \in \R_+ \mid Q(s,C/t) \text{ contains a zero of } E(t,z) \}$$
	
	\noindent
	is bounded for almost all $s \in \R$ and any $C>0.$ By the previous step we can then immediately conclude Theorem \ref{non-linear carleson}. First, we will show that $T_0(s,C)$ can not cover a half-line, i.e. the zeros of $E(t,z)$ can not stay in the box $Q(s,C/t)$ forever. This is shown in section $3.3.$ We will employ the fact the $|E(t,x)|$ converges in measure on $\R$ by Lemma \ref{measure}. If we now had that the zeros of $E(t,z)$ remained in the box forever, then $|E(t,x)|\approx |\sin tx|$ for points $x$ near $s.$ Since $|\sin tx|$ does not converge in measure, neither should $|E(t,x)|,$ but this contradicts Lemma \ref{measure}. With the same argument, we could prove that zeros can not travel infinitely often through the box $Q(s,C/t)$ for any fixed time scale. However, there seems to be no minimal time scale, depending on $f,$ in which the zeros travel through the box. Thus, with this approach we can not immediately conclude Theorem \ref{non-linear carleson}. The next natural step is to study further the dynamics of the zeros. This is done in section $3.4.$ From the study of the dynamics we get stronger approximations for $E(t,z).$ Since the proof of the stronger approximation (see Theorem \ref{Lemma_10}) is quite technical and complicated we postpone it to section $3.7.$
	
	To show that $T_0(s,C)$ is bounded, we consider the boxes $Q(s,3C/t)$ and $Q(s,C/t).$ Whenever a zero goes inside $Q(s,C/t)$ it has to travel through the larger box $Q(s,3C/t).$ We have just seen that the zeros of $E(t,z)$ can not stay inside $Q(s,C/t)$ forever. Thus, they would have to travel through $Q(s,3C/t)$ infinitely often if $T_0(s,C)$ was unbounded. We will show that for each time the zero travels through the box $Q(s,3C/t),$ this costs a chunk of $||f||_{L^2(\R)}.$ Since we assumed $f\in L^2(\R)$ we immediately obtain a contradiction if $T_0(s,C)$ was unbounded.

	To prove such an estimate it is natural to relate the scattering data from $E(t,z)$ to $f.$ In section 3.7 we will define scattering functions $a_{t_1\rightarrow t_2}$ for any interval, where there is a zero of $E(t,z)$ inside the box $Q(s,3C/t).$ Together with \eqref{non_linear_parseval} we obtain the non-linear Parseval identity 
	
	$$||\log|a_{t_1\rightarrow t_2}(s)|||_{L^1(\R)}=||f||_{L^2((t_1,t_2))}.$$
	
	When trying to estimate the quantity on the left hand side by its series expansion we get a bound of the type 
	
	$$||\log |a_{t_1\rightarrow t_2} (s)|||_{L^1(\R)}\gg ||\arg a_{t_1\rightarrow t_2}||_{L^{1,\infty}(\R)} \geq D|\varepsilon_1| + O(\varepsilon_1^2+\varepsilon_2^2),$$
	
	\noindent
	where $\varepsilon_1$ measures the horizontal increment in the movement of the zero in the time interval $(t_1,t_2)$ and $\varepsilon_2$ measures the vertical component and $D>0$ is some constant depending only on $s$ and the size of the box $C$. Since this estimate is linear in $\varepsilon_1,$ but quadratic in $\varepsilon_2$ we will have to consider intervals where the movement of the zero has a horizontal component comparable to the vertical one and intervals where the zero moves almost verticaly separately. These two cases will be covered in sections 3.6 and  3.7 respectively. Both arguments use a version of Parseval's identity. For the horizontal intervals we will use the non-linear Parseval identity stated above and for vertical intervals, we are going to use the linear Parseval identity. After both cases have been covered Theorem \ref{non-linear carleson} finally follows.
	
	\section{Preliminaries}
	In this section we are going to introduce all the relevant concepts neeeded for the proof of Theorem \ref{non-linear carleson}. We start with a short discussion on Hardy spaces, which provide tools for the systematic study of bounded analytic functions in the upper half-plane. We restrict ourselves to refering for all relevant proofs to the literature and prove only the more special applications of the theory here. We will most often refer to the book by J.Garnett \cite{Garnett}. The reason for the rather short discussion on this interesting subject is, that they are merely used as a tool in proving the approximations of the reproducing kernels in section $3.1.$ After section $3.1,$ Hardy spaces are never used again and the remainder of the proof is elementary but very technical complex analysis. 
	
	The more delicate input are the de Branges spaces, which are discussed in section $3.2,$ which are reproducing kernel Hilbert spaces of entire functions we can associate to the Hermite-Biehler functions $E(t,z).$ We will benefit a lot from the existence of a  Poisson finite spectral measure $\mu=w(s)ds+d\mu_s$ which provides for all $t\geq0$ an isometric embedding 
	
	$$B(E(t,z)) \hookrightarrow L^2(\mu).$$
	
	The key ingredient which tells us that $w(s)$ is reasonably well behaved is the fact that 
	
	$$\frac{\log w(t)}{1+t^2} \in L^1(\R).$$
	
	\noindent
	This fact is one of the key corner stones of the proof and is so important, since on the one hand we see that $w(s)\neq 0$ for almost all $s \in \R$ and on the other hand it enables us to use the theory of Hardy spaces, see the discussion on outer functions below or in \cite{Garnett}.
	
	Finally, we are going to establish basic properties of the functions $E(t,z)$ that we will need frequently throughout the proof. The most useful identity which will be used throughout later when approximating $E(t,z)$ is Corollary \ref{determinant_2} which states that
	
	$$\det\begin{pmatrix}
	E(t,z) && \tilde{E}(t,z) \\
	E^{\sharp}(t,z) && \tilde{E}^{\sharp}(t,z)
	\end{pmatrix}=2i.$$
	
	\subsection{Bounded analytic functions}
	The upper half-plane will be denoted by
	$$\C_+:=\{z=x+iy \in \C \mid \Im(z)=y>0\}.$$
	
	We define the Hardy space on $\C_+$ as a certain subspace of the analytic functions on the upper half-plane as follows. Let $1<p<\infty$ and set 
	
	$$\Hardy^p(\C_+):=\bigg\{f: \C_+ \rightarrow \C \mid f \text{ is analytic and } \sup_{y>0} \int_{-\infty}^{\infty} |f(x+iy)|^p dx < \infty\bigg\}.$$
	
	\noindent
	We turn $\Hardy^p(\C_+)$ into a normed space by defining the $\Hardy^p$ norm as
	
	$$||f||_{\Hardy^p}:=\sup_{y>0} \bigg( \int_{-\infty}^{\infty} |f(x+iy)|^p dx \bigg)^{\frac{1}{p}}.$$

	\begin{lemma}\label{hardy_space_2}
		Let $1<p<\infty$ and let $f$ be a function in $\Hardy^p(\C_+).$ Then, for almost all $x\in \R$ the non-tangential limit
		
		$$f(x):=\lim_{y\searrow 0} f(x+iy)$$
		
		\noindent
		exists and it satisfies the following properties
		
		$$||f||_p^p=||f||_{\Hardy^p}^p,$$
		
		$$\lim_{y\searrow 0} ||f(x+iy)-f(x)||_p = 0.$$
	\end{lemma}

	\begin{proof}
		\cite{Garnett} p.55 
	\end{proof}

	\begin{lemma}\label{hardy_space_3}
		Let $1<p<\infty$ and let $f$ be a function in $\Hardy^p(\C_+).$ Then, for all $z \in \C_+$
		
		$$|f(x+iy)|\leq \bigg(\frac{2}{\pi y}\bigg)^{\frac{1}{p}} ||f||_{\Hardy^p}.$$
	\end{lemma}

	\begin{proof}
		\cite{Garnett} pp. 17-18.
	\end{proof}

	For any $z=x+iy \in \C_+$ we define the Poisson kernel in the upper half-plane by
	
	$$P_z(t)=\frac{1}{\pi} \frac{y}{(x-t)^2+y^2}.$$ 
	
	Now let $d\mu(t)$ be a positive measure such that 
	
	$$\int_{-\infty}^{\infty} \frac{1}{1+t^2}d\mu(t) < \infty.$$
	
	\noindent
	We will call such measures Poisson-finite and define the Poissson extension to the upper half-plane by
	
	$$P\mu(x+iy):=\frac{1}{\pi}\int_{-\infty}^{\infty} \frac{y}{(x-t)^2+y^2} d\mu(t).$$
	
	\begin{lemma}\label{harmonic_extension}
		Let $d\mu(t)$ be a positive Poisson-finite measure and let $P\mu(x+iy)$ be its Poisson extension to the upper half plane. Then, $P\mu(x+iy)$ is a harmonic function in $\C_+.$
	\end{lemma}

	\begin{proof}
		\cite{Garnett} pp.11-12.
	\end{proof}
	
	A family of functions $\{\varphi_y(t)\}_{y>0}$ which are integrable on $\R$ is called an approximate identity, if the following conditions are fulfilled
	
	\begin{itemize}
		\item [(i)] $\int_{-\infty}^{\infty} \phi_y(t) dt =1,$
		\item [(ii)] $||\phi_y||_1 \leq M $ for some positive constant $M>0.$
		\item [(iii)] For all $\delta>0$ we have $\lim_{y\searrow 0} \sup_{|t|>\delta} |\phi_y(t)|=0,$
		\item [(iv)] For all $\delta>0$ we have $\lim_{y \searrow 0} \int_{|t|>\delta} |\phi_y(t)| dt=0.$
	\end{itemize}
	
	\noindent
	The Poisson kernel $P_z(t)$ is an example of an approximate identity.
	
	\begin{lemma}\label{approximate_identity}
		Let $f(x) \in L^p(\R)$ for any $1\leq p \leq \infty$ and assume that $f$ is continous at some point $x_0 \in \R.$ Furthermore, let $\{\varphi\}_{y>0}$ an approximate identity and set $u(z)=P_y * f(x)$ for $z=x+iy$ in $\C_+.$ Then,
		
		$$\lim_{(x,y) \rightarrow x_0} u(x+iy) = f(x_0).$$
	\end{lemma}

	\begin{proof}
		\cite{Garnett} p.16 together with the remark on p.19.
	\end{proof}

	For $x\in \R$ define the cone 
	
	$$\Gamma(x):=\{z \in \C_+ \mid |\Re(z-x)| < \Im z\}.$$
	
	\noindent
	For any function $u:\C_+\rightarrow \C$ we introduce the nontangential maximal function of $u$ on $\R$ as
	
	$$Mu(x):=\sup_{z\in \Gamma(x)} |u(z)|.$$
	
	Furthermore we can associate the maximal function to any bounded cone, i.e. for $z \in \Gamma(x)$ with $\Im z<C$ for some $C>0.$ We will use the same notation $Mu(x).$ It will unambigious in the text, whether we use the bounded non-tangential maximal operator or the unbounded one. More precisely, Lemma \ref{Lemma_2} uses the bounded, Lemma \ref{Lemma_4} the unbounded operator.
	
	\begin{lemma}\label{hardy_space_5}
		Let $P\mu$ denote the Poisson extension of a Poisson finite measure on $\R.$ Then, $MP\mu$ is weak-$L^1,$ where $MP\mu$ is the maximal function with respect to a bounded cone.
	\end{lemma}

	\begin{proof}
		This follows immediately from the correspong result in the disk, see \cite{Garnett} pp.55-56, and the fact that Poisson-integrable functions on the real line transfer into integrable functions on the disk under the Cayley transform.
	\end{proof}
	
	We obtain immediately
	
	\begin{corollary}\label{hardy_space_6}
		Let $P\mu$ be the Poisson extension of a Poisson-finite measure on $\R.$ Then, $\sqrt{MP\mu} \in L_{loc}^1(\R).$
	\end{corollary}
	
	\begin{lemma}\label{hardy_space_8}
		Let $f\in \Hardy^p(\C_+)$ for $1< p < \infty.$ Then, the non-tangential maximal function $Mf(x)$ is in $L^p(\R)$ and there is an absolute constant $A_p>0$ such that 
		
		$$||Mf||_{L^p(\R)}\leq  A_p ||f||_{\Hardy^p(\R)}.$$
	\end{lemma}
	
	\begin{proof}
		\cite{Garnett} p.27
	\end{proof}

	\begin{lemma}\label{singular_measure}
		Let $\{\varphi_y(t)\}_{y>0}$ be an approximate identity and let $\nu$ be a finite singular measure on $\R.$ Then, $(\varphi_y*\nu)(x)$ converges non-tangentially to zero almost everywhere.
	\end{lemma}

	\begin{proof}
		\cite{Garnett} p.30
	\end{proof}
	
	Let $f\in L_{loc}^1(\R^n).$ We say that $x\in \R^n$ is a Lebesgue point of $f,$ if 
	
	$$\lim_{r\searrow 0} \frac{1}{|B_r(x)|} \int_{B_r(x)} f(y)dy=f(x).$$
	
	\noindent
	It is well known, that almost all $x\in \R^n$ are Lebesgue points.
	
	Let $f: \C_+\rightarrow \C$ be an analytic function on the upper half plane. We say that $f$ is an inner function on $\C_+,$ if the following two conditions are fulfilled
	
	\begin{itemize}
		\item [(i)] $f$ is bounded,
		\item [(ii)] the non-tangential limit exists almost everywhere and is unimodular, i.e. for almost all $x\in \R$ we have
		
		$$\lim_{y\searrow 0} |f(x+iy)|=1.$$
	\end{itemize}

	We say that $f$ is an outer function in $\C_+$ if it is of the form 
	
	$$f(z)=\alpha e^{u(z)+iv(z)},$$
	
	\noindent
	where $\alpha$ is a unimodular constant, i.e. $|\alpha|=1,$ $u(z)=(P_y*\log w)(x)$ for a non-negative real function $w(t)$ satisfying 
	
	$$\int_{-\infty}^{\infty} \frac{\log w(t)}{1+t^2}dt < \infty$$
	
	\noindent
	and $v(z)$ is the harmonic conjugate of $u(z).$ We say that $w(t)$ is the boundary function of $f.$
	
	\begin{lemma}{\label{outer_function_boundary_values}}
		Let $f:\C_+\rightarrow \C$ be an outer function withboundary function $w(t).$ Then, for almost all $x\in \R$ the non-tangential limit exists and 
		
		$$\lim_{ z \in \Gamma(x)} |f(z)|=w(x).$$
	\end{lemma}
	\begin{proof}
		\cite{Garnett} p.64
	\end{proof}

	\begin{lemma}\label{hardy_space}
		Let $f\in L^2(\R).$ Then, $f$ is the non-tangential limit of an $\Hardy^2(\C_+)$ function if $\hat{f}(s)=0$ for almost all $s<0.$
	\end{lemma}

	\begin{proof}
		\cite{Garnett} p.84
	\end{proof}

	\begin{lemma}\label{hardy_space_7}
		Let $g\in \Hardy^2(\C_+)$ and $g'\in \Hardy^2(\C_-).$ Suppose $H:\C \rightarrow \C$ is an entire function equal to $e^{-iz}g$ in $\C_+$ and equal to $e^{iz}g'$ in $\C_-,$ then $H\in \PW_1.$
	\end{lemma}

	\begin{proof}
		This is an immediate consequence of the previous Lemma, which asserts that
		
		 $$\supp \widehat{g} \subset [0,\infty), \quad \supp \widehat{g'} \subset (-\infty, 0].$$
		 
		 \noindent
		 However, multipliying by $e^{-iz}$ and $e^{iz}$ respectively, shifts the support of the Fourier transform such that 
		 
		 $$\supp \widehat{e^{-iz}g} \subset [-1,\infty), \quad \supp \widehat{e^{iz}g'} \subset (-\infty, 1].$$
		 
		 \noindent
		 By the assumption on $H$ we have that 
		 	$$\supp \widehat{H} \subset (-\infty,1] \cap [1,\infty) = [-1,1].$$ 
		 	
		 \noindent
		 Therefore, $H$ is an entire function whose Fourier transform is supported on $[-1,1]$ and thus $H \in \PW_1.$
	\end{proof}
	
	We say that a sequence of function $f_n:\C \rightarrow \C$ converges normally to $f,$ if $f_n$ converges uniformly on compact subsets of $\C.$
	
	\begin{lemma}\label{hardy_space_4}
		Let $g_n \in \Hardy^2(\C_+)$ be a sequence converging weakly to some function $g\in \Hardy(\C_+) .$ Then, $g_n \rightarrow g$ normally in $\C_+.$
	\end{lemma}

	\begin{proof}
		\cite{Garnett}
	\end{proof}

	\subsection{Hilbert spaces of entire functions}
	The results of this section can be found in \cite{DeBranges}. We say that an entire function $E(z)$ is a Hermite-Biehler function if it satisfies for all $z \in \C_+$ the inequality
	
	$$|E(z)|>|E(\overline{z})|.$$
	
	To each Hermite-Biehler function $E(z)$ we can associate a de Branges space
	
	$$B(E):=\{f: \C \rightarrow \C \mid f \text{ is entire},\quad f/E \in H^2(\C_+), \quad f^{\sharp}/E \in H^2(\C_+)\}.$$
	
	\noindent
	We can turn $B(E)$ into a Hilbert space with inner product
	
	$$\langle f,g \rangle_{B(E)}=\int_{-\infty}^{\infty} \frac{f(x)g(x)}{|E(x)|^2} dx.$$

	By Krein's Theorem all functions in $B(E)$ have exponential type at most that of $E.$ Furthermore, for each $\lambda \in \C$ the pointwise evaluation functional 
	
	$$F_\lambda : B(E) \rightarrow \C $$
		$$\quad \quad \quad f \mapsto f(\lambda) $$
		
	\noindent	
	is bounded. Hence, by the Riesz representation Theorem there is a function $K(\lambda, \cdot) \in B(E)$ such that for all $f\in B(E)$ we can recover its value at $\lambda$ by integrating against $K(\lambda, \cdot),$ i.e.
	
	$$f(\lambda)= \langle f, K(\lambda, \cdot) \rangle_{B(E)}.$$
	
	\noindent
	\begin{lemma}\label{misc_1}
	Let $K(\lambda, \cdot)$ be the reproducing kernel from above. Then,
	
	$$||K(\lambda, \cdot)||_{B(E)}=||K(\lambda,\cdot)/E||_{H^2(\C_+)}=\sqrt{K(\lambda,\lambda)}=\sup_{f \in B(E), ||f||_{B(E)}\leq 1} |f(\lambda)|.$$
	\end{lemma}
	
	\begin{proof}
		By Riesz representation Theorem there is an antilinear isometry
		
		$$ \varphi: B(E) \rightarrow B(E)^*$$
		$$ \quad \quad \quad f \mapsto \langle f, \cdot \rangle.$$
		
		Because $K(\lambda, \cdot)$ is a reproducing kernel we have $\varphi(K(\lambda,\cdot))=F_\lambda.$ Since $\varphi$ is an isometry $$||K(\lambda,\cdot)||_{B(E)}=||F_\lambda||_{B(E)^*}=\sup_{f \in B(E), ||f||_{B(E)}\leq 1} |f(\lambda)|.$$
		
		The remaining claims follow from the equation $K(\lambda,\lambda)=\langle K(\lambda,\cdot), K(\lambda,\cdot)\rangle_{B(E)}.$
	\end{proof}

	\subsection{Properties of the non-linear Fourier transform}
	Recall that our analysis of the non-linear Fourier transform is based on the solutions $E(t,z)$ and $\tilde{E}(t,z)$ of the real Dirac system \eqref{dirac_system} with Neumann and Dirichlet initial conditions respectively. Most of the results of this section can be stated with arbitrary solutions of the real Dirac system regardless of the initial condition. To avoid confusion we will denote a solution of \eqref{dirac_system} with arbitrary initial condition by $S(t,z).$
	
	\begin{lemma}
		For each fixed $t>0$ the function $S(t,z)$ is a Hermite-Biehler function, i.e. for all $z \in \C_+,$
		
		$$|S(t,z)|>|S(t,\overline{z})|.$$
	\end{lemma}

	\begin{proof}
		A proof can be found in \cite{Romanov}.
	\end{proof}

	Hence, we can consider the de-Branges spaces $B(E(t,z))$ and $B(\tilde{E}(t,z)).$ By Krein's Theorem the function $E(t,z)$ has exponential type $t.$ We thus, have an inclusion $B(E(t,z)) \subset \PW_t$ as sets. It is natural to ask whether the reverse inclusion holds as well, i.e. if $B(E(t,z))=\PW_t$ as sets. Indeed, this is true if the potential function of the Dirac system \eqref{dirac_system} satisfies $f\in L_{loc}^1(\R_+).$
	
	 To prove this result, we will effectively have to show that the norms on $B(E(t,z))$ and $\PW_t$ are equivalent. More, precisely the norm on $\PW_t$ is the ordinary $L^2$-norm, and $B(E(t,z))$ is equipped with the $L^2$-norm, weighted by the function $|E(t,z)|.$ Hence, it suffices to show that $|E(t,z)|$ is bounded from above and from below. This will follow from a couple observations on the solutions of the Dirac system \eqref{dirac_system}.
	 
	 First, let us more conveniently keep track of our two solutions $E(t,z)$ and $\tilde{E}(t,z)$ simultaneously by defining the matrix 
	 
	 \begin{equation}\label{matrix}
	 M(t,z)=\begin{pmatrix}
	 A(t,z) && B(t,z) \\ C(t,z) && D(t,z)
	 \end{pmatrix}.
	 \end{equation}
	 
	 \noindent
	 Hence, $M(t,z)$ satisfies the differential equation
	 
	 \begin{equation}\label{dirac_system_matrix}
	 \begin{pmatrix}
	 0 && 1 \\ -1 && 0
	 \end{pmatrix} \frac{d}{dt} M(t,z) = zM(t,z)-\begin{pmatrix}
	 0 && f(t) \\ f(t) && 0
	 \end{pmatrix}M(t,z)
	 \end{equation}
	 
	 \noindent
	 The following result shows, that $E(t,z)$ and $\tilde{E}(t,z)$ are related.
	 
	 \begin{lemma}\label{determinant_1}
	 	The matrix $M(t,z)$ satisfies $\det M(t,z)=1$ for all $t\geq 0$ and $z\in \C.$
	 \end{lemma}
 
	\begin{proof}
		We calculate 
		
		$$\frac{d}{dt} \det M(t,z)= \frac{d}{dt} A(t,z)D(t,z) + \frac{d}{dt}D(t,z)A(t,z)-\frac{d}{dt} B(t,z)C(t,z)-\frac{d}{dt}C(t,z)B(t,z)$$
		
		From \eqref{dirac_system_matrix} we find 
		
		$$\begin{pmatrix}
			\frac{d}{dt} C(t,z) && \frac{d}{dt} D(t,z) \\ -\frac{d}{dt} A(t,z) && -\frac{d}{dt} B(t,z)
		\end{pmatrix}
		= \begin{pmatrix}
		zA(t,z)-f(t)C(t,z) && zB(t,z)-f(t)D(t,z) \\ zC(t,z)-f(t)A(t,z) && zD(t,z)-f(t)B(t,z)
		\end{pmatrix}.
		$$
		
		\noindent
		Hence, we arrive at 
		
		$$\frac{d}{dt} \det M(t,z)= (-zC(t,z)+f(t)A(t,z))D(t,z)+(zB(t,z)-f(t)D(t,z))A(t,z)$$
		$$+(zD(t,z)-f(t)B(t,z))C(t,z)+(-zA(t,z)+f(t)C(t,z))B(t,z)=0.$$
		
		\noindent
		This implies, that $\det M(t,z)$ is constant. The initial conditions of $E(t,z)$ and $\tilde{E}(t,z)$ give
		
		$$\det M(0,z)=\det \begin{pmatrix}
		1 && 0 \\ 0 && 1
		\end{pmatrix} = 1.$$
	\end{proof}		

	This relation can be rewritten in terms of $E(t,z)$ and $\tilde{E}(t,z)$ and will play a central role in the proof of Theorem \ref{non-linear carleson}.
	
	\begin{corollary}\label{determinant_2}
	$$\det\begin{pmatrix}
		E(t,z) && \tilde{E}(t,z) \\
		E^{\sharp}(t,z) && \tilde{E}^{\sharp}(t,z)
	\end{pmatrix}=2i.$$
	\end{corollary}

	\begin{proof}
		This follows immediately from the equation
		
		$$ \begin{pmatrix}
		A(t,z) && B(t,z) \\ C(t,z) && D(t,z)
		\end{pmatrix} = \frac{1}{2}\begin{pmatrix}
		 1 && 1 \\ i && -i
		\end{pmatrix}
		\begin{pmatrix}
		E(t,z) && \tilde{E}(t,z) \\
		E^{\sharp}(t,z) && \tilde{E}^{\sharp}(t,z)
		\end{pmatrix}.$$
	\end{proof}

	In the following we will talk for simplicity only about $E(t,z).$ All results remain true as well for $\tilde{E}(t,z).$ The next step towards showing $B(E(t,z))=\PW_t$ as sets, is the following 
	
	\begin{lemma}\label{differential_equation}
		The Hermite-Biehler function $S(t,z)$ solves the differential equation
		
		$$\frac{d}{dt} S(t,z) =-izS(t,z)+f(t)S^{\sharp}(t,z).$$
	\end{lemma}

	\begin{proof}
		From \eqref{dirac_system_matrix} we get 
		
		$$\begin{pmatrix}
		\frac{d}{dt} A(t,z) \\ \frac{d}{dt} C(t,z) 
		\end{pmatrix}
		= \begin{pmatrix}
		-zC(t,z)+f(t)A(t,z) \\ zA(t,z)-f(t)C(t,z)
		\end{pmatrix}.
		$$
		
		We compute
		
		$$\frac{d}{dt} E(t,z)= \frac{d}{dt}A(t,z)-iC(t,z)= -zC(t,z)+f(t)A(t,z)-i(zA(t,z)-f(t)C(t,z))$$
		
		$$=-z(C(t,z)+iA(t,z))+f(t)(A(t,z)+iC(t,z))=-izE(t,z)+f(t)E^{\sharp}(t,z),$$
		
		\noindent
		where we used in the last step that $A^{\sharp}=A$ and $C^{\sharp}=C.$
	\end{proof}

	To any Hermite-Biehler function $S(t,z)$ we can associate a scattering function by
	
	$$\mathfrak{S}(t,z)=e^{itz}S(t,z).$$

	\begin{lemma}\label{differential_equation_2}
		The scattering function $\mathfrak{S}(t,z)$ satisfies the differential equation
		
		$$\frac{d}{dt} \mathfrak{S}(t,z)=iz\mathfrak{S}(t,z)+e^{itz}\frac{d}{dt} S(t,z)=f(t)e^{2itz}\mathfrak{S}^{\sharp}(t,z).$$
	\end{lemma}

	\begin{proof}
		A short calculation involving Lemma \ref{differential_equation} gives
		
		$$\frac{d}{dt} \mathfrak{S}(t,z)=ize^{itz}S(t,z)+e^{itz}\frac{d}{dt}S(t,z)=iz\mathfrak{S}(t,z)+e^{itz}\frac{d}{dt}S(t,z)$$
		$$=iz\mathfrak{S}(t,z)+e^{itz}\big(-izS(t,z)+f(t)S^{\sharp}(t,z)\big)=f(t)e^{2itz}\mathfrak{S}^{\sharp}(t,z).$$

	\end{proof}

	We denote by $\arg E$ the continuous branch of the argument of $E$ in the closed upper half-plane satisfying $\arg E(t,0)=0.$ For $\tilde{E}$ we choose the argument such that $\arg \tilde{E}(t,0)=-\frac{\pi}{2}.$ The last Lemma gives rise to a differential equation for $|E(t,x)|$ where $x\in \R.$ 
	
	\begin{lemma}\label{derivative_formula}
		For all $x \in \R$ the function $|S(t,x)|$ satisfies
		
		$$\frac{d}{dt} |S(t,x)|= f(t)|S(t,x)|\cos[2\arg S(t,x)].$$
	\end{lemma}

	\begin{proof}
		Recall that for all $z\in\C$ $A(t,\overline{z})=\overline{A(t,z)}$ and $C(t,\overline{z})=\overline{C(t,z)}.$ Hence, $A(t,x)$ and $C(t,x)$ are both real on the real line. Thus,
			
			\begin{equation}\label{derivative}
			\frac{d}{dt} |E(t,x)|=\frac{\Re(E(t,x)\overline{E'(t,x)})}{|E(t,x)|}.
			\end{equation}

		\noindent
		Note that this expression is well-defined by Lemma \ref{determinant_1}. By Lemma \ref{differential_equation} we find
		
		$$\Re[E(t,x)\overline{E'(t,x)}]=\Re[(A(t,x)-iC(t,x))(ix(A(t,x)+iC(t,x))+f(t)(A(t,x)-iC(t,x)))]$$
		
		$$=f(t)(A(t,x)^2-C(t,x)^2)=f(t)\Re[(A(t,x)-iC(t,x))^2]=f(t)\Re[E(t,x)^2]$$
		$$=f(t)|E(t,x)|^2\cos[2\arg E(t,x)].$$
		
		We conclude by plugging this back into \eqref{derivative}.
	\end{proof}

	Again, the same result holds for $\tilde{E}(t,z).$ We are finally ready to prove the main result of this chapter.
	
	\begin{lemma}\label{pw}
		The de Branges spaces $B(E(t,z))$ and $B(\tilde{E}(t,z))$ are equal to the Paley-Wiener space $\PW_t$ as sets
	\end{lemma}

	\begin{proof}
		We will show that the norms on $B(E(t,z))$ and $\PW_t$ are equivalent. Recall that the norm on the Paley-Wiener space $\PW_t$ is just the ordinary $L^2$-norm on the real line. Hence, it suffices to show that $E(t,z)$ is bounded from above and below.
		
		Let us first show that $E(t,z)$ is bounded from above. By Lemma \ref{determinant_1}, $|E(t,x)|\neq 0$ for all $x\in \R.$ From Lemma \ref{derivative_formula} we get
		
		$$\frac{\frac{d}{dt}{|E(t,x)|}}{|E(t,x)|}=f(t)\cos[2\arg E(t,x)].$$
		
		\noindent
		Integrating this equation and taking exponentials yields
		
		$$|E(t,x)|=|E(0,x)|\int_{0}^{t} f(s) \cos[2\arg E(s,x)] ds.$$
		
		\noindent
		Since $E(0,z)=1$ for all $z \in \C,$ the last equation simplifies to
		
		$$|E(t,x)|\leq \int_{0}^{t} |f(s)| ds,$$
		
		\noindent
		which concludes the proof of the upper bound. The same bound holds for $\tilde{E}(t,x).$	Let us now turn towards the lower bound. Lemma \ref{determinant_1} tells us that for all $z \in \C$
		
		$$\det \begin{pmatrix}
		 A(t,x) && B(t,x) \\ C(t,x) && D(t,x)
		\end{pmatrix}=1.$$
		
		If $E(t,x)=A(t,x)-iC(t,x)$ was not bounded away from zero, then $\tilde{E}(t,x)=B(t,x)-iD(t,x)$ had to be arbtitrarily large, contradicting the upper bound of $\tilde{E}(t,x).$ Switching the roles of $E(t,x)$ and $\tilde{E}(t,x)$ in the previous argument gives the lower bound for $\tilde{E}(t,x).$
	\end{proof}

	Since we have shown that $B(E(t,z))=\PW_t$ as sets, we see that if $t<t',$ then $B(E(t,z)) \subset B(E(t',z)).$ Hence one could guess that there is a measure $\mu$ on $\R$ such that we have an isometric embedding $B(E(t,z)) \hookrightarrow L^2(\mu)$ for all $t\geq0.$ This is indeed true. We will call $\mu$ the spectral measure of $E(t,z).$ Moreover by Lebesgue's decomposition Theorem we can decompose 
	
	$$\mu=w(s)ds + d\mu_s$$ 
	
	\noindent
	into an absolute continuous and singular part. Again the same construction for $\tilde{E}(t,z)$ gives rise to a spectral measure $\tilde{\mu}.$ The following non-trivial Theorem is crucial for the proof of our main result and used to be a difficult problem in its own right. Essentially this boils down to show that the one-dimensional Schrödinger operator with $L^2$-potential has absolutely continuous spectrum which covers the real line.
	
	\begin{theorem}\label{spectral_measure}
		Let $w(s)$ and $\tilde{w}(s)$ be the densities of the absolute continuous parts of the spectral measures $\mu$ and $\tilde{\mu}$ of $E(t,z)$ and $\tilde{E}(t,z),$ respectively. Then,
		
		$$\frac{\log w(t)}{1+t^2} \in L^1(\R), \quad \frac{\log \tilde{w}(t)}{1+t^2} \in L^1(\R).$$
	\end{theorem}

	\begin{proof}
		A proof can be found in \cite{Deift} and \cite{Denisov}.
	\end{proof}

	This result implies in particular that $w(s),\tilde{w}(s)\neq 0$ almost everywhere on $\R.$ As a corollary we obtain the following 
	
	\begin{lemma}\label{measure}
		$|E(t,x)|$ converges in measure to $\frac{1}{\sqrt{w(s)}}$ on $\R$
	\end{lemma}
	
	\begin{proof}
		Since we have the isometric embedding for all $t>0$
		
		$$\int_{\R} \frac{|f(x)|^2}{|E(t,x)|^2} dx = \int_{\R}  |f(x)|^2 d\mu,$$
		
		one sees that as $t\rightarrow \infty$
		
		$$\frac{1}{|E(t,s)|^2}\rightarrow d\mu(s)=w(s)$$ 
		
		in measure on $\R.$ Since $w(s)\neq 0$ almost everywhere, we obtain the desired result.
	\end{proof}
	
	Thus, combining Lemma \ref{misc_1}, Lemma \ref{pw} and the isometric embeddings $B(E(t,z)) \hookrightarrow L^2(\mu)$ and $B(\tilde{E}(t,z)) \hookrightarrow L^2(\mu)$ gives
	
	\begin{equation}\label{misc_2}
		||K(t,\lambda, z)||_{B(E(t,z)}=\sup_{f \in \PW_t, ||f||_{\mu}\leq 1} |f(\lambda)|, \quad ||\tilde{K}(t,\lambda, z)||_{B(\tilde{E}(t,z)}=\sup_{f \in \PW_t, ||f||_{\tilde{\mu}}\leq 1} |f(\lambda)|,
	\end{equation} 
	
	\noindent
	where $K(t,\lambda, \dot)$ and $\tilde{K}(t,\lambda,\dot)$ are the reproducing kernels of $B(E(t,z))$ and $B(\tilde{E}(t,z)).$
	
	Furthermore, we can even give an explicit expression for the reproducing kernels.
	
	\begin{lemma}\label{reproducing_kernel}
		For any $t>0$ and $\lambda, z \in \C$ we have
		
		$$K(t,\lambda,z)=\frac{1}{2\pi i} \frac{E(z)E^{\sharp}(\overline{\lambda})-E^{\sharp}(\overline{z})E(\overline{\lambda})}{\overline{\lambda}-z}=\frac{1}{\pi} \frac{A(z)C(\overline{\lambda})-C(z)A(\overline{\lambda})}{\overline{\lambda}-z},$$
		
		$$\tilde{K}(t,\lambda,z)=\frac{1}{2\pi i} \frac{\tilde{E}(z)\tilde{E}^{\sharp}(\overline{\lambda})-\tilde{E}^{\sharp}(\overline{z})\tilde{E}(\overline{\lambda})}{\overline{\lambda}-z}=\frac{1}{\pi} \frac{B(z)D(\overline{\lambda})-D(z)B(\overline{\lambda})}{\overline{\lambda}-z}.$$
	\end{lemma}
	
	\begin{proof}
		A proof can be found in \cite{DeBranges}.
	\end{proof}
	
	We will now solve the Dirac System \eqref{dirac_system_matrix} for $f=0.$ This is referred to as the free case. In this case the system of differential equations we have to solve simplifies to 
	
	$$\begin{pmatrix}
		0 && 1 \\ -1 && 0
	\end{pmatrix} \frac{d}{dt} M(t,z) = zM(t,z).$$
	
		One can check quite fast that the solutions are given by $E(t,z)=e^{-itz}=\cos(tz)-i\sin(tz)$ and $\tilde{E}(t,z)=-ie^{itz}=-\sin(tz)-i\cos(tz).$ Since $|E(t,z)|=|\tilde{E}(t,z)|=1$ for all $z \in \C$ we find that $B(E(t,z))=B(\tilde{E}(t,z))=\PW_t$ as normed vector spaces. Recall that in general we only have equality as sets.

	\section{Proof of non-linear Carleson}
	\subsection{Proximity of reproducing kernels}
	\noindent
	The first key insight used to get an approximation for $E(t,z)$ is that the reproducing kernel $K(t,\lambda,z)$ of the associated de Branges space $B(E(t,z))$ allows reconstructing the function $E(t,z)$ since by Lemma \ref{reproducing_kernel}
	
	$$K(t,\lambda,z)=\frac{1}{2\pi i} \frac{E(z)E^{\sharp}(\overline{\lambda})-E^{\sharp}(\overline{z})E(\overline{\lambda})}{\overline{\lambda}-z}=\frac{1}{\pi} \frac{A(z)C(\overline{\lambda})-C(z)A(\overline{\lambda})}{\overline{\lambda}-z}.$$
	
	Since the potential function $f$ of the Dirac-system \eqref{dirac_system_matrix} is in $L^2(\R)$ we have at least 
	
	$$\liminf_{t \rightarrow \infty} f(t) =0.$$
	
	For most well-behaved functions, we even have 
	
	$$\lim_{t \rightarrow \infty} f(t)=0.$$
	
	Carrying this information toward the Dirac system 
	
	$$\begin{pmatrix}
		0 && 1 \\ -1 && 0
	\end{pmatrix} \frac{d}{dt} M(t,z) = zM(t,z)-\begin{pmatrix}
		0 && f(t) \\ f(t) && 0
	\end{pmatrix}M(t,z),$$
	
	\noindent
	suggests that the last term should play a less significant role for large $t>0$ and hence gives rise to an approximation of the free system where $f=0.$ Following this logic, we would expect that the solution of the Dirac system will be approximated by the solution of the free system. This is precisely what we will show under the additional condition that there are no zeros of $E(t,z)$ too close to the real line for large times $t.$ We will show this estimate by proving first that the reproducing kernels of the associated de Brange spaces are close, since we can then easily pass to the functions $E(t,z)$ themselves as hinted above. 
	
	For $f=0$ and $t>0$ the associated de Brange space is the Paley-Wiener space $\PW_t$ (as explained in the previous section) and its reproducing kernel is the $\sinc$ function 
	
	$$\sinc(t,\lambda,z)=\frac{1}{\pi} \frac{\sin(t(z-\overline{\lambda}))}{z-\overline{\lambda}}.$$

	Hence, the big goal of this section will be to show that approximately 
	
	$$K(t,\lambda,z)=\frac{1}{\pi} \frac{A(z)C(\overline{\lambda})-C(z)A(\overline{\lambda})}{\overline{\lambda}-z} \approx \frac{1}{\pi} \frac{\sin(t(z-\overline{\lambda}))}{z-\overline{\lambda}}=\sinc(t,\lambda,z).$$
	
	The strategy will be to first show that the norms of the reproducing kernels are close and in the next step to pass from the norms to the functions themselves. During this procedure Theorem \ref{spectral_measure} shows its importance as it allows us to construct an outer function $G(s)$ with boundary function $\sqrt{w(s)}.$
	
	\begin{lemma}\label{Lemma_2}
		For almost all $s \in \R$ and all fixed $C>0$ we have 
		
		$$\sup_{z \in Q(s,C/t)}\bigg(\frac{w(s)||K(t,z,\cdot)||_\mu^2}{||\sinc(t,z,\cdot)||_2^2}-1\bigg)=o(1)$$
		
		as $t\rightarrow \infty.$
	\end{lemma}

	\begin{proof}
		Let us first prove the convergence pointwise. Choose for all $t>0$ a point $z_t \in Q(s,C/t).$ We will show that 
		
		$$\frac{w(s)||K(t,z_t,\cdot)||_\mu^2}{||\sinc(t,z_t,\cdot)||_2^2}-1=o(1).$$
		
		The first step will be to rewrite this equation in the more convenient form
		
		$$w(s)||K(t,z_t,\cdot)||_\mu^2=(1+o(1))||\sinc(t,z_t,\cdot)||_2^2.$$
		
		\noindent
		This equation is true if and only if and only if
		
		$$\sqrt{w(s)}||K(t,z_t,\cdot)||_\mu=(1+o(1))||\sinc(t,z_t,\cdot)||_2.$$
		
		\noindent
		The advantage of this formulation is that we can now plug in the relations 
		
		\begin{equation}\label{misc3}
		||K(t,z_t,\cdot)||_{\mu}= \sup_{f \in \PW_t, ||f||_\mu \leq 1} |f(z_t)|, \quad ||\sinc(t,z_t,\cdot)||_2=\sup_{f \in \PW_t, ||f||_2 \leq 1} |f(z_t)|,
		\end{equation}
		
		\noindent
		to arrive at the equation we are going to show
		
		\begin{equation}\label{lemma_2}
		\sqrt{w(s)}\sup_{f \in \PW_t, ||f||_\mu \leq 1} |f(z_t)|=(1+o(1))\sup_{f \in \PW_t, ||f||_2 \leq 1} |f(z_t)|.
		\end{equation}
		
		We will show both inequalities separately. Let us first show that the left-hand side is at most as large as the right-hand side. We define for all $t>0$ and $x\in \R$ the function
		
		$$s_t(x)= \frac{|\sinc(t,z_t,x)|^2}{||\sinc(t,z_t,x)||_2^2}.$$
		
		Note that $s_t(x)$ is an approximate identity at the point $s$ and it follows from Lemma \ref{approximate_identity} and Lemma \ref{singular_measure}  that for almost all $s\in \R$
		
		\begin{equation}\label{misc48}
		\int_{-\infty}^{\infty} s_t(x)d\mu =w(s) +o(1).
		\end{equation}
		
		\noindent
		Indeed, we can write
		$$\int_{-\infty}^{\infty} s_t(x) d\mu = \int_{-\infty}^{\infty} s_t(x) dw(x)+ \int_{-\infty}^{\infty} s_t(x) d\mu_s.$$
		
		\noindent
		The first integral converges to $w(s)$ by Lemma \ref{approximate_identity} since $s_t(x)$ is an approximate identity at $s.$ Since $\mu_s$ is Poisson-finite it follows by Dominated Convergence Theorem that the second integral over the complement of any finite interval is converging to zero. Restricting $\mu_s$ to the finite interval, gives a finite measure, so we can conclude by Lemma $\ref{singular_measure}$ that the second integral tends to zero.
		
		Recall that $\sinc(t,z_t,\cdot) \in \PW_t$ and it is thus rescaled by its $||\cdot||_\mu$ norm a valid test function for the operator norm of $K(t,z_t,\cdot).$ Moreover, since $\sinc(t,z_t,\cdot)$ is the reproducing kernel of $\PW_t,$ we have by Lemma \ref{misc_1} that $||\sinc(t,z_t,\cdot)||_2=\sqrt{\sinc(t,z_t,z_t)}.$ Hence,
		
		$$ ||K(t,z_t,\cdot)||_{\mu}=\sup_{f \in \PW_t, ||f||_\mu \leq 1} ||f(z)| \geq \frac{\sinc(t,z_t,z_t)}{||\sinc(t,z_t,\cdot)||_\mu}=\frac{\sinc(t,z_t,z_t)}{\sqrt{||\sinc(t,z_t,\cdot)||_\mu^2}}$$
		
		$$=\frac{\sinc(t,z_t,z_t)}{\sqrt{||\sinc(t,z_t,\cdot)||_2^2\int_{-\infty}^{\infty}s_t(x) d\mu}}=\frac{\sinc(t,z_t,z_t)}{\sqrt{||\sinc(t,z_t,\cdot)||_2^2(w(s)+o(1))}}=\frac{||\sinc(t,z_t,\cdot)||_2}{\sqrt{w(s)+o(1)}}$$
		
		Now, using \eqref{misc3} and rearranging the terms gives

		$$\sqrt{w(s)}\sup_{f \in \PW_t, ||f||_\mu \leq 1} ||f(z_t)|\leq(1+o(1))\sup_{f \in \PW_t, ||f||_2 \leq 1} ||f(z_t)|.$$
		
		Since $||f||_\mu \geq ||f||_w$ we can instead prove
		
		\begin{equation}\label{misc4}
		\sqrt{w(s)}\sup_{f \in \PW_t, ||f||_w \leq 1} ||f(z_t)|\leq(1+o(1))\sup_{f \in \PW_t, ||f||_2 \leq 1} ||f(z_t)|.
		\end{equation}
		
		 We will now show the reverse inequality, proving \eqref{lemma_2}. The first step is to rescale the problem in the variable $t,$ so that we can work on the fixed domain $Q(s,C).$ The hard part will be to control the function $w(s)$ during this process. It is precisely for this reason, that we will use the theory of Hardy spaces and Theorem \ref{spectral_measure}, which told us that 
		
		$$\frac{\log(w(t))}{1+t^2} \in L^1(\R).$$
		
		\noindent
		In particular, the last statement is still true if we replace $w(t)$ by $\sqrt{w(t)}.$ Hence, we can pick an outer function $G:\C_+\rightarrow \C$ with boundary function $\sqrt{w(t)}.$ By Lemma \ref{outer_function_boundary_values} $|G(x)|=\sqrt{w(x)}$ for almost all $x \in \R.$ We can assume without loss of generality that $|G(s)|=\sqrt{w(s)}.$ By changing the unimodular constant in the definition of $G,$ we can assume that $G(s)=\sqrt{w(s)}.$ Without loss of generality we can assume $w(s)=1,$ since $w(s)\neq 0$ for almost all $s\in \R.$
		
		Let us now turn towards rescaling the problem. For any $t>0$ the map 
		
		$$\varphi_t : Q(s,C) \rightarrow Q(s,C/t)$$
		$$ \quad \quad \quad \quad z \mapsto s+\frac{z-s}{t} $$
		
		\noindent
		is bijective. Hence, we can lift our chosen points $z_t$ on $Q(s,C/t)$ to $Q(s,C)$ by setting $\tilde{z_t}=\varphi^{-1}(z_t).$ By abuse of notation we will just write $z_t$ instead of $\tilde{z_t}$ for the rest of the proof.

		 We will now have to rescale the norms in \eqref{misc4}. For each $f \in \PW_t$ we set 
		 
		 $$\tilde{f}(z_t)=\frac{1}{\sqrt{t}}f\bigg(s+\frac{z_t-s}{t}\bigg). \text{ and } G_t(z)=G\bigg(s+\frac{z_t-s}{t}\bigg).$$
		 
		 \noindent
		 Then, $f\in \PW_1,$ because $\supp\hat{\tilde{f}} \subset \frac{1}{t}\supp\hat{f} \subset\frac{1}{t} [-t,t]=[-1,1].$ Now recall that on $\R$ we have $|G(t)|^2=w(t)$ and thus,
		 
		 $$||\tilde{f}||_w=\frac{1}{\sqrt{t}}\bigg(\int_{-\infty}^{\infty} \bigg|f\bigg(s+\frac{x-s}{t}\bigg)G(x)\bigg|^2 dx\bigg)^{\frac{1}{2}}=\frac{1}{\sqrt{t}}\bigg(\int_{-\infty}^{\infty}\sqrt{t}\bigg|f(x)G_t(x)\bigg|^2 dx\bigg)^{\frac{1}{2}}=||fG_t||_2.$$ 
		 
		 \noindent
		 Hence, we have the rescaled equation
		 
		 $$\sup_{f \in \PW_t, ||f||_w \leq 1} \bigg|f\bigg(s+\frac{z_t-s}{t}\bigg)\bigg|=\sqrt{t} \sup_{f \in \PW_1, ||fG_t||_2 \leq 1} |f(z_t)|.$$
		 
		 \noindent
		 and similarly, we prove
		 
		 $$\sup_{f \in \PW_t, ||f||_2 \leq 1} \bigg|f\bigg(s+\frac{z_t-s}{t}\bigg)\bigg|=\sqrt{t} \sup_{f \in \PW_1, ||f||_2 \leq 1} |f(z_t)|$$
		 
		 Thus, the inequality we will need to prove becomes
		 
		 $$\sup_{f \in \PW_1, ||fG_t||_2 \leq 1} ||f(z_t)|\leq(1+o(1))D_t,$$
		 
		 \noindent
		 where we put to simplify notation
		 
		 $$D_t=\sup_{f \in \PW_1, ||f||_2 \leq 1} |f(z_t)|.$$

		 We will prove this inequality by contradiction. Suppose that for some $\varepsilon>0$ there is a sequence of functions $f_n \in \PW_1$ such that for some subsequence $k_n \rightarrow \infty$ and we have 
		 
		 $$||f_nG_{k_n}||_2 \leq 1, \text{ and } |f_n(z_{k_n})| > D_{k_n} + \varepsilon.$$
		 
		\noindent 
		By construction all $z_t \in Q(s,C)$ and hence we can without loss of generality assume that $z_{k_n} \rightarrow z_0 \in Q(s,C),$ for the box is compact. We set 
		
		$$D=\sup_{f \in \PW_1, ||f||_2 \leq 1} |f(z_0)|.$$
		 
		\noindent 
		Notice that since $\sinc(t,\lambda,z) \in \PW_t$ is the reproducing kernel, we have by continuity
		 
		 $$D_t=\sqrt{\sinc(1,z_t,z_t)}\rightarrow \sqrt{\sinc(1,z_0,z_0)}=D.$$
		 
		Our strategy is to construct an admissible test function $H,$ i.e. $H\in \PW_1$ and $||H||_2\leq 1$ satisfying $|H(z_0)|\geq D+ \varepsilon.$ To construct the function $H,$ we will use the theory of Hardy spaces. By Lemma \ref{hardy_space} the function $e^{iz}f_n\in \Hardy^2(\C_+)$.	Indeed, we have that $\supp \hat{f_n} \subset [-1,1].$ Now $\supp \widehat{e^{iz}f_n}\subset 
		[0,\infty)$ and thus $e^{iz}f_n \in \Hardy^2(\C_+)$ by Lemma \ref{hardy_space}.

		 Since $G_{k_n}$ is outer and $||f_nG_{k_n}||_2\leq 1$ we find $g_n \in \Hardy^2(\C_+)$ Moreover, by Lemma \ref{hardy_space_2} $||g_n||_{H^2(\C_+)}=||g_n||_2\leq 1,$ since $||f_nG_{k_n}||_2 \leq 1.$
		 
		Hence, we can without loss of generality assume, passing to a subsequence if necessary, that $g_n\rightarrow g$ for some $g \in \Hardy^2(\C_+).$ By Lemma \ref{hardy_space_4} $g_n\rightarrow g$ normally in $\C_+.$ Moreover, since $g_n \rightarrow g$ weakly, we have by Lemma \ref{hardy_space_2}
		
		$$||g||_2=||g||_{H^2(\C_+)}\leq \liminf_{n \rightarrow \infty} ||g_n||_{H^2(\C_+)}\leq 1.$$

		Since by Lemma \ref{outer_function_boundary_values}
		
		$$\lim_{z \in \Gamma(x)} |G(z)|=1,$$
		 
		\noindent
		we find $G_t \overset{t\rightarrow \infty}{\longrightarrow} 1$ normally in $\C_+.$ We will conclude that $f_n$ converges normally to some function in $\C_+.$ Let $K \subset \C_+$ be compact and denote by $||\cdot||_{\infty,K}$ the supremum norm on $K.$ We will drop $K$ from the notation. Recall that by definition outer functions vanish nowhere in $\C_+.$
		
		$$||e^{iz}f_n-g||_{\infty}\leq\bigg|\bigg|e^{iz}f_n-\frac{g}{G_{k_n}}\bigg|\bigg|_{\infty}+\bigg|\bigg|\frac{g}{G_{k_n}}-g\bigg|\bigg|_{\infty}.$$
		
		\noindent
		However,
		
		$$\bigg|\bigg|e^{iz}f_n-\frac{g}{G_{k_n}}\bigg|\bigg|_{\infty} \leq ||G_{k_n}||_{\infty} ||e^{iz}f_nG_{k_n}-g||_{\infty} \rightarrow 0,$$
		
		$$\bigg|\bigg|\frac{g}{G_{k_n}}-g\bigg|\bigg|_{\infty} \leq ||g||_{\infty}\bigg|\bigg|\frac{1}{G_{k_n}}\bigg|\bigg|_{\infty}||G_{k_n}-1||_{\infty} \rightarrow 0.$$
		
		\noindent
		Hence, $e^{iz}f_n \rightarrow g$ normally in $\C_+.$ From this, we obtain quickly that $f_n \rightarrow e^{-iz}g$ normally in $\C_+.$ Indeed,
		
		$$||f_n-e^{-iz}g||_{\infty} \leq ||e^{iz}||_{\infty}||e^{iz}f_n-g||_{\infty}\rightarrow 0.$$
		
		By repeating the same argument in the lower half-plane $\C_-$ we obtain a function $g'$ such that $f\rightarrow e^{iz}g'$ normally in $\C_-.$ Since a normal limit  of holomorphic functions is holomorphic, we have a candidate function, which is analytic in the upper and lower half-planes. However, we still do not know yet what happens around the real line. To remedy this, we will try to find an integrable function that dominates $f_n$ close to the real line and conclude with dominated convergence. By Lemma \ref{hardy_space_3} applied to $g_n=e^{iz}f_nG_{k_n}$ we have for all $z\in \C$ the bound 
		
		$$|f_nG_{k_n}(x+iy)|\leq \bigg(\frac{2}{\pi |y|}\bigg)^{\frac{1}{2}}||g_n||_{H^2(\C_+)}\leq \frac{1}{\sqrt{|y|}}.$$
		
		Hence, we now have to work on bounding the outer function $G_{k_n}$ away from zero. From the definition of the outer function $G$ we see that
		
		\begin{equation}\label{misc5}
		|G(z)|=|e^{P\log\sqrt{w}(z)}|\geq e^{-MP\log\sqrt{w}(\Re z)},
		\end{equation}
		
		\noindent
		where $MP\log\sqrt{w}$ is the non-tangential maximal function of $P\log\sqrt{w}$ with respect to a large bounded cone $\Gamma(x).$ By Corollary \ref{hardy_space_6} the function $\sqrt{MP\log\sqrt{w}}$ is locally integrable. Hence, we can assume that $s$ is a Lebesgue point of $\sqrt{MP\log \sqrt{w}}.$ 
		
		Fix a constant $L>0.$ For large enough $n$ we can choose $c_n$ with $L<c_n<2L$ such that $P\log \sqrt{w}$ is uniformly bounded on the union of the lines $\Re(z-s)=\pm \frac{c_n}{k_n}.$ Indeed, suppose that for all $ x \in B_{\frac{2L}{k_n}}(s)\setminus B_{\frac{L}{k_n}}(s)$ we had $\sqrt{MP\log\sqrt{w(x)}}>3\sqrt{MP\log\sqrt{w(s)}}.$ Then, 
		
		$$\int_{B_{\frac{2L}{k_n}}(s)\setminus B_{\frac{L}{k_n}}(s)} \sqrt{MP\log\sqrt{w(t)}} dt > 3 \sqrt{MP\log \sqrt{w(s)}} \big|B_{\frac{L}{k_n}}(s)\big|,$$
		
		\noindent
		But this is a contradiction to the fact that $s$ is a Lebesgue point of $MP\log\sqrt{w},$ since
		
		$$ \frac{1}{\big|B_{\frac{L}{k_n}}(s)\big|} \int_{B_{\frac{2L}{k_n}}(s)\setminus B_{\frac{L}{k_n}}(s)} \sqrt{MP\log\sqrt{w(t)}} dt
			\leq 2\frac{1}{\big|B_{\frac{2L}{k_n}}(s)\big|} \int_{B_{\frac{2L}{k_n}}(s)} \sqrt{MP\log\sqrt{w(t)}} dt \rightarrow 2f(s).$$
		
		\noindent
		Hence, for each large enough $n,$ without loss of generality for all $n\geq 1,$ there has to be a point $c_n \in 	B_{\frac{2L}{k_n}}(s)\setminus B_{\frac{L}{k_n}}(s)$ such that $\sqrt{MP\log\sqrt{w(\frac{c_n}{k_n})}}\leq 3\sqrt{MP\log\sqrt{w(s)}},$ which shows as claimed, that 
		$P\log \sqrt{w}$ is uniformly bounded on the union of the lines $\Re(z-s)=\pm \frac{c_n}{k_n}.$
		
		Now consider for each $n\geq 1$ a sequence of squares $R_n,$ centered at $s$ with side length $2c_n.$ Hence, on the vertical sides of $R_n$ we have by construction and \eqref{misc5}. 
		
		$$|G_{k_n}(s\pm c_n +iy)|=\bigg|G\bigg(s+\frac{s\pm cn +iy -s}{k_n}\bigg)\bigg|\geq e^{MP \log \sqrt{w}(s\pm \frac{c_n}{k_n})}=\delta>0.$$
		
		\noindent
		Note that since $\frac{c_n}{k_n}\leq 2$ it is really sufficient to consider only a bounded sector $\Gamma(x)$ to estimate $MP\log \sqrt w$ on the square $R_n.$
		
		Thus, on the vertical sides of $R_n$ we have
		
		$$|f_n(s\pm c_n +iy)| \leq \frac{1}{\delta \sqrt{|y|}}.$$ 
		
		We can now finally conclude, that $f_n$ converges to some entire function in all of $\C.$ Denote by $R$ the space inscribed by the square centered as $s$ with side length $L.$ Then, by using the Cauchy integral formula with the paths traced out by $R_n,$ and $R,$ together with the dominated convergence Theorem we find that for all $a\in R$ that 
		
		$$ f_n(a)= \frac{1}{2\pi} \int_{R} \frac{f_n(z)}{z-a} dz  \rightarrow \frac{1}{2\pi} \int_{R} \frac{f(z)}{z-a} dz=:H(z).$$
		
		The function $H$ is thus entire and by earlier calculations equal to $e^{-iz}g$ in $\C_+$ and equal to $e^{iz}g'$ in $\C_-.$ Since $g\in \Hardy^2(\C_+)$ and $g'\in \Hardy^2(\C_-)$ we have by Lemma \ref{hardy_space_7} that $H\in \PW_1.$ Furthermore, by construction $||H||_2=||g||_2\leq 1.$ Thus, $H$ is an admissible test function. Finally, because $f_n$ converges normally to $H$ we have 
		
		$$|H(z_0)|=\lim_{n \rightarrow \infty}|f_n(z_{k_n})|\leq \lim_{n \rightarrow \infty} D_{k_n}+\varepsilon=D+\varepsilon,$$
		
		\noindent
		achieving our desired contradiction. Hence, we proved for our chosen points $z_t\in Q(s,C/t)$ that
		
		$$\sqrt{w(s)}\sup_{f \in \PW_t, ||f||_\mu \leq 1} ||f(z)|\geq(1+o(1))\sup_{f \in \PW_t, ||f||_2 \leq 1} ||f(z)|.$$
		
		\noindent
		and thus we have all together shown that
		
		$$\frac{w(s)||K(t,z_t,\cdot)||_\mu^2}{||\sinc(t,z_t,\cdot)||_2^2}-1=o(1).$$
		
		To upgrade this convergence to 
		
		$$\sup_{z \in Q(s,C/t)}\bigg(\frac{w(s)||K(t,z,\cdot)||_\mu^2}{||\sinc(t,z,\cdot)||_2^2}-1\bigg)=o(1)$$
		
		\noindent
		we first note, that since $||K(t,z,\cdot)||_\mu^2=K(t,z,z)\in \PW_t$ and $||\sinc(t,z,\cdot)||_2^2=\sinc(t,z,z)\in \PW_t$ both expressions are continuous in $z.$ Hence, for each $t>0$ we choose by compactness of $Q(s,C/t)$ the point $z_t^*$ which maximizes
		
		$$\bigg|\frac{w(s)||K(t,z_t^*,\cdot)||_\mu^2}{||\sinc(t,z_t^*,\cdot)||_2^2}-1\bigg|.$$
		
		\noindent
		Since we have proven the convergence for all choices $z_t$ and $z_t^*$ is a legitimate choice, we achieve an error term $o(1)$ uniform over all choices $z_t,$ which finishes the proof.
	\end{proof}
	
	\begin{corollary}\label{Lemma_3}
		For almost all $s \in \R$ and for all $C>0$
		
		$$\sup_{z\in Q(s,C/t)}\bigg|\bigg| K(t,z,\cdot)-\frac{1}{w(s)} \sinc(t,z,\cdot)\bigg|\bigg|_\mu^2=o(t),$$
		
		\noindent
		as $t \rightarrow \infty.$ 
	\end{corollary}

  Note that since $w(s)\neq 0$ for almost all $s \in \R,$ this expression is well-defined almost everywhere.
	
	\begin{proof}
		Let $z_t \in Q(s,C/t)$ for all $t>0$ and let $s_t$ be the function defined in the proof of Lemma \ref{Lemma_2}. Recall that 
		
		$$\int_{-\infty}^{\infty} s_t(x) d\mu = w(s) +o(1).$$
		
		Using that $L^2(\mu)$ is a Hilbert space with inner product $\langle \cdot, \cdot \rangle_\mu$ and the fact that $K(t,z_t,\cdot)$ is a reproducing kernel of the space $\PW_t$ equipped with the $||\cdot||_\mu$ norm, we find
		
		$$\bigg|\bigg| K(t,z_t,\cdot)-\frac{1}{w(s)} \sinc(t,z_t,\cdot)\bigg|\bigg|_\mu^2=\bigg\langle K(t,z_t,\cdot)-\frac{1}{w(s)} \sinc(t,z_t,\cdot),K(t,z_t,\cdot)-\frac{1}{w(s)}\sinc(t,z_t,\cdot)\bigg\rangle_\mu$$
		
		$$=||K(t,z_t,\cdot)||_\mu^2-\frac{2}{w(s)} \langle \sinc(t,z_t,\cdot), K(t,z_t,\cdot) \rangle_\mu+\frac{1}{w^2(s)}||\sinc(t,z_t,\cdot)||_\mu^2$$
		
		$$=||K(t,z_t,\cdot)||_\mu^2-\frac{2}{w(s)} \sinc(t,z_t,z_t) +\frac{||\sinc(t,z_t,\cdot)||_2^2}{w^2(s)} \int_{-\infty}^{\infty} s_t(x) d\mu$$
		
		$$=||K(t,z_t,\cdot)||_\mu^2-\frac{2}{w(s)} ||\sinc(t,z_t,\cdot)||_2^2 +\frac{||\sinc(t,z_t,\cdot)||_2^2}{w^2(s)} (w(s)+o(1))$$
		
		$$=||K(t,z_t,\cdot)||_\mu^2-\frac{1}{w(s)} ||\sinc(t,z_t,\cdot)||_2^2 +o(||\sinc(t,z_t,\cdot)||_2^2)=o(||\sinc(t,z_t,\cdot)||_2^2),$$
		
		\noindent
		where we used \ref{Lemma_2} in the last step. Now the proof is finished by noting that 
		
		$$\sinc(t,z_t,z_t)=||\sinc(t,z_t,z_t)||_2^2 \asymp t.$$
	\end{proof}

	We are now ready to prove our main result of this section.
	
	\begin{lemma}\label{Lemma_4}
		For almost all $s \in \R$ and for all $C>0$
		
		$$\sup_{\lambda, z \in Q(s,C/t)} \bigg|K(t,\lambda,z)-\frac{1}{w(s)} \sinc(t,\lambda,z)\bigg|=o(t),$$
		
		\noindent
		as $t\rightarrow \infty.$
	\end{lemma}

	\begin{proof}
		Let $z_t \in Q(s,C/t)$ for all $t>0$ and set 
		
		$$\Delta(t,z)=K(t,z_t,z)-\frac{1}{w(s)}\sinc(t,z_t,z).$$
		
		By an argument similar to the one in the proof of Lemma \ref{Lemma_2} we find that the function $F(t,z)=e^{itz}G(z)\Delta(t,z)\in \Hardy^2(\C_+),$ where 
		$G$ is again an outer function with boundary function $\sqrt{w}.$ Since neither $G,$ nor $e^{itz}$ has zeros in $\C_+,$ we can write
		
		$$\Delta(t,z)=F(t,z)e^{itz}G(z).$$
		
		We will make use of Lemma \ref{Lemma_3} to get an upper bound on $F,$ while we can treat the other terms with almost elementary methods. Let us first estimate $F.$ We calculate
		
		$$||F(t,\cdot)||_{H^2(\C_+)}^2=||\Delta(t,x)G(x)||_2^2=\int_{-\infty}^{\infty} |\Delta(t,x)G(x)|^2 dx$$
		
		$$=\int_{-\infty}^{\infty} |\Delta(t,x)|^2w(x) dx \leq \int_{-\infty}^{\infty} |\Delta(t,x)|^2 d\mu(x)=o(t),$$
		
		\noindent
		where we used Lemma \ref{Lemma_3} in the last step. By Lemma \ref{hardy_space_8}
		
		$$ ||MF(t,\cdot)||_2^2=o(t).$$
		
		From this estimate of the norm of the maximal function, we will derive an estimate for $MF.$ We set $I_t=Q(s,C/t)\cap \R$ and we define for any positive scalar $\lambda >0$ the set $\lambda I_t$ as the interval, where each point in $I_t$ is rescaled by $\lambda.$ Consider the set $5I_t \setminus 3I_t$ which is a union of two intervals $J_t^1$ and $J_t^2.$ We will prove that on two-thirds of each of $J_t^1$ and $J_t^2$ we have 
		
		\begin{equation}\label{misc6}
		(MF(x))^2\leq \frac{3||MF||_2^2 t}{C}=o(t^2).
		\end{equation}
	
		\noindent
		Indeed, let $A$ be the subset of $J_t^1$ such that \eqref{misc6} fails. Suppose that $|A|>\frac{1}{3}\frac{2C}{t}.$ Then,
		
		$$\int_{J_t^1} (MF(x))^2 dx \geq \int_{A} (MF(x))^2 dx  \geq \int_{A} \frac{3||MF||_2^2 t}{C} dx = |A|\frac{3||MF||_2^2 t}{C} > 2||MF||_2^2,$$
		
		\noindent
		which gives a contradiction. The proof for the interval $J_t^2$ is completely analogous. We denote by $S_t^1$ and $S_t^2$ the subsets of $J_t^1$ and $J_t^2,$ respectively, where \eqref{misc6} holds.
		
		The next step will be to estimate the outer function $G(z).$ We will do it similarly to the way we did it in the proof of Lemma \ref{Lemma_2}. We recall that the function $R=MP\log \sqrt{w}$ is locally integrable by Corollary \ref{hardy_space_6} and assume that $s$ is a Lebesgue point for this function. Then a calculation like in the proof of Lemma \ref{Lemma_2} shows that there are points $x_1 \in S_t^1$ and $x_2 \in S_t^2$ such that $R(x_1),R(x_2)<R(s).$ Thus, we have for $k=1,2$
		
		\begin{equation}\label{misc7}
		\sup_{z \in \Gamma(x_k), \Im z \leq \frac{5C}{t}} |\Delta(t,z)|=\sup_{z \in \Gamma(x_k), \Im z \leq \frac{5C}{t}} |e^{itz}G(z)F(t,z)|\leq e^{5C}e^{2R(s)}MF(x_k)=o(t),
		\end{equation}
				
		\noindent
		where the error term $o(t)$ is uniform over all choices $z_t \in Q(s,C/t).$ The last inequality holds in particular on the boundary of the rhombus
		
		$$ \C \setminus \bigcup_{x \notin (x_1,x_2)} \big(\Gamma(x)\cup \overline{\Gamma(x)}\big)$$
		
		\noindent
		which lies in the upper half-plane. Repeating the same construction with an outer function $\tilde{G}$ on the lower half-plane, gives the same estimate \eqref{misc7} in $\C_-.$ Hence, the inequality \eqref{misc7} holds on the complete boundary of the rhombus, after possibly adjusting $x_1$ and $x_2,$ and hence it holds by the maximum principle in all of the interior. Our result follows since the rhombus contains the box $Q(s,C/t).$
	\end{proof}
	\subsection{Approximation of the Hermite-Bieher functions}

	In this section, we are going to use the results from last section to get an approximation for the Hermite-Biehler functions $E(t,z)$ and $\tilde{E}(t,s).$ First of all we notice that the functions 
	
	$$K(t,\lambda,z)=\frac{1}{\pi} \frac{A(z)C(\overline{\lambda}-C(z)A(\overline{\lambda})}{\overline{\lambda}-z}, \quad \sinc(t,\lambda,z) = \frac{1}{\pi} \frac{\sin(t(\overline{\lambda}-z))}{\overline{\lambda}-z}$$
	
	\noindent
	have the same demoninators. Together with Lemma \ref{Lemma_4} we find that the numerators are close. More precisely,

	\begin{corollary}\label{corollary_1}
		For almost all $s \in \R$ and all $C>0$ we have 
		
		$$\sup_{\lambda, z \in Q(s,C/t)} \bigg|[A(t,z)C(t,\overline{\lambda})-C(t,z)A(t,\overline{\lambda})]-\frac{1}{w(s)}\sin(t(\overline{\lambda}-z))\bigg|=o(1).$$
	\end{corollary}

	\begin{proof}
		By Lemma \ref{Lemma_4}
		$$\sup_{\lambda, z \in Q(s,C/t)} \bigg|[A(t,z)C(t,\overline{\lambda})-C(t,z)A(t,\overline{\lambda})]-\frac{1}{w(s)}\sin(t(\overline{\lambda}-z))\bigg|$$
		$$=\sup_{\lambda, z \in Q(s,C/t)} |\pi (\overline{\lambda}-z)|\bigg|K(t,\lambda,z)-\frac{1}{w(s)} \sinc(t,\lambda,z)\bigg|$$
		
		$$=\frac{\pi\sqrt{2}C}{t} \sup_{\lambda, z \in Q(s,C/t)} \bigg|K(t,\lambda,z)-\frac{1}{w(s)} \sinc(t,\lambda,z)\bigg|=\frac{\pi\sqrt{2}C}{t}o(t)=o(1).$$
	\end{proof}

     Since the last statement holds for all $C>0,$ we can formulate it equivalently by considering a increasing function $C(t)>0,$ $C(t) \rightarrow \infty$ as $t \rightarrow \infty,$ instead of a constant $C>0.$
	
	\begin{corollary}\label{corollary_2}
		 For almost all $s \in \R$ there exists a increasing function $C(t)>0,$ $C(t) \rightarrow \infty$ as $t \rightarrow \infty,$ such that
		 
		 $$\sup_{\lambda, z \in Q(s,C(t)/t)} \bigg|[A(t,z)C(t,\overline{\lambda})-C(t,z)A(t,\overline{\lambda})]-\frac{1}{w(s)}\sin[t(\overline{\lambda}-z)]\bigg|=o(1).$$
	\end{corollary}

	\begin{proof}
		Choose some $C>0.$ We consider two ascending chains of real numbers
		
		$$C_1=C<C_2=C+1<C_3=C+2<\dots,$$
		
		$$t_1<t_2<t_3<\dots$$
		
		\noindent
		such that for all $i\geq1$ we have that for all $t>t_i$ 
		
		$$\sup_{\lambda, z \in Q(s,C_i/t)} \bigg|[A(t,z)C(t,\overline{\lambda})-C(t,z)A(t,\overline{\lambda})]-\frac{1}{w(s)}\sin[t(\overline{\lambda}-z)]\bigg|< 1/{2^i}.$$

		Define $C(t)=0$ for $t \in [0,t_1]$ and $C(t)=C_i$ for $t \in [t_i,t_{i+1}].$ Now let $\varepsilon>0.$  There is $i\geq 1$ such that $2^{-i}<\varepsilon.$ Thus, for all $t>t_i$ 
		
		$$\sup_{\lambda, z \in Q(s,C(t)/t)} \bigg|[A(t,z)C(t,\overline{\lambda})-C(t,z)A(t,\overline{\lambda})]-\frac{1}{w(s)}\sin[t(\overline{\lambda}-z)]\bigg|<\varepsilon.$$
	\end{proof}
	
	From this it follows that the Hermite-Biehler function $E(t,z)$ is close to $\sin z,$ whenever there is a zero in the box $Q(s,C/t).$ We will introduce some notation to keep track of when a zero enters the box. Let $C$ be a constant or slowly growing function and $s\in \R$. We define
	
	$$T_0(s,C)=\{t >0 \mid \text{ there is a zero } z(t) \text{ of } E(t,z) \text{ inside } Q(s,C/t)\},$$
	
	$$T_1(s,C)=\{t> 0 \mid \text{ all zeros } z(t) \text{ of } E(t,z) \text{ inside } Q(s,C/t) \text{ satisfy } \Im z(t) < -1/t\}.$$
	
	Recall that all zeros of $E(t,z)$ are in the lower half plane, since for all $z\in \C_+$
	
	$$\overline{|E(t,z)|}>|E(t,\overline{z})|.$$
		
	For $t>0$ we introduce the function
	
	$$\gamma(t)=\frac{\sqrt{2}}{\sqrt{\sinh[2t]}}.$$
	\begin{lemma}\label{Lemma_6}
		For almost all $s \in \R$ and for all $D>1$ such that $T_0(s,D)$ is unbounded, there exists an increasing function $C(t)>D$ for all $t>0,$ $C(t) \rightarrow \infty$ as $t\rightarrow \infty$ such that the following holds.
		
		Let $z(t)=x(t)-iy(t)$ be a continuous function in $Q(s,C/t)$ for $t\in T_0(s,C(t))$ with $E(t,z(t))=0.$ Then, there exists a unimodular function $\alpha(t),$ such that the following holds
		
	\begin{equation}\label{approximation_by_sine}	
		\sup_{z \in Q(s,C(t)/t)} \bigg| E(t,z) - \frac{\alpha(t)\gamma(ty(t))}{\sqrt{w(s)}} \sin[t(z-z(t))]\bigg|=o(1),
	\end{equation}

		\noindent
		as $t \rightarrow \infty,$ $t \in T_0(s,C(t))\cap T_1(s,C(t)).$ The conclusion holds as well, if we only consider $t\in T_0(s,C(t)),$ however in that case the error term in \eqref{approximation_by_sine} becomes $o(\gamma(ty(t))).$
	\end{lemma}
	
	\begin{proof}
		Let $s \in \R$ and $C_1(t)$ be such that they satisfy Corollary $\ref{corollary_2}.$ By assumption $E(t,z(t))=A(t,z(t))-iC(t,z(t))=0.$ Hence,
		
		$$ A(t,z(t))=iC(t,z(t))=\beta(t).$$
		
		Note that by \eqref{determinant_1} we have $\beta(t) \neq 0.$ For all $\lambda \in \C$ we thus have 
		
		$$A(t,z(t))C(t,\overline{\lambda})-C(t,z(t))A(t,\overline{\lambda})=\beta(t) C(t,\overline{\lambda})+i\beta(t) A(t,\overline{\lambda})=i\beta E(t,\overline{\lambda}).$$
		
		Hence, by Corollary \ref{corollary_2} we have 
		
		\begin{equation}\label{misc8}
		\sup_{z \in Q(s,C_1(t)/t)} \bigg| i\beta(t) E(t,z)- \frac{1}{w(s)} \sin[t(z-z(t))]\bigg|=o(1).
		\end{equation}
		
		We are left to work out $\beta(t).$ It will suffice to show \eqref{approximation_by_sine} for any constant $L>2D.$ By the standard argument we then obtain the statement for an increasing function $C(t).$ The trick will be to cleverly calculate the determinant of the matrix 
		
		$$\begin{pmatrix}
		E(t,z) && E(t,\overline{\lambda}) \\ E^{\sharp}(t,z) && E^{\sharp}(t,\overline{\lambda})
		\end{pmatrix}$$
		
		\noindent
		for any $z,\lambda \in \C$ in two different ways. Once by using our approximation of $E(t,z)$ from above and the other way by using Corollary \ref{corollary_1} by noting that 
		
		$$A(t,z)C(t,\overline{\lambda})-C(t,z)A(t,\overline{\lambda})=\det\begin{pmatrix}
		A(t,z) && A(t,\overline{\lambda}) \\ C(t,z) && C(t,\overline{\lambda})
		\end{pmatrix}= \frac{1}{2i} \det\begin{pmatrix}
		E(t,z) && E(t,\overline{\lambda}) \\ E^{\sharp}(t,z) && E^{\sharp}(t,\overline{\lambda})
		\end{pmatrix}.$$
		
		For all $z,\lambda \in Q(s,L/t)$ we have by \eqref{misc8} 
		
		$$\det\begin{pmatrix}
		i \beta(t) E(t,z) && i \beta(t) E(t,\overline{\lambda}) \\ -i \overline{\beta}(t) E^{\sharp}(t,z) && -i \overline{\beta}(t) E^{\sharp}(t,\overline{\lambda})\end{pmatrix}$$
		$$= \frac{1}{w(s)^2}\det \begin{pmatrix}
		\sin[t(z-z(t))] && \sin[t(\overline{\lambda}-z(t))] \\ \sin[t(z-\overline{z}(t))] && \sin[t(\overline{\lambda}-\overline{z}(t))]
		\end{pmatrix}+o(1)\psi_1(t,z,\lambda).$$

		\noindent
		Multiplying out the determinant and using the formula
		
		$$\sin z = \frac{e^{iz}-e^{-iz}}{2i},$$
		
		\noindent
		the last line further simplifies to 
		
		$$=\frac{1}{w(s)^2}\sin[2ity(t)]\sin[t(\overline{\lambda}-z)]+o(1)\psi_1(t,z,\lambda) =\frac{i}{w(s)^2}\sinh[2ty(t)]\sin[t(\overline{\lambda}-z)]+o(1)\psi_1(t,z,\lambda).$$
		
		\noindent
		as $t \rightarrow \infty, t\in T_0(s,L)$ for some bounded function $\psi_1.$ On the other hand by Corollary \ref{corollary_1}
		
		$$\frac{2i}{w(s)} \sin[t(\overline{\lambda}-z)]=\det\begin{pmatrix}
		E(t,z) && E(t,\overline{\lambda}) \\ E^{\sharp}(t,z) && E^{\sharp}(t,\overline{\lambda})
		\end{pmatrix} + o(1)\psi_2(t,z,\lambda)$$
		
		\noindent
		as $t\rightarrow \infty, t\in T_0(s,L)$ for some bounded function $\psi_2.$	Combining both equations we arrive at
		
		$$2|\beta(t)|^2w(s)=(1+o(1))\sinh[2ty(t)].$$
		
		\noindent
		Hence, we now find an unimodular continuous function $\alpha$ such that
		
		$$\beta(t)=(1+o(1))\frac{\alpha(t)\sinh[2ty(t)]}{2w(s)}.$$
		
		Substituting this back into \eqref{misc8} we get
		
		$$\sup_{z \in Q(s,L/t)} \bigg| E(t,z) - \frac{\alpha(t)\sqrt{2}}{\sqrt{w(s)}\sqrt{\sinh[2ty(t)]}} \sin[t(z-z(t))]\bigg|=o\bigg(\frac{1}{\sqrt{\sinh[2ty(t)]}}\bigg)$$
		
		\noindent
		as $t\rightarrow \infty, t\in T_0(s,L).$ If additionally, $t\in T_1(s,L)$ the error term becomes $o(1),$ as desired.
	\end{proof}

	Let us now investigate what happens when there is no zero of $E(t,z)$ close to the real line. Suppose $z(t)=x(t)-iy(t)$ is a continuous path traced out by a zero in the lower half-plane. Assume that for some large $C>0$ the zero $z(t)$ is close to the bottom boundary of the box $Q(s,C/t)$ and \eqref{approximation_by_sine} holds for all $t>0$. Then, $y(t)$ is very large. We calculate
	
	$$\frac{\gamma(ty(t))}{\sqrt{w(s)}} \sin[t(z-z(t))]=\frac{\gamma(ty(t))}{\sqrt{w(s)}}\sin[t(z-x(t)+iy(t))]$$
	
	$$=\frac{\gamma(ty(t))}{\sqrt{w(s)}}\bigg(\sin[t(z-x(t))]\cos[ity(t)]+\sin[ity(t)]\cos[t(z-x(t))]\bigg).$$
	
	\noindent
	If $y(t) \rightarrow \infty$ we observe that 
	
	$$\frac{\sqrt{2}}{\sinh[2ty(t)]}\cos[ity(t)] \rightarrow 1, \quad \frac{\sqrt{2}}{\sqrt{\sinh[2ty(t)]}}\sin[ity(t)] \rightarrow i.$$
	
	\noindent
	Hence, 
	
	$$\bigg|\frac{\gamma(ty(t))}{\sqrt{w(s)}} \sin[t(z-z(t))]- \frac{\beta(t)}{\sqrt{w(s)}}e^{itz}\bigg|=o(1)$$
	
	\noindent
	for $\beta(t)=e^{-itx(t)}.$ Thus, we expect that $E(t,z)$ is close to $e^{itz},$ if the zeros are far away from the real line. The calculation above is rigorous only if there is a zero near the boundary of the box and \eqref{approximation_by_sine} holds. Thus, let us check what happens if there is no zero inside the box.

	Let $s\in \R$ and $C(t)$ be like in Corollary \ref{corollary_2}. We will consider only $t>0$ such that the box $Q(s,C(t)/t)$ does not contain any zeros of $E(t,z),$ i.e. $t \notin T_0(s,C(t)).$ We set $I_t=\R\cap Q(s,C(t)/t).$ For any $x,w \in I_t$ we can interpret Corollary \ref{corollary_2} as an approximation of a scalar product of two real vectors
	
	\begin{equation}\label{misc9}
	\bigg \langle  \begin{pmatrix}
	A(t,x) \\ C(t,x)
	\end{pmatrix}, \begin{pmatrix}
	C(t,w) \\ -A(t,w)
	\end{pmatrix} \bigg \rangle = \frac{1}{w(s)} \sin[t(w-x)]+o(1)\psi(t,x,w)
	\end{equation}
	
	\noindent
	as $t\rightarrow \infty$ for some bounded function $\psi.$ We will assume that $t$ is large enough, such that $o(1)\psi(t,x) \leq \frac{1}{w(s)}$ and $C(t)\geq 4\pi.$ Choose points $w_1 (t), w_2 (t) = w_1 (t) + \frac{\pi}{2t}.$ Since \eqref{misc9} holds for all $x \in I_t,$ we see that the vector 
	
	$$\begin{pmatrix}
	A(t,x) \\ C(t,x)
	\end{pmatrix}$$
	
	\noindent
	has modulus bounded away from zero and rotates around the origin. In particular, since $C(t)\geq 4\pi$ the vector makes at least one full rotation. Hence, there are $x_0(t)$ and $x_1(t)$ such that $E(t,x_0(t))$ is positive and $E(t,x_1(t))$ is positive imaginary, i.e. $C(t,x_0(t))=0,$ and $A(t,x_1(t))=0.$ If we set $x=x_0(t)$ and $w=x_1(t)$ in \eqref{misc9}, we see that $A(t,x_0(t))=c_1(t)$ and $C(t,x_1(t))=-c_2(t),$ where $c_1(t)$ and $c_2(t)$ are positive numbers satisfying $-c_1(t)c_2(t)=\frac{1}{w(s)}\sin[t(x_0(t)-x_1(t)]+ o(1).$
	
	\noindent
	We set $x_2(t)=x_1(t)-\frac{\pi}{2t}.$ By Corollary \ref{corollary_2} we have for $z \in Q(s,C(t)/t)$
	
	$$C(t,z)= \frac{1}{c_1(t)w(s)} \sin[t(z-x_0(t))] +o(1), \quad A(t,z)=-\frac{1}{c_2(t)w(s)}\cos[t(z-x_2(t))]+o(1).$$
	
	Hence, for all $z \in Q(s,C(t)/t)$ we find the approximation
	
	$$E(t,z)=-\frac{1}{c_2(t)w(s)}\cos[t(z-x_2(t))] -i\frac{1}{c_1(t)w(s)} \sin[t(z-x_0(t))] +o(1).$$
	
	This is almost $-e^{itz}.$ We have the following
	
	\begin{lemma}\label{approximation_by_exp}
		$$|t(x_2(t)-x_0(t))|=o(1) \mod 2\pi, \quad  |c_1(t)-c_2(t)|=o(1).$$
	\end{lemma}

	 We will postpone the proof to section $3.4$ where we study the dynamics of the zeros. Thus, $c_1,c_2=\frac{1}{\sqrt{w(s)}}+o(1)$ and there is a unimodular function $\beta(t)$ and an increasing function $C_1(t),$ $C_1(t) \rightarrow \infty$ as $t \rightarrow \infty,$ such that for all $z\in Q(s,C_1(t)/t)$ we have
		
		$$\frac{\beta(t)}{\sqrt{w(s)}}e^{itz}.$$
	
	We summarize these results in the following 
	
	\begin{corollary}\label{corollary_4}
		For almost all $s \in \R$ there exists a function $C(t)>0, C(t) \rightarrow \infty$ as $t\rightarrow \infty$ and $z(t)=x(t)-iy(t) \in \C_-$ such that 
		
		\begin{equation}\label{approximation_by_sine2}	
		\sup_{z \in Q(s,C(t)/t)} \bigg| E(t,z) - \frac{\alpha(t)\gamma(ty(t))}{\sqrt{w(s)}} \sin[t(z-z(t))]\bigg|=o(1),
		\end{equation}
		
		\noindent
		for some $\alpha=\alpha(s,t), |\alpha|=1$ as $t\rightarrow \infty, t\in T_1(s,C(t)).$ If $t\in T_0(s,C(t)),$ $z(t)$ can be chosen as a zero of $E(t,z).$ If additionally $t \notin T_0(s,C(t))$ for all sufficiently large $t,$ then for all $D>0,$
		
		\begin{equation}\label{approximation_by_exp2}
			\sup_{z\in Q(s,D/t)} \bigg|E(t,z)- \frac{-i\alpha(s,t)}{\sqrt{w(s)}}e^{itz}\bigg|=o(1).
		\end{equation}
		
		One can drop the assumption $ty(t)>1,$ if one replaces the error term $o(1)$ by $o(\gamma(ty(t)))$ in \eqref{approximation_by_sine2}.
		
	\end{corollary}

	In conclusion, if the zeros of $E(t,z)$ are bounded away by at least $1/t$ from the real line, then $E(t,z)$ behaves like $e^{itz}.$ In that case, we can immediately conclude the convergence of the non-linear Fourier transform, see section $3.8.$ However, if a zero is close to the real line, $E(t,z)$ is close to $\sin z,$ which is not good enough to conclude convergence. Hence, from this point on we will need to show that all zeros of $E(t,z)$ are bounded locally at least by $1/t$ away from the real line. 
	
	If the zeros of $E(t,z)$ do not stay away from the box $Q(s,C(t)/t),$ i.e. $T_0(s,C)$ is unbounded, there are precisely two situations in which this can happen. Fix $s \in \R$ and $C>0.$
	
	\begin{itemize}
		\item [(i)] There is a time $T>0$ s.t. for all $t>T$ the box $Q(s,C/t)$ contains a zero of $E(t,z),$ i.e. $T_0(s,C)$ covers a half-line, so after time $T$ at least one zero stays in the box forever.
		
		\item[(ii)] For all times $T>0,$ such that $Q(s,C/T)$ does not contain a zero of $E(t,z),$ there exists $t>T$ such that the box $Q(s,C/t)$ contains a zero of $E(t,z),$ i.e. the zeros jump in and out of the box infinitely often.
	\end{itemize}

	While we can rule out (i) by the results proven already, we will have to put a lot more effort into showing that case (ii) can not occur.
	
	\subsection{Resonances can not stay close to the real line forever}
	In this section, we are going to show that for almost all $s\in \R$ and any $C>0$ the zeros of $E(t,z)$ can not stay in the box $Q(s,C/t)$ forever. Our strategy will be as follows. If at least one zero of $E(t,z)$ were in the box $Q(s,C/t)$ for all large enough $t,$ then by Lemma \ref{Lemma_6} $|E(t,x)|$ is close to $|\sin(tx)|$ near $s.$ Since $|\sin(tx)|$ does not converge in measure on $\R,$ the same holds for the function $|E(t,x)|.$ However, this contradicts Lemma \ref{measure}. We set 
	
	$$\Sigma=\{s \in \R \mid \text{There exists } T>0 \text{ such that for all } t>T \text{ the box } Q(s,C/t) \text{ contains a zero}\}.$$
	
	We can assume without loss of generality that $0<|\Sigma|<\infty.$  For two sets $A$ and $B$ we denote the symmetric difference by
	
	$$A\Delta B=(A\setminus B) \cup (B \setminus A).$$
	
	The following well-known result allows us to assume that $\Sigma$ is a finite union of intervals.
	
	\begin{lemma}
		Let $\Sigma$ be a measurable set with $|\Sigma| < \infty.$ Then, for all $\varepsilon > 0$ there are finitely many intervals $I_1, \dots I_n,$ such that 
		
		$$\bigg|\Sigma \Delta \bigcup_{k=1}^{n} I_k\bigg| < \varepsilon.$$
	\end{lemma}
	
	We will contradict convergence on each interval separately, so we might as well assume from the very beginning that $\Sigma$ is an interval.
	
	The following criterion allows us to show that a given sequence of functions does not converge in measure.
	
	\begin{lemma}\label{measure2}
		Let $\Sigma$ be a measurable set with $|\Sigma|>0$ and let $f_n(x)$ be a sequence of real functions. Suppose there exists $\delta>0$  and $N\geq1$ such that for all $x \in \Sigma$ and all $n,m\geq N$ we have
		
		$$||f_n(x)|-|f_m(x)||>\delta.$$
		
		Then, $|f_n|$ does not converge in measure on $\Sigma.$
	\end{lemma}

	\begin{proof}
		Suppose $|f_n|\rightarrow f$ as $n\rightarrow \infty$ for some function $f.$ Then, there exists $N\geq 0$ such that for all $n\geq N$
		
		$$\big|\{x \in \Sigma : ||f_n|-f|>\delta\}\big|< \frac{|\Sigma|}{2}.$$
		
		Thus, we immediately obtain a contradiction
		
		$$|\Sigma|= \big|\{x \in \Sigma :||f_n(x)|-|f_m(x)||>\delta \}\big|$$
		$$\leq \big|\{x \in \Sigma : ||f_n(x)|-f|+||f_m(x)|-f| > \delta \} \big|$$
		$$\leq \big|\{x \in \Sigma : ||f_n(x)|-f|> \delta \} \big| + \big|\{x \in \Sigma : ||f_m(x)|-f| > \delta \}\big|$$
		$$< \frac{|\Sigma|}{2}+ \frac{|\Sigma|}{2}=|\Sigma|.$$
	\end{proof}

	Furthermore, we will need the following technical inequality.
	
	\begin{lemma}\label{measure3}
		Let $\Sigma$ be a non-empty interval. There exists $\delta>0$ independent of $\Sigma$ and $N\geq1$ dependent on $\Sigma,$ such that for all $a,b>0$ and $z,w \in Q(s,C)\cap \C_-$ and all $n\geq N$ we have 
		
		\begin{equation}\label{misc10}
		\bigg| a\big|\sin[2^n(x-z)]\big|-b\big|\sin[2^{n+1}(x-w)]\big| \bigg| > \max(a,b)\delta
		\end{equation}
		
		\noindent
		on a subset of $\Sigma$ of measure at least $\frac{\delta}{24} |\Sigma|.$
	\end{lemma}

	\begin{proof}
		Let us first prove the case where $a=b=1$ and $\Sigma=[0,\frac{\pi}{2^n}]$ for some $n\geq1.$ The function 
		
		$$f_{z,w,n}(x)=\big|\sin[2^n(x-z)]\big|-\big|\sin[2^{n+1}(x-w)]\big|$$
		
		\noindent
		is periodic with shortest period $\frac{\pi}{2^n}.$ The factor $2^n$ just squeezes the function from the interval $\Sigma=[0,\frac{\pi}{2}]$ onto the interval $[0,\frac{\pi}{2^n}]$ which covers exactly the shortest period of $f_{z,w,n}.$ Hence, it suffices to show the inequality for $f_{z,w}=f_{z,w,1}$ and $\Sigma=[0,\frac{\pi}{2}].$
		
		Observe that by continuity of $\sup$ and $\inf$ on compact sets, we have the following min-max inequalities
		
		$$\delta_1=\sup_{z,w \in Q(s,C)} \inf_{x \in [0,\frac{\pi}{2}]} \big|\sin[2(x-z)]\big|-\big|\sin[4(x-w)]\big|<0,$$
		
		$$\delta_2=\inf_{z,w \in Q(s,C)} \sup_{x \in [0, \frac{\pi}{2}]} \big|\sin[2(x-z)]\big| - \big|\sin[4(x-w)]\big|>0.$$
		
		Set $\delta=\min(|\delta_1|,|\delta_2|)>0.$ Furthermore, $f_{z,w}$ is Lipschitz continuous, since $f_{z,w}'$ is uniformly bounded for all $z,w\in Q(s,C)\cap\C_-.$ Indeed, by the same method as in the proof of Lemma \ref{derivative_formula} we find for $z=a+ib\in Q(s,C)\cap\C_-$ that

		$$\frac{d}{dx} \big|\sin[2(x-z)]\big| = \frac{d}{dx} \big|\sin[2(x-a)]\cosh(2b)+i\cos[2(x-a)]\sinh(2b)\big|$$
		$$=\frac{2\cos[2(x-a)]\sin[2(x-a)](\cosh^2(2b)-\sinh^2(2b))}{\big|\sin[2(x-z)]\big|}$$
		$$=2\frac{\cos[2(x-a)]\sin[2(x-a)]}{\big|\sin[2(x-z)]\big|}.$$
		
		Notice that for all $z\in Q(s,C)$ we have
		
		$$\frac{\cos[2(x-a)]\sin[2(x-a)]}{\big|\sin[2(x-z)]\big|} \leq 1.$$

		Hence we have for all $z,w\in Q(s,C)$ and for $x\in [0,\frac{\pi}{2}]$ 
		
		$$|f_{z,w}'(x)| \leq 6.$$
		
		Choose a point $x_0 \in [0,\frac{\pi}{2}]$ such that $f_{z,w}(x_0)>\delta.$ Then, if $|x_0-y|<\frac{\delta}{12},$ we have 
			
			$$|f_{z,w}(x_0)-f_{z,w}(y)|<\frac{\delta}{2}.$$
		
		Thus, on an interval $I_1$ of length $\frac{\delta}{6}$ the inequality \eqref{misc10} holds with $\frac{\delta}{2}.$ Similarly take a point $x_1 \in [0,\frac{\pi}{2}]$ such that $f_{z,w}(x_1)<-\delta.$ By the same argument we obtain an interval $I_2$ of length $\frac{\delta}{6}$ on which the inequality \eqref{misc10} holds with $\frac{\delta}{2}.$
		
		Let now $a,b>0$ and suppose that $a\geq b.$ Then for all $x \in I_1$
		
		$$a\big|\sin[2(x-z)]\big|-b\big|\sin[4(x-w)]\big| \geq  a\big(\big|\sin[2(x-z)]\big|-\big|\sin[4(x-w)]\big|\big)=\max(a,b)\frac{\delta}{2}.$$
		
		In the case $a\leq b$, a similar inequality holds on $I_2.$ By translation, this concludes the proof for the case $\Sigma=[k\frac{\pi}{2^n},(k+1)\frac{\pi}{2^n}]$ for any $n\geq1$ and $k\in \Z.$ Let now $\Sigma$ be any non-empty interval. Then, choose $N$ large enough such that there are finitely many scaled dyadic intervals $I_1 \dots I_m$ of the form $[k\frac{\pi}{2^N},(k+1)\frac{\pi}{2^N}],$ all contained in $\Sigma,$ satisfying
		
		$$\bigg|\Sigma \setminus \bigcup_{n=1}^m I_n\bigg|\leq\frac{|\Sigma|}{2}.$$
		
		Then, the inequality is satisfied for all $n\geq N$ on a subset of $\Sigma$ of measure at least 
		
		$$\frac{\delta}{12}\bigg| \bigcup_{n=1}^m I_n \bigg|=\frac{\delta}{12}\bigg(|\Sigma|- \bigg|\Sigma \setminus \bigcup_{n=1}^m I_n\bigg| \bigg) \geq \frac{\delta}{24}|\Sigma|.$$
	\end{proof}

	\begin{corollary}\label{corollary_3}
		$|\Sigma|=0.$
	\end{corollary}

	\begin{proof}
		Suppose $|\Sigma|>0.$ For all $s \in \Sigma$ there is $T(s)>0$ such that the box $Q(s,C/t)$ contains a zero of $E(t,z)$ for all $t>T(s)$ and the left hand side of \eqref{approximation_by_sine} is smaller than $\varepsilon \gamma(ty(t)),$ where $\varepsilon>0$ is a small number, which we specify later. By Lusin's Theorem, we can assume that $T(s)$ is the same for all $s \in \Sigma.$ Using Lusin's Theorem again we can assume that $w(s)<D<\infty$ for all $s\in \Sigma.$ Let $n\geq 1$ such that $2^n>T.$ Observe that 
		
		$$\bigg|E(2^n,x)-\alpha(2^n)\frac{\gamma(2^ny(2^n))}{\sqrt{w(s)}} \sin[2^n(x-z(2^n))]\bigg|<\varepsilon \gamma(2^ny(2^n))$$
		
		\noindent
		implies
		
		$$|E(2^n,x)|\leq \varepsilon \gamma(2^ny(2^n))+ \frac{\gamma(2^ny(2^n))}{\sqrt{w(s)}}\big|\sin[2^n(x-z(2^n))]\big|.$$
		
		Note that by inverse triangle inequality
		
		$$|E(2^{n+1},x)|\geq  \frac{\gamma(2^{n+1}y(2^{n+1}))}{\sqrt{w(s)}} \big|\sin[2^{n+1}(x-z(2^{n+1}))]\big|$$
		$$ -\bigg|E(2^{n+1},x) - \alpha(2^{n+1})\frac{\gamma(2^{n+1}y(2^{n+1}))}{\sqrt{w(s)}} \sin[2^{n+1}(x-z(2^{n+1}))]\bigg|$$
		$$\geq \bigg|\frac{\gamma(2^{n+1}y(2^{n+1}))}{\sqrt{w(s)}} \sin[2^{n+1}(x-z(2^{n+1}))]\bigg| - \varepsilon \gamma(2^{n+1}y(2^{n+1})).$$
		
		Thus,
		
		$$\bigg| |E(2^n,x)|-|E(2^{n+1},x)|\bigg| \geq |E(2^n,x)|-|E(2^{n+1},x)|$$
		$$\geq  \frac{\gamma(2^{n+1}y(2^{n+1}))}{\sqrt{w(s)}}\big|\sin[2^{n+1}(x-z(2^{n+1}))]\big| 
		- \frac{\gamma(2^ny(2^n))}{\sqrt{w(s)}}\big|\sin[2^n(x-z(2^n))]\big|$$
		$$ - \varepsilon \gamma(2^{n+1}y(2^{n+1})) - \varepsilon \gamma(2^ny(2^n)).$$
		
		\noindent
		We do the same calculation after reversing the roles of $E(2^n,x)$ and $E(2^{n+1},x).$ Combining both inequalities gives
		
		$$\bigg| |E(2^n,x)|-|E(2^{n+1},x)|\bigg| \geq |E(2^n,x)|-|E(2^{n+1},x)|$$
		$$\geq  \bigg|\frac{\gamma(2^{n+1}y(2^{n+1}))}{\sqrt{w(s)}}\big|\sin[2^{n+1}(x-z(2^{n+1}))]\big| 
		- \frac{\gamma(2^ny(2^n))}{\sqrt{w(s)}}\big|\sin[2^n(x-z(2^n))]\big|\bigg|$$
		$$ - \varepsilon \gamma(2^{n+1}y(2^{n+1})) - \varepsilon \gamma(2^ny(2^n)).$$

		\noindent
		Note that since $z(t) \in Q(s,C/t)$ we have for all $t>0$
	
		$$\gamma(ty(t))=\frac{\sqrt{2}}{\sqrt{\sinh[2ty(t)]}}\geq \frac{\sqrt{2}}{\sqrt{\sinh[2C]}}=C_1>0.$$
		
		\noindent
		Recall that we assumed that $w(s)<D<\infty.$ Hence, we can apply Lemma \ref{measure3} with $\Sigma$ to obtain $\delta>0$ such that on a subset of $\Sigma$ of measure at least $\frac{\delta}{12}|\Sigma|$ we have 
	
		$$\bigg| |E(2^n,x)|-|E(2^{n+1},x)|\bigg| >\max\bigg(\gamma(2^{n+1}y(2^{n+1})),\gamma(2^{n+1}y(2^{n+1}))\bigg)\delta$$
		$$ - \varepsilon \gamma(2^{n+1}y(2^{n+1})) - \varepsilon \gamma(2^ny(2^n)). $$
	
		\noindent
		We choose now $\varepsilon=\frac{\delta}{4}.$ Then,
		
		$$\bigg| |E(2^n,x)|-|E(2^{n+1},x)|\bigg| > \max\bigg(\gamma(2^{n+1}y(2^{n+1})),\gamma(2^{n+1}y(2^{n+1}))\bigg) \frac{\delta}{2} \geq \frac{C_1}{2\sqrt{D}}\delta >0$$
		
		\noindent
		on a subset of $\Sigma$ of measure at least $\frac{\delta}{12}|\Sigma|,$ contradicting the convergence of $|E(t,x)|$ in measure on $\R.$
	\end{proof}
	
	We now showed that the zeros of $E(t,z)$ can not stay inside the box $Q(s,C/t)$ forever. For the remainder of this discussion, we will have to rule out that the zeros travel infinitely many times in and out of the box $Q(s,C/t).$ It is a natural question to ask, whether one can modify the methods we used slightly to get this stronger result. The key inequality which made everything work is Lemma \ref{measure3}. It is stated in particular such that we can use the approximation from Lemma \ref{Lemma_6} for the times $t_1=2^n$ and $t_2=2^{n+1}.$ This is only possible since we know by assumption that there is a zero inside the box $Q(s,C/t)$ for all large enough times. In the stronger case we want to prove, this does not hold anymore, i.e. one would have to replace $t_1=2^n,t_2=2^{n+1}$ by other values of $t.$ The problem arises since we can not give a lower bound on the time interval a zero has to stay inside the box $Q(s,C/t),$ i.e. the inequality \eqref{misc10} breaks down. More precisely, $|t_2-t_1|\rightarrow 0$ gives $\delta \rightarrow 0,$ so we can not hope anymore to contradict convergence in measure of $|E(t,x)|.$
	
	Thus, we will need to further analyze the dynamics of the zeros to have extra leverage for proving the stronger claim. It will turn out that the contradiction arises through a different observation. Namely, we will observe that each time the zero travels in and out of the box $Q(s,C/t),$ this will require a chunk of $L^2-$norm of $f.$ Hence, if a zero of $E(t,z)$ traveled infinitely often in and out of the box $Q(s,C/t),$ this would immediately contradict the fact that $f\in L^2(\R).$
	
	\subsection{Meromorphic inner functions and dynamics of the zeros}
	We are going to study the dynamics of the zeros of $E(t,z).$ The appropriate tool for doing so is the class of meromorphic inner functions, which can be assigned to any Hermite-Biehler function.
	
	Let $f:\C_+\rightarrow\C$ be an inner function. We say that $f$ is an meromorphic inner function, if $f$ admits a meromorphic continuation to the whole complex plane. Meromorphic inner functions enjoy the property, that they have a product representation in terms of their zeros and poles.
	
	\begin{lemma}\label{MIF}
		Let $f:\C \rightarrow \C$ be a meromorphic inner function. Then, there exists a sequence of points $\lambda_n \in \C_+$ with $|\lambda_n|\rightarrow \infty$ as $n \rightarrow \infty$ satisfying the Blaschke condition
		
		$$ \sum_{n=1}^{\infty} \frac{\Im \lambda_n}{1+|\lambda_n|^2}<\infty.$$ 
		
		Furthermore, there is $c\geq0$ and a unimodular constant $\alpha$ such that 
		
		$$f(z)=\alpha e^{icz}\prod_{n=1}^{\infty} \frac{\overline{\lambda}_n}{\lambda_n} \frac{z-\lambda_n}{z-\overline{\lambda}_n}.$$
	\end{lemma}

	\begin{proof}
		A proof and many more results on meromorphic inner functions, as well as further applications can be found in \cite{Poltoratski2}.
	\end{proof}

	From this we immediately conclude that every meromorphic inner function is analytic in a neighborhood of the real line and hence all boundary values and derivatives are well-defined everywhere on $\R.$ Furthermore, the derivatives on $\R$ can be used to locate the zeros.
	
	\begin{lemma}\label{Lemma_1}
		Let $\theta$ be a meromorphic inner function and let $0<\varepsilon<\varepsilon_0$ for some small fixed constant $\varepsilon_0.$ Suppose there are $x,y \in \R$ such that 
		
		\begin{equation}\label{misc11}
		\frac{|\theta'(x)|}{|\theta'(y)|}>1+\varepsilon.
		\end{equation}
		
		Then, the ball $\{|z-x| < 4|y-x|/\varepsilon\}$ contains a zero of $\theta.$
	\end{lemma}

	\begin{proof}
		We use Lemma \ref{MIF} to write 
		
		$$\theta(z)=\alpha e^{icz}\prod_{n=1}^{\infty} \frac{\overline{\lambda}_n}{\lambda_n} \frac{z-\lambda_n}{z-\overline{\lambda}_n}$$
		
		\noindent
		with a unimodular constant $\alpha,$ a non-negative constant $c$ and a sequence of points $\lambda_n\in \C_+$ satisfying 
		
		$$ \sum_{n=1}^{\infty} \frac{\Im \lambda_n}{1+|\lambda_n|^2}<\infty.$$ 
		
		Assume without loss of generality that $\alpha=1.$ Since $|\theta(x)|=1$ for almost all $x\in \R,$ we can write $\theta(x)=e^{i\arg\theta(x)}.$ Together with the relation $\ln\theta(x)=i\arg\theta(x)$ we see that
		
		$$|\theta'(x)|=\frac{d}{dx}\arg\theta(x)=\frac{1}{i}\frac{d}{dx}\ln\theta(x).$$
		
		Using the product representation of $\theta(x)$ we calculate
		
		$$\frac{d}{dx} \ln\theta(x)=\frac{d}{dx}\ln\bigg(e^{icx}\prod_{n=1}^{\infty} \frac{\overline{\lambda}_n}{\lambda_n} \frac{x-\lambda_n}{x-\overline{\lambda}_n}\bigg)=ic+\sum_{n=1}^{\infty}\frac{d}{dx}\ln\bigg( \frac{\overline{\lambda}_n}{\lambda_n} \frac{x-\lambda_n}{x-\overline{\lambda}_n}\bigg).$$
		
		We will now calculate the derivative of a single Blaschke factor. Write $\lambda_n=x_n+iy_n.$
		
		$$\frac{d}{dx}\ln\bigg( \frac{\overline{\lambda}_n}{\lambda_n} \frac{x-\lambda_n}{x-\overline{\lambda}_n}\bigg)=\frac{d}{dx}\ln(x-\lambda_n)-\frac{d}{dx}\ln(x-\overline{\lambda}_n)$$
		$$=\frac{1}{(x-x_n)-iy_n}-\frac{1}{(x-x_n)+iy_n}=\frac{2iy_n}{(x-x_n)^2+y_n^2}.$$
		
		Thus, combining everything we see
		
		$$|\theta'(x)|=c+ \sum_{n=1}^{\infty} \frac{2y_n}{(x-x_n)^2+y_n^2}.$$
		
		By \eqref{misc11} there is an $n\geq 1$ such that 
		
		\begin{equation}\label{misc12}
		\frac{y_n}{(x-x_n)^2+y_n^2}\frac{(y-x_n)^2+y_n^2}{y_n}=\frac{|y-\lambda_n|^2}{|x-\lambda_n|^2}> 1+\varepsilon.
		\end{equation}
		
		\noindent
		Indeed, suppose the reverse inequality would hold for all $n\geq 1.$ Then,
		
		$$\frac{|\theta'(x)|}{|\theta'(y)|}=\frac{c+ \sum_{n=1}^{\infty} \frac{2y_n}{(x-x_n)^2+y_n^2}}{c+ \sum_{n=1}^{\infty} \frac{2y_n}{(y-x_n)^2+y_n^2}}$$
		
		$$\leq \frac{c+ (1+\varepsilon)\sum_{n=1}^{\infty} \frac{2y_n}{(y-x_n)^2+y_n^2}}{c+ \sum_{n=1}^{\infty} \frac{2y_n}{(y-x_n)^2+y_n^2}}= 1+ \varepsilon \frac{\sum_{n=1}^{\infty} \frac{2y_n}{(y-x_n)^2+y_n^2}}{c+ \sum_{n=1}^{\infty} \frac{2y_n}{(y-x_n)^2+y_n^2}}\leq 1+\varepsilon,$$
		
		\noindent
		contradicting \eqref{misc11}. Hence, pick $n\geq 1$ such that \eqref{misc12} holds. Then, $\lambda_n$ is contained in the ball $\{|z-x| < \frac{4}{3}|y-x|/\varepsilon\}.$ Indeed,
		from 
		
		$$\frac{|y-\lambda_n|}{|x-\lambda_n|}>\sqrt{1+\varepsilon}$$
		
		\noindent
		we find by triangle inequality that
		
		$$\sqrt{1+\varepsilon}|x-\lambda_n|<|x-y|+|x-\lambda_n| \implies |x-\lambda_n|< \frac{|x-y|}{\sqrt{1+\varepsilon}-1}.$$
		
		We use Taylor series expansion of $\sqrt{1+x}$
		
		$$\sqrt{1+x}=\sum_{k=0}^{\infty} \binom{1/2}{k}x^k$$
		
		\noindent
		and cut it off at the linear term to write $\sqrt{1+\varepsilon}=1+\frac{\varepsilon}{2}+o(\varepsilon).$ For $0<\varepsilon<\varepsilon_0$ small enough we have that $|o(\varepsilon)|<\frac{1}{4}\varepsilon.$ Thus, as claimed
		
		$$|x-\lambda_n|< \frac{|x-y|}{\sqrt{1+\varepsilon}-1}<\frac{|x-y|}{1+\frac{\varepsilon}{2}-\frac{\varepsilon}{4}-1}=4\frac{|x-y|}{\varepsilon}.$$
	\end{proof}

	If $E(z)$ is any Hermite-Biehler function we can associate a meromorphic inner function
	by
	
	$$\theta_{E}(z)=\frac{E^{\sharp}(z)}{E(z)}.$$

	\noindent
	In particular, we will denote for the remainder of the paper the associated meromorphic inner functions to the Hermite-Biehler functions $E(t,z)$ and $\tilde{E}(t,z)$ respectively by
	
	$$\theta_{E}(t,z)=\frac{E^{\sharp}(t,z)}{E(t,z)}, \quad \theta_{\tilde{E}}=\frac{\tilde{E}^{\sharp}(t,z)}{\tilde{E}(t,z)}.$$
	
	\noindent
	Note that the zeros of $E(t,z)$ are the precisely the poles of $\theta_{E}$ and the conjugates of the zeros of $E(t,z)$ are precisely the zeros of $\theta_E.$
	
	We now have the tools to give a proof for Lemma \ref{approximation_by_exp}.
	
	\begin{proof}[Proof of Lemma \ref{approximation_by_exp}]
	We set $\phi(x)=\arg E(t,x)$ and set $J_t(s-\frac{4\pi}{t},s+\frac{4\pi}{t}).$ Note that 
	
	$$|\theta_E'(x)|=\bigg|\frac{d}{dx} e^{-2i\arg E(t,x)}\bigg|=2|\phi'(x)|.$$

	If there now was $\delta>0$ such that $\limsup_{t\rightarrow \infty} |t(x_2(t)-x_0(t)|>\delta$ or $\limsup_{t\rightarrow \infty} |c_1(t)-c_2(t)|>\delta,$ then
	
	$$ \frac{\sup_{J_t}\phi'}{\inf_{J_t}\phi'}>1+\varepsilon,$$
	
	\noindent
	for some $\varepsilon>0$ depending on $\delta.$ However, since $C(t)\rightarrow \infty,$ Lemma \ref{Lemma_1} gives that there is a zero of $E(t,z)$ in $Q(s,C(t)/t)$ for large enough $t,$ contradicting our assumption $t\notin T_0(s,C(t)).$
	\end{proof}
	
	Let now $z(t)$ be a differentiable curve in $\C$ such that $\theta=\theta_{E}(t,z(t))=0.$ We want to study further $z'(t)$ and $\theta_z(t,z(t)).$ Suppose in the following that $z\in \C$ is such that $E(t,z)\neq 0.$ We have 
	
	$$\frac{d}{dt}\theta(t,z)=\frac{(\frac{d}{dt}A+i\frac{d}{dt}C)(A-iC)-(A+iC)(\frac{d}{dt}A-i\frac{d}{dt}C)}{(A-iC)^2}=2i\frac{A\frac{d}{dt}C-C\frac{d}{dt}A}{(A-iC)^2}$$
	
	$$=2i\frac{z(A^2+C^2)-2fAC}{(A-iC)^2}=2iz\frac{A+iC}{A-iC}-4if\frac{AC}{(A-iC)^2}$$
	
	$$=2iz\frac{A+iC}{A-iC}-f \frac{[(A+iC)-(A-iC)][(A+iC)+(A-iC)]}{(A-iC)^2}=2iz\theta(t,z)+f(1-\theta^2(t,z)).$$
	
	Thus, the meromorphic inner function $\theta$ satisfies the differential equation
	
	\begin{equation}\label{misc13}
	\frac{d}{dt}\theta=2iz\theta+f(1-\theta^2).
	\end{equation}
	
	Let now $t>0$ be such that $\theta_z(t,z(t))\neq 0.$ Thus,
	
	$$\frac{d}{dt}\theta(t,z(t))=0=\theta_t(t,z(t))+\theta_z(t,z(t))z'(t),$$
	
	\noindent
	which immediately yields using \eqref{misc13}
	
	\begin{equation}\label{derivative_zero}
	z'(t)=-\frac{f(t)}{\theta_z(t,z(t))}.
	\end{equation}
	
	Furthermore, it will be useful to know the change of the derivative of $\theta_z$ at the zero $z(t).$ Using \eqref{misc13} and \eqref{derivative_zero} we obtain
	
	\begin{equation}\label{riccati}
	\frac{d}{dt}\theta_z(t,z(t))=\theta_{zt}(t,z(t))+\theta_{zz}(t,z(t))z'(t)=2iz(t)\theta_z(t,z(t))-f(t)\frac{\theta_{zz}(t,z(t))}{\theta_z(t,z(t))}
	\end{equation}
	
	The natural question to ask is when \eqref{derivative_zero} is well-defined, i.e. when $\theta_z(t,z(t))\neq 0.$ It turns out that when there is a zero in the box $Q(s,C/t),$ then its is a simple zero for large enough $t.$ To see this we will prove a more general result that will turn out useful in many situations later in the proof. Essentially it is a quantitative version of the open mapping Theorem. The proof is taken from \cite{Stein}.
	
	\begin{lemma}\label{open_mapping_theorem}
		Let $\Omega \subset \C$ be a domain and $f: \Omega \rightarrow \C$ a holomorphic function. Let $z_0 \in \Omega$ and $w_0=f(z_0)$ and choose $\delta>0$ such that $\overline{B(z_0,\delta)} \subset \Omega$ and $f(z_0)\neq w_0$ on the circle $|z-z_0|=\delta.$
		
		If $\varepsilon>0$ is such that $|f(z)-w_0|\geq \varepsilon$ on the circle $|z-z_0|=\delta,$ then $f$ attains every value in $B(w_0,\varepsilon).$
	\end{lemma}
	
	\begin{proof}
		Define $g(z)=f(z)-w$ and write $g(z)=(f(z)-w_0)+(w_0-w)=F(z)+G(z).$ If now $|w-w_0|<\varepsilon$ we have $|F(z)|>|G(z)|$ on the circle $|z-z_0|=\delta$ and thus by Rouché's Theorem we find that $g=F+G$ has a zero inside the circle since $F$ has one.
	\end{proof}

	\begin{lemma}\label{zeros_of_E_are_simple}
		For almost all $s\in \R$ and for all $C>4\pi$ such that the box $Q(s,C/t)$ contains a zero of $E(t,z),$ the zeros of $E(t,z)$ in $Q(s,C/t)$ are simple for all large enough $t.$
	\end{lemma}

	\begin{proof}
		This follows immediately from Lemma \ref{Lemma_6} combined with the fact that the zeros of $\sin(z)$ are simple and a double zero of $E(t,z)$ would imply a double zero of $\sin(z)$ by Lemma \ref{open_mapping_theorem}.
	\end{proof}
	
	Furthermore, by Lemma \ref{Lemma_6} we can find an approximation for $\theta(t,z)$ and $\theta_z(t,z(t)).$ We set $Q_{+}(s,C(t)/t)=Q(s,C(t)/t)\cap\C_+.$ Note that in the following Lemma $\overline{z}(t)$ will play the role of $z(t),$ since the conjugates of the zeros of $E(t,z)$ are the zeros of $\theta(t,z).$
	
	\begin{lemma}\label{formulas_mif}
		Let $s, C(t), \alpha(t), D$ and $z(t)$ be like in Lemma \ref{Lemma_6}.
		
		\begin{itemize}
		
		\item[(i)]\begin{equation}\label{approximation_mif}
		\sup_{z \in Q_{+}(s,C(t)/t)} \bigg|\theta(t,z)-\overline{\alpha}^2(t)\frac{\sin[t(z-\overline{z}(t))]}{\sin[t(z-z(t))]}\bigg|=o(1),
		\end{equation}
		
		\noindent
		as $t\rightarrow \infty, t\in T_1(s,C(t))=T_0(s,C(t))\cap\{ty(t)>1\}.$
		
		\item[(ii)]$$\alpha^2(t)\theta_z(t,\overline{z}(t))=(1+o(1))t\frac{1}{\sin[2ity(t)]},$$
		
		\item[(iii)]$$\alpha^2(t)\theta_{zz}(t,\overline{z}(t))=-(1+o(1))t^2\frac{2\cos[2ity(t)]}{\sin^2[2ity(t)]},$$
				
		\noindent
		as $t\rightarrow \infty, t\in T_1(s,D)=T_0(s,D)\cap\{ty(t)>1\}.$
	\end{itemize}
	\end{lemma}	

	\begin{proof}
		Let $z\in Q_+(s,C(t)/t).$ We calculate
		
		$$\bigg|\theta(t,z)-\overline{\alpha}^2(t)\frac{\sin[t(z-\overline{z}(t))]}{\sin[t(z-z(t))]}\bigg|\leq \bigg|\frac{E^{\sharp}(t,z)-\frac{\overline{\alpha}(t)\gamma(ty(t))}{\sqrt{w(s)}} \sin[t(z-\overline{z}(t))]}{E(t,z)}\bigg|$$
		$$+\bigg|\frac{\frac{\overline{\alpha}(t)\gamma(ty(t))}{\sqrt{w(s)}} \sin[t(z-\overline{z}(t))]}{\frac{\overline{\alpha}(t)\gamma(ty(t))}{\sqrt{w(s)}} \sin[t(z-z(t))]E(t,z)}\bigg|\cdot \bigg|E(t,z)-\frac{\overline{\alpha}(t)\gamma(ty(t))}{\sqrt{w(s)}} \sin[t(z-z(t))] \bigg|.$$
		
		\noindent
		Now by \eqref{approximation_by_sine} there is a $T>0$ such that for all $t>T$ and $z\in Q_+(s,C(t)/t)$
		
		$$|E(t,z)|> \bigg|\frac{\frac{\overline{\alpha}(t)\gamma(ty(t))}{\sqrt{w(s)}} \sin[t(z-z(t))]}{2}\bigg|.$$
		
		\noindent
		Thus,
		
		$$\bigg|\theta(t,z)-\overline{\alpha}^2(t)\frac{\sin[t(z-\overline{z}(t))]}{\sin[t(z-z(t))]}\bigg|\leq \bigg|\frac{2}{\frac{\overline{\alpha}(t)\gamma(ty(t))}{\sqrt{w(s)}} \sin[t(z-z(t))]}\bigg|\bigg|E^{\sharp}(t,z)-\sin[t(z-\overline{z}(t))]\bigg|$$
		
		$$+2\bigg|\frac{\frac{\overline{\alpha}(t)\gamma(ty(t))}{\sqrt{w(s)}} \sin[t(z-\overline{z}(t))]}{\bigg(\frac{\overline{\alpha}(t)\gamma(ty(t))}{\sqrt{w(s)}} \sin[t(z-z(t))]\bigg)^2}\bigg|\bigg|E(t,z)-\sin[t(z-z(t))]\bigg|.$$
		
		\noindent
		Since $ty(t)>1$ there exists a constant $D>0$ such that 
		
		$$\bigg|\theta(t,z)-\overline{\alpha}^2(t)\frac{\sin[t(z-\overline{z}(t))]}{\sin[t(z-z(t))]}\bigg|\leq D\bigg[\big|E^{\sharp}(t,z)-\sin[t(z-\overline{z}(t))]\big|+ \big|E(t,z)-\sin[t(z-z(t))]\big|\bigg].$$
		
		\noindent
		Now again by Lemma \ref{Lemma_6}, the last term tends to zero as $t\rightarrow \infty$ uniformly over all $z\in Q_+(s,C(t)/t).$ This proves $(i).$
		
		The functions in \eqref{approximation_mif} are holomorphic on $Q_+(s,C(t)/t).$ Suppose that $C(t)>2D.$ Hence, by the Cauchy estimates for the derivative we find on the smaller box $Q_+(s,D/t)$
		
		$$\sup_{z\in Q_+(s,D/t)} \bigg|\theta_z(t,z)-\frac{d}{dz}\overline{\alpha}^2(t)\frac{\sin[t(z-\overline{z}(t))]}{\sin[t(z-z(t))]}\bigg|=o(t).$$
		
		On the other hand we compute
		
		$$\frac{d}{dz} \frac{\sin[t(z-\overline{z}(t))]}{\sin[t(z-z(t))]}\bigg|_{z=\overline{z}(t)}=\frac{t}{\sin[2ity(t)]}.$$
		
		Combining both formulas we obtain as desired
		
		$$\alpha^2(t)\theta_z(t,\overline{z}(t))=(1+o(1))t\frac{1}{\sin[2ity(t)]}.$$
		
		The proof for $(iii)$ is immediate by repeating the argument from the proof of $(ii).$
	\end{proof}

	The function $\theta_z(t,\overline{z}(t))$ recognizes the direction the zero moves in. By our previous approximations, we will be able to carry over this information to $\alpha(t).$
	
	\begin{lemma}\label{Lemma_7}
		Let $s, C(t), \alpha(t), D$ and $z(t)$ be like in Lemma \ref{Lemma_6}. Then there exists $T>0$ such that for all intervals $[t_1,t_2]\subset T_1(s,D)\cap[T,\infty)$ we can change $\alpha(t)$ slightly so that it satisfies Lemma \ref{Lemma_6} and we have the following relation
		
		\begin{equation}\label{angle}
			\frac{\pi}{2}+2\arg \alpha(s,t)=-\arg\theta_z(t,\overline{z}(t)).
		\end{equation}
		
		Furthermore, there is a real function $\psi$ on $[t_1,t_2]$ satisfying 
		
		$$\alpha(s,t)=e^{i[st+\psi(t)]}$$
		
		\noindent
		and for all $t_1\leq x_1\leq x_2 \leq t_2$ 
		
		\begin{equation}\label{misc14}
		|\psi(x_2)-\psi(x_1)|\leq \int_{x_1}^{x_2}\bigg(3\cosh(2D) \cdot |f(t)| + \frac{D}{t}\bigg) dt.
		\end{equation}
	\end{lemma}

	\begin{proof}
		By Lemma \ref{formulas_mif} $(ii),$ 
		
		$$\alpha^2(s,t)\theta_z(t,\overline{z}(t))=(1+o(1))t\frac{1}{\sin[2ity(t)]}.$$
		
		\noindent
		After rearranging the terms, using $\sin[2ity(t)]=i\sinh[2ty(t)]$ and using the fact that $1<ty(t)<D$ we arrive at
		
		$$\frac{\sinh[2ty(t)]}{t}\alpha^2(s,t)\theta_z(t,\overline{z}(t))-1=o(1).$$
		
		\noindent
		This immediately gives us
		
		$$\arg i \alpha^2(s,t)\theta_z(t,\overline{z}(t))=o(1).$$
		
		After changing $\alpha$ slightly we can eliminate the error term. Using the functional equation of the argument we arrive at \eqref{angle}
		
		$$ \frac{\pi}{2}+2\arg \alpha(s,t)=-\arg \theta_z(t,\overline{z}(t)).$$
		
		We can thus recover $\alpha(s,t)$ by studying 
		
		$$\frac{d}{dt} \arg \theta_z(t,\overline{z}(t)).$$
		
		For that, we will use orthogonal projections. In the following, we will identify $\C\cong\R^2$ and denote for two vectors $u,v\in \R^2, v\neq 0$ the orthogonal projection of $u$ onto $v$ by
		
		$$ \proj_v u = \frac{\langle u, v \rangle}{||v||^2}v.$$
		
		We compute for all $z$ such that $\theta_z(t,z)\neq 0$
		
		$$\proj_{i\theta_z(t,z)}\frac{d}{dt}\theta_z(t,z)=\proj_{i\theta_z(t,z)} |\theta_z(t,z)|e^{i\arg\theta_z(t,z)}$$
		
		$$=\proj_{i\theta_z(t,z)}\bigg(\frac{d}{dt} |\theta_z(t,z)|e^{i\arg\theta_z(t,z)} + i\frac{d}{dt}\arg\theta_z(t,z)|\theta_z(t,z)|e^{i\arg\theta_z(t,z)} \bigg)=\frac{d}{dt}\arg\theta_z(t,z)i\theta_{z}(t,z).$$
		
		\noindent
		Recall that since the zeros $E(t,z)$ are simple, $\theta_z(t,\overline{z}(t))\neq 0.$ Combining the above formula with \eqref{riccati} gives
		
		$$\frac{d}{dt}\arg\theta_z(t,\overline{z}(t))i\theta_z(t,\overline{z}(t))=\proj_{i\theta_z(t,\overline{z}(t))}\frac{d}{dt}\theta_z(t,\overline{z}(t))$$
		
		$$=\proj_{i\theta_z(t,\overline{z}(t))}2iz(t)\theta_z(t,\overline{z}(t))-f(t)\frac{\theta_{zz}(t,\overline{z}(t))}{\theta_z(t,\overline{z}(t))}$$
		
		$$=i \Re \overline{z}(t)\theta_z(t,\overline{z}(t))-\proj_{i\theta_z(t,\overline{z}(t))} f(t) \frac{\theta_{zz}(t,\overline{z}(t))}{\theta_z(t,\overline{z}(t))}$$
		
		$$=(s+\Re(\overline{z}(t)-s)-f(t)A(t))i\theta_z(t,\overline{z}(t)),$$
		
		\noindent
		for some real valued function $A(t).$ Using the definition of $\proj_v u$ together with Lemma \ref{formulas_mif} $(ii)$ and $(iii)$ we find for large enough $t$
		
		$$|A(t)|\leq\frac{\bigg|\frac{\theta_{zz}(t,\overline{z}(t))}{\theta_z(t,\overline{z}(t))}\bigg|}{|\theta_z(t,\overline{z}(t))|}\leq 2\cosh[2ty(t)]+o(1)\leq 3\cosh(2D).$$
		
		\noindent
		Moreover, since $\overline z(t) \in Q(s,D/t),$ we find $|\Re \overline{z}(t)-s|\leq D/t.$ We now define for $t_1\leq x \leq t_2$  $$\psi(x)=\int_{t_1}^{x}\Re(\overline{z}(t)-s)-f(t)A(t).$$
		
		\noindent
		Hence, we have for all $t_1\leq x_1\leq x_2 \leq t_2$
		
		$$|\psi(x_2)-\psi(x_1)|\leq \frac{1}{2}\int_{x_1}^{x_2} \bigg|\frac{d}{dt} \arg \theta_z(t,\overline{z}(t)) \bigg| dt \leq \int_{x_1}^{x_2} \bigg(3\cosh(2D) \cdot |f(t)| + D/t\bigg) dt.$$
		
		Finally, by using \eqref{angle} we find
		
		$$\alpha(s,t)=e^{i[st+\psi(t)]}.$$
	\end{proof}

	Hence by virtue of \eqref{angle}, $\alpha(t)$ carries information on the movement of the zeros. If $|\Im \alpha(t)|$ is small, the zero moves almost vertically and if $|\Im \alpha(t)|$ close to one, then the zero moves almost horizontally. We now have the tools at hand to continue our proof that $T_0(s,C)$ is bounded for almost all $s \in \R$ and for all $C>0.$
	
	For the remainder of this section let $C>1$ be a constant. We define

	$$S=\{s \in \R \mid T_0(s,C) \text{ is unbounded}\}.$$ 
	
	We will assume without loss of generality that $|S|<\infty.$ By Corollary \ref{corollary_3} the set $T_0(s,3C)$ can not cover a half-line for almost all $s\in \R.$ After changing $S$ by a set of measure zero, we can assume that for all $s \in S$ this is the case. Thus, the zeros will have to travel infinitely often from outside of $Q(s,3C/t)$ into $Q(s,C/t).$ We will show that every time this happens, this will require a chunk of the $L^2-$norm of $f.$ More precisely, for all $s\in S$ there will exist infinitely many intervals $\mathcal{L}_s(T_1, T_2)$ with the following properties arbitrarily far in time.
	
	\begin{itemize}
		\item [(i)] There exists a zero $z(s,t)$ of $E(t,z)$ in $Q(s,3C/t)$ for all $t\in \mathcal{L}_s; t \mapsto z(s,t)$ is a differentiable curve for $t \in \mathcal{L}_s$ and $\Im z(s,t) < -1/t$ for $t \in \mathcal{L}_s,$
		
		\item[(ii)] $z(s,T_1) \in \partial Q(s,3C/T_1),$
		
		\item [(iii)] If $2^n \leq T_1 < 2^{n+1},$ then $T_1<T_2\leq 2^{n+2}$ is such that $z(s,T_2) \in \partial Q(s,C/T_2).$
		
		\item[(iv)] For all $t\in (T_1, T_2)$ moves continuously inside $Q(s,3C/t)\setminus Q(s,C/t).$
		
	\end{itemize}	
		Indeed, we can get $(iii)$ by looking at the proof of Corollary\ref{corollary_3}. That we can achieve $(i)$ follows immediately by Lemma \ref{open_mapping_theorem}, since we can pick $z(t)$ as the first zero to reenter the box $Q(s,C/t)$ at $t=t_2.$ If there now was a point $t\in (t_1,t_2)$ such that $\Im z(s,t) \geq -\frac{1}{t},$ then immediately by Lemma \ref{Lemma_6} and Lemma \ref{open_mapping_theorem} we see that there were already other zeros in the box $Q(s,C/t)$ at time $t,$ contradicting our choice of $z(t).$ For each $s\in S$ we now have infinitely many intervals $\mathcal{L}_s=(T_1,T_2)$ with the above properties. If $2^n\leq T_1 \leq 2^{n+1},$ we will call that interval $\mathcal{L}_s^n.$ If there is more than one interval satisfying $(i)-(iv)$ and $2^n\leq T_1 < 2^{n+1},$ we will choose one of them. We will denote by $S^n$ the set of all $s\in S,$ such that there exists an interval $\mathcal{L}_s^n$ with the above properties. Thus, for any $k\geq 1,$
		
		$$S\subset \bigcup_{n=k}^{\infty} S^n.$$
		
		\begin{lemma}
			Let $s\in S$ and let $\mathcal{L}_s^n=(T_1,T_2)$ be one of the considered intervals. There exists a constant $\Delta>0,$ depending only on $C,$ such that 
			
			$$\int_{T_1}^{T_2} |f(t)|dt>\Delta.$$
		\end{lemma}
		
		\begin{proof}
			By Lusin's Theorem, we can assume that all error terms in Lemma \ref{formulas_mif} are uniform over all $s \in S.$ Then, by Lemma \ref{formulas_mif} $(ii)$ and property $(i)$ of $\mathcal{L}_s^n$ we have $|\theta_z(t,z(s,t))| \asymp t.$ Furthermore, by properties $(ii)$ and $(iii)$ of $\mathcal{L}_s^n$ we get $1/t \ll |z(s,T_1)-z(s,T_2)|.$ Both implicit constants depend only on $C.$ Thus, by \eqref{derivative_zero}
			
			$$\frac{1}{T_2}\leq \frac{1}{t}\ll|z(s,T_1)-z(s,T_2)|\leq \int_{T_1}^{T_2} \frac{|f(t)|}{|\theta_z(t,\overline{z}(s,t))|} dt \ll \int_{T_1}^{T_2} \frac{|f(t)|}{t} dt\ll \frac{1}{T_1}\int_{T_1}^{T_2}|f(t)|dt.$$
		\end{proof}
	
		By making $\Delta>0$ smaller if necessary, we will from now on assume that 
		
		\begin{equation}\label{misc17}
		0<\Delta< \frac{1}{100\cosh 6C}.
		\end{equation}
		
		We will also need upper bounds. The following Lemma tells us that this is possible by passing to a subinterval $L_s^n\subset \mathcal{L}_s^n.$
		
		\begin{lemma}\label{upper_bounds}
			Let $s\in S^n$ and $\mathcal{L}_s^n$ the associated interval satisfying properties $(i)-(iv)$ from above. Then, there exists a subinterval $L_s^n=(\tau_1,\tau_2)\subset\mathcal{L}_s^n$ satisfying properties $(i)$ and $(iv)$ from above and we have the following upper bounds by making $\Delta$ a bit smaller if necessary
			
			 $$0<\Delta<\int_{\tau_1}^{\tau_2} |f(t)|dt< 2\Delta,$$
			 
			 $$\int_{\tau_1}^{\tau_2} \frac{3C}{t}dt < \frac{1}{100}.$$
		\end{lemma}
	
		\begin{proof}
			Since $T_2-T_1\leq 3\cdot 2^n,$ we find 
			
			$$\int_{T_1}^{T_2} \frac{3C}{t}dt\leq \int_{T_1}^{T_2} \frac{3C}{T_1}dt \leq \frac{3\cdot 2^n \cdot 3C }{2^n}=9C.$$
			
			Dividing $\mathcal{L}_s^n$ into $k$ subintervals of equal size, where $k>900C,$ gives by the pigeon hole principle a subinterval $I=(\tau_1,\tau_2),$ where 
			
			$$0<\frac{\Delta}{k}<\int_{\tau_1}^{\tau_2} |f(t)|dt,$$
			
			$$\int_{\tau_1}^{\tau_2} \frac{3C}{t} dt < \frac{3\cdot 2^n \cdot 3C}{k\cdot2^n}<\frac{1}{100}.$$
			
			\noindent
			Redefining $\Delta$ as $\Delta/k$ and decreasing $\tau_2,$ we can achieve further
			
			$$0<\Delta<\int_{\tau_1}^{\tau_2} |f(t)|dt < 2\Delta.$$
			
			\noindent
			The desired interval is then $L_s^n=(\tau_1,\tau_2).$
			
		\end{proof}
		
		In the following, we will only consider the intervals $L_s^n.$	We will have to further subdivide $S^n$ into two parts, depending on whether the zero moves almost vertically or has a horizontal component, comparable to the total increment. Each case will require a different argument. To rule out a criterion for when the zero moves horizontally or vertically, we will use Lemma \ref{Lemma_7}. Let $\alpha(s,t)$ be the function from Lemma \ref{Lemma_7}. By \eqref{misc14}, $\alpha(s,t)$ is continuous. For fixed $s\in S$ let $I_s\subset T_0(s,3C).$ We say that 
		
		\begin{itemize}
			\item [(i)] $I_s$ is a $V_s-$interval if $|\Im \alpha^2(s,t)| < \frac{1}{100}$ for all $t\in I_s,$
			
			\item [(ii)] $I_s$ is a $H_s-$interval if $|\Im \alpha^2(s,t)| \geq \frac{1}{200}$ for all $t\in I_s.$ 
		\end{itemize}
		
		Note that an interval can be both an $V_s-$and $H_s-$interval. By continuity of $\alpha(s,t)$ we can partition $L_s^n,$ up to countably many points, into disjoint intervals, which are either $V_s-$ or $H_s-$ intervals or both.
		
		We can now split up the set $S^n$ into disjoint subsets
		
		$$S_V^n=\bigg\{s \in S^n \bigg| \int_{\bigcup_{I_s \in VL_s^n} I_s} |f(t)|dt \geq \frac{99}{100} \int_{L_s^n} |f(t)|dt\bigg\},$$
		
		$$S_H^n=\bigg\{s \in S^n \bigg| \int_{\bigcup_{I_s \in VL_s^n}I_s} |f(t)|dt < \frac{1}{100} \int_{L_s^n} |f(t)|dt\bigg\}.$$
		
		Note that since $L_s^n$ is partitioned, up to countably many points, into $V_s-$ and $H_s-$intervals, we find for all $s\in S_H^n$ 
		
		$$ \int_{\bigcup_{I_s \in HL_s^n}I_s} |f(t)| > \frac{1}{100} \int_{L_s^n} |f(t)|dt.$$
		
		Hence, we now have for all $k\geq 1$
		
		$$S\subset \bigcup_{n=k}^{\infty} S^n \subset \bigcup_{n=k}^{\infty} S_V^n \cup \bigcup_{n=k}^{\infty} S_H^n.$$
		
		In the next two sections we are going to prove that there exists a constant $D>0$ such that for all $n\geq 1$ 
		
		$$ D|S_V^n| < ||f||_{L^2([2^n,2^{n+2}])}^2, \quad D|S_H^n| < ||f||_{L^2([2^n,2^{n+2}])}^2.$$
		
		From this, we find quickly that $|S|>0$ contradicts $f\in L^2(\R),$ compare Theorem \ref{measure4}.
	
	\subsection{Vertical intervals}
	In this section, we are going to show that there exists a constant $D>0$ such that 
	
	$$ D|S_V^n| < ||f||_{L^2([2^n,2^{n+2}])}^2.$$

	Consider $s\in S_V^n \subset S^n.$ Then there exists an interval $(\tau_1, \tau_2)$ with $2^n\leq \tau_1 < 2^{n+1}.$ We will split $L_s^n$ into three disjoint intervals 
	
	$$L_s^n =T_{-}^n\cup T_s^n\cup T_{+}^n$$
	
	\noindent
	such that $T_{-}^n$ is to the left of $T_s^n$ and $T_{+}^n$ is to the right of $T_s^n$ and

	$$\int_{T_{-}^n} |f(t)|dt>\frac{\Delta}{3}, \quad \int_{T_s^n} |f(t)|dt=\frac{\Delta}{3}, \quad \int_{T_{+}^n} |f(t)|dt>\frac{\Delta}{3}.$$
	
	\noindent
	Furthermore, we choose $T_s^n$ to be open. We split $f$ into its positive and negative part $f=f_+-f_-.$ We have 
	
	$$\int_{T_s^n}f_+(t)dt \geq \frac{\Delta}{6}, \quad \text{ or }  \int_{T_s^n} f_-(t)dt\geq\frac{\Delta}{6},$$
	
	\noindent
	and we will assume that the inequality with $f_+$ holds for all $T_s^n.$ In the following we consider the collection
	
	$$W=\bigcup_{s\in S_V^n} \{T_s^n\}.$$
	
	Our next result tells us that we can pick a finite subcollection and proceed with that subcollection.
	
	\begin{lemma}
		There exists a finite subcollection 
		
		$$\mathcal{T}=\bigcup_{k=1}^{N} \{T_{s_k}^n\}$$
		
		\noindent
		with the following properties (abusing slightly the notation by identifying the collection with a union of intervals)
		
		\begin{itemize}
			\item [(i)] $$\int_{W\setminus \mathcal{T}} |f(t)|dt<\frac{\Delta}{100},$$
			\item[(ii)] Each $T_{s_k}^n \in \mathcal{T}$ intersects at most two other intervals from $\mathcal{T}$ and each point in $\mathcal{T}$ is contained at most in two intervals from $\mathcal{T}.$
		\end{itemize}
	\end{lemma}

	\begin{proof}
		First, we show that there exists a countable subcollection $W'$ of $W$ that already covers $W.$ Indeed enumerate by $\N$ all open balls $B_r(x),$ where $x,r\in \Q.$ For any $B_r(x)$ we choose if possible an interval $\mathcal{T}$ from $W$ such that $B_i \subset \mathcal{T}$ and add this interval to the collection $W'.$ By construction $W'$ is countable since its indexed by the balls $B_r(x),$ which are countably many. We now show that $W'$ covers $W.$ Indeed, assume that $x \in W.$ Then, there exists an interval $T_s^n$ with $x \in T_s^n.$ Since $T_s^n$ is open there is a ball $B_r(x)\subset T_s^n.$ But then by definiton of $W'$ there is an interval $\mathcal{T}$ (possibly different from $T_s^n$) with $x \in \mathcal{T}$ and $ \mathcal{T} \in W'.$
		
		Enumerate $W'=\{\mathcal{T}^1, \mathcal{T}^2, \dots\}$ by $\N.$ We now choose an interval $\mathcal{T}^1$ from $W'$ such that 
		
		$$\int_{\mathcal{T}^1} |f(t)| dt > \frac{1}{2} \sup_{n \in N} \int_{\mathcal{T}^n} |f(t)| dt, \quad \mathcal{I}^1=\mathcal{T}^1$$
		
		\noindent
		In the $n$-th step choose an interval $\mathcal{T}^n \in W'$ such that
		
		$$\int_{\mathcal{T}^n} |f(t)| dt > \frac{1}{2} \sup_{m \in \N} \int_{\big(W \setminus \bigcup_{k=1}^{n-1} \mathcal{T}^k \big) \cap \mathcal{T}^m} |f(t)| dt, \quad \mathcal{I}^n=\bigg(W \setminus \bigcup_{k=1}^{n-1} \mathcal{T}^k \bigg) \cap \mathcal{T}^n.$$
		
		\noindent
		By construction the sets $I^n$ are disjoint and cover $W'=W.$ Hence, we have 
		
		$$\int_{W} |f(t)| dt=\int_{W'} |f(t)| dt=\int_{\bigcup_{n=1}^\infty \mathcal{I}^n} |f(t)| dt= \sum_{n=1}^\infty \int_{\mathcal{I}^n} |f(t)| dt.$$
		
		\noindent
		Since the series converges there exists $N\in \N$ such that 
		
		$$\sum_{n=N+1}^{\infty} \int_{\mathcal{I}^n} |f(t)| dt < \frac{\Delta}{100}.$$
		
		\noindent
		Set $\mathcal{T}=\{\mathcal{T}^1,\dots \mathcal{T}^N\}.$ Since $\mathcal{I}^n \subset \mathcal{T}^n$ property $(i)$ holds. Indeed,
		
		$$\int_{W\setminus \mathcal{T}} |f(t)| dt = \int_{W \setminus \bigcup_{n=1}^N \mathcal{I}^n} |f(t)| dt = \sum_{n=N+1}^{\infty} \int_{\mathcal{I}^n} |f(t)| dt < \frac{\Delta}{100}.$$
		
		 We now need to reduce the collection. Note that if three intervals $I_1,I_2,I_3 \subset \R$ intersect in one point, then one of them is contained in the union of the others. We remove all those intervals which are contained in the union of two others intervals from the collection. But then at most two intervals intersect and each point in $\mathcal{T}$ is contained in at most two intervals from $\mathcal{T},$ which is precisely condition $(ii).$
	\end{proof}

	Then by property $(i)$,  for all intervals $T_s^n, s \in S_V^n,$ there is an interval $\mathcal{T}^m\in \mathcal{T}$ such that 
	
	\begin{equation}\label{misc15}
	\int_{T_s^n\cap \mathcal{T}^m} f_+(t)dt >0.
	\end{equation}
	
	\begin{lemma}
		Let $s\in S_V^n$ and $\mathcal{T}^m \in \mathcal{T}$ such that \eqref{misc15} holds. Then,
		
		$$\bigg|\int_{\mathcal{T}^m} f_+(t)e^{its}dt\bigg|>\frac{\Delta}{25}.$$
	\end{lemma}
	
	\begin{proof}
		Since $T_s^n \cap \mathcal{T}^m\neq \emptyset,$ we have $\mathcal{T}^m \subset L_s^n.$ By definition of $S_V^n$ and Lemma \ref{upper_bounds}
		
		$$\int_{\mathcal{T}^m\setminus \bigcup_{I_s \in VL_s^n} I_s} f_+(t)dt \leq \int_{L_s^n \setminus \bigcup_{I_s \in VL_s^n} I_s} |f(t)| dt = \int_{L_s^n} |f(t)|dt - \int_{\bigcup_{I_s \in VL_s^n}I_s} |f(t)|dt.$$
		
		$$\leq \int_{L_s^n} |f(t)| dt - \frac{99}{100} \int_{L_s^n} |f(t)|dt =\frac{1}{100} \int_{L_s^n} |f(t)|dt < \frac{\Delta}{50}$$
		
		\noindent
		Hence,
		
		\begin{equation}\label{misc16}
		\int_{\mathcal{T}^m\cap \bigcup_{I_s \in VL_s^n} I_s} f_+(t)dt= \int_{\mathcal{T}^m} f_+(t)dt-\int_{\mathcal{T}^m\setminus \bigcup_{I_s \in VL_s^n} I_s} f_+(t)dt\geq \frac{\Delta}{6}-\frac{\Delta}{50}>\frac{\Delta}{10}.
		\end{equation}
		
		The interval was partitioned, up to countably many points, into the collections of intervals $VL_s^n$ and $HL_s^n.$ By definiton an interval $I\in VL_s^n,$ if and only if $|\Im \alpha^2(s,t)|<\frac{1}{100}$ for all $t\in I.$ Since $|\alpha(s,t)|=1,$ this implies $|\Re \alpha(s,t)|\geq\frac{1}{2}.$ Thus,
		
		$$\int_{\mathcal{T}^m\cap \bigcup_{I_s \in VL_s^n} I_s} f_+(t)dt \leq 2 \bigg| \int_{\mathcal{T}^m\cap \bigcup_{I_s \in VL_s^n} I_s} f_+(t) \alpha(s,t) dt\bigg|.$$
		
		\noindent
		Since $\alpha(s,t)$ is like in Lemma \ref{Lemma_7} and using Lemma \ref{upper_bounds}, we find for any $\tau_1\leq t \leq \tau_2,$
		
		$$|\alpha(s,t)\overline{\alpha}(s,\tau_1)-e^{is(t-\tau_1)}|=|e^{i[\psi(t)-\psi(\tau_1)]}-1|=2\bigg| \sin\bigg(\frac{\psi(t)-\psi(\tau_1)}{2}\bigg)\bigg|\leq |\psi(t)-\psi(\tau_1)|$$
		
		$$\leq \int_{\tau_1}^{t} \bigg(3\cosh(6C) \cdot |f(t)| + D/t\bigg) dt \leq 6\cosh(6C)+\frac{1}{100} < \frac{1}{10}.$$
		
		\noindent
		Thus, we estimate further 
		
		$$2 \bigg| \int_{\mathcal{T}^m\cap \bigcup_{I_s \in VL_s^n} I_s} f_+(t) \alpha(s,t) dt\bigg|$$
		$$=2 \bigg| \int_{\mathcal{T}^m\cap \bigcup_{I_s \in VL_s^n} I_s} f_+(t)(\alpha(s,t)-\alpha(s,\tau_1)e^{is(t-\tau_1)} +\alpha(s,\tau_1)e^{is(t-\tau_1)}) dt\bigg|$$
		
		$$\leq 2\bigg| \int_{\mathcal{T}^m\cap \bigcup_{I_s \in VL_s^n} I_s} f_+(t)e^{ist} dt\bigg|+ 2\bigg| \int_{\mathcal{T}^m\cap \bigcup_{I_s \in VL_s^n} I_s} f_+(t)(\alpha(s,t)-\alpha(s,\tau_1)e^{is(t-\tau_1)})dt\bigg|$$
		
		$$< 2\bigg| \int_{\mathcal{T}^m\cap \bigcup_{I_s \in VL_s^n} I_s} f_+(t)e^{ist} dt\bigg|+ \frac{1}{5} \int_{\mathcal{T}^m\cap \bigcup_{I_s \in VL_s^n} I_s} f_+(t)dt.$$
		
		Combining this estimate with \eqref{misc16} gives
		
		$$\bigg|\int_{\mathcal{T}^m} f_+(t)e^{ist}dt\bigg|\geq \bigg| \int_{\mathcal{T}^m\cap \bigcup_{I_s \in VL_s^n} I_s} f_+(t)e^{ist} dt\bigg| > \frac{\Delta}{25}.$$
	\end{proof}

	Define $\mathcal{S}^m$ to be the set of all $s\in S_V^n$ such that 
	
	$$\int_{T_s^n\cap\mathcal{T}^m} f_+(t)dt \geq \frac{\Delta}{12}.$$
	
	\noindent
	Then, as we have shown above $\cup_{k=1}^{N}\mathcal{S}^k=S_V^n$ and by the previous Lemma we have for all $s \in \mathcal{S}^m,$
	
	$$\bigg|\int_{\mathcal{T}^m} f_+(t)e^{its}dt\bigg|>\frac{\Delta}{25}.$$
	
	Since $f\in L^2(\R)$ we have Plancherel's Theorem
	
	$$\int_{-\infty}^{\infty} |f(t)|^2dt=\int_{-\infty}^{\infty} |\hat{f}(t)|^2dt.$$
	
	\noindent
	This gives
	
	$$||f_+||_{L^2(\mathcal{T}^m)}^2=\int_{-\infty}^{\infty} |\widehat{f_+\chi_{\mathcal{T}^m}}(s)|^2ds=\int_{-\infty}^{\infty} \bigg|\int_{\mathcal{T}^m} f_+(t)e^{-2\pi ist} dt\bigg|^2 ds$$
	
	$$=\frac{1}{4\pi^2}\int_{-\infty}^{\infty} \bigg|\int_{\mathcal{T}^m} f_+(t)e^{ist} dt\bigg|^2 ds>\frac{1}{4\pi^2} \int_{S^m} \bigg|\int_{\mathcal{T}^m} f_+(t)e^{ist} dt\bigg|^2 ds > \frac{\Delta^2}{2500\pi^2} |S^m|.$$
	
	\noindent
	If we set $D=\frac{\Delta^2}{2500\pi^2}>0,$ we obtain
	
	$$\sum_{k=1}^{N} ||f_+||_{L^2(\mathcal{T}^m)}^2 \geq D\sum_{k=1}^{N}|S^m|\geq D|S_V^n|,$$
	
	 Thus, since each point in $\cup_{k=1}^{N} \mathcal{T}^k$ is covered by at most two intervals, we have as desired the following result.
	
	\begin{corollary}\label{claim_1}
		There exists a constant $D>0$ such that
		
		$$D|S_V^n|<||f||_{L^2([2^n,2^{n+2}])}^2.$$
	\end{corollary}

	\subsection{Horizontal intervals}
	In this section, we are going to prove that there exists a constant $D>0$ such that 
	
	$$D|S_H^n|<||f||_{L^2([2^n,2^{n+2}])}.$$
	
	In contrast to the previous section, we will now need to work with the scattering data and conclude with the non-linear Parseval identity. More precisely, for an interval $(t_1,t_2) \subset \R$ we associate the scattering function $a_{t_1\rightarrow t_2}$ constructed from the solution of the Dirac system \eqref{dirac_system} with function $f\chi_{(t_1,t_2)}.$
	
	We can construct $a_{t_1\rightarrow t_2}$ from the solution of \eqref{dirac_system} with potential function $f$ as follows. If $M(t,z)$ is the matrix from \eqref{matrix}, then we define the transfer matrix from $t=t_1$ to $t=t_2$ as 
	
	$$M_{t_1\rightarrow t_2} (z) = \begin{pmatrix}
	A_{t_1\rightarrow t_2}(z) && B_{t_1 \rightarrow t_2}(z) \\
	C_{t_1\rightarrow t_2}(z) && D_{t_1 \rightarrow t_2}(z)
	\end{pmatrix}=M(t_2,z)M^{-1}(t_1,z).$$
	
	The associated Hermite-Biehler functions $E_{t_1\rightarrow t_2}(z)$ and $\tilde{E}_{t_1\rightarrow t_2}(z)$ are then defined as
	
	$$E_{t_1\rightarrow t_2}=A_{t_1 \rightarrow t_2}-iC_{t_1 \rightarrow t_2}, \quad \tilde{E}_{t_1 \rightarrow t_2}=B_{t_1 \rightarrow t_2} -i D_{t_1 \rightarrow t_2}.$$
	
	Hence, the scattering function $a_{t_1 \rightarrow t_2}$ can be written as
	
	$$a_{t_1 \rightarrow t_2}(z)=\frac{1}{2}e^{i(t_2-t_1)z}(E_{t_1\rightarrow t_2}(z)+i\tilde{E}_{t_1 \rightarrow t_2}(z)).$$
	
	In this setting the non-linear Parseval identity \eqref{non_linear_parseval} becomes
	
	\begin{equation}\label{non_linear_parseval_2}
	||\log|a_{t_1\rightarrow t_2}|||_{L^1(\R)}=||f||_{L^2((t_1,t_2))}^2.
	\end{equation}
	
	Therefore, the strategy will be to derive lower bounds on $$||\log|a_{t_1\rightarrow t_2}|||_{L^1(\R)}$$ for special intervals $(t_1,t_2) \subset \R.$ This program will require a stronger approximation of the functions $E(t,z)$ and $\tilde{E}(t,z),$ than the approximations, we have used so far. In particular, we wish to approximate both of the functions simultaneously in the imaginary component and show that the zeros of $E(t,z)$ and $\tilde{E}(t,z)$ move similar to the zeros of the approximating functions. We are going to formulate the precise result below and postpone the proof to the next section.
	
	First, we will need the notion of $\sigma-$interval for $f$ which is an interval $I\subset \R$ with the property that
	
	$$\bigg| \int_I f(t)dt \bigg| \geq (1-\sigma) \int_I |f(t)| dt,$$
	
	\noindent
	where $0<\sigma<1$ is a small number. On a $\sigma-$interval, the function $f$ has almost the same sign. We fix once and for all a number $0<\sigma<\frac{1}{100},$ whose exact value does not matter.
	
	\begin{lemma}\label{sigma-intervals}
		Let $f\in L_{loc}(\R).$ Then, for almost all $s\in \R$ with $f(s)\neq 0$ all sufficiently small intervals $I,$ which are centered around $s,$ are $\sigma-$intervals.
	\end{lemma}
		
	\begin{proof}
		Let $I_n$ be any sequence of intervals $I$ centered at $s$ with $|I|=2^n.$ Almost all $s \in \R$ are Lebesgue-points of $f$ and $|f|.$ Thus, we can assume without loss of generality that $s$ is indeed a Lebesgue-point for both functions. Hence,
		
		$$\lim_{n\rightarrow \infty} \frac{1}{|I_n|} \bigg| \int_{I_n} f(t)dt \bigg| \rightarrow |f(s)|=\lim_{n\rightarrow \infty} \frac{1}{|I_n|}\int_{I_n} |f(t)|dt.$$
		
		\noindent
		Then, for large enough $n$ we find
		
		$$\frac{1}{|I_n|} \bigg| \int_{I_n} f(t)dt \bigg| \geq \bigg(1+\frac{\sigma}{2}\bigg)(1-\sigma) |f(s)| \geq \frac{(1+\frac{\sigma}{2})(1-\sigma)}{(1+\frac{\sigma}{2})} \frac{1}{|I_n|}\int_{I_n} |f(t)|dt.$$
	\end{proof}

	We are now ready to formulate the stronger approximation result.
	
	\begin{theorem}[Strong approximation]\label{Lemma_10}
		For almost all $s\in \R$ there exists a increasing functions $C(t)>0, C(t) \rightarrow \infty$ and a function $\psi(t)>0, \psi(t)\rightarrow 0$ as $t\rightarrow \infty$ such that the following holds.
		
		Let $(t_1,t_2) \subset T_1(s,C(t)), t_2-t_1 \leq \frac{1}{|s|+1}$ be a $\sigma-$interval such that 
		
		$$\int_{t_1}^{t_2} |f(t)|dt < \frac{1}{100\cosh[2A]},$$
		
		\noindent
		and $C(t)>10\pi$ for $t>t_1,$ where $A$ is a constant satisfying $10\pi< A < C(t)$ for all $t>t_1.$ Let $\xi_1$ a zero of $E(t_1,z)$ in $Q(s,A/t_1)$ which moves continuously inside $Q(s,A/t)$ to a zero $\xi_2$ of $E(t_2,z)$ as $t$ changes from $t_1$ to $t_2.$ Let $\tilde{\xi_1}, \tilde{\xi_2}$ be similar zeros of $\tilde{E}$ inside $Q(s,A/t).$ Since $(t_1,t_2)\subset T_1(s,C(t)/t),$ we have $t_k\Im \xi_k>2, t_k \Im \tilde{\xi_2}>2$ for $k=1,2.$
		
		Then, the zeros of $E(t,z)$ and $\tilde{E}(t,z)$ change in similar ways as $t$ changes from $t_1$ to $t_2,$ 
		
		$$|(\xi_2-\xi_1)-(\tilde{\xi_2}-\tilde{\xi_1})|\leq \psi(t_1)|\xi_2-\xi_1|,$$
		
		\noindent
		and there exist real continuous functions $y_{t}, x_{t}, \tilde{x}_{t}$ and unimodular continuous functions $\alpha_{t}$ on the interval $[t_1,t_2]$ such that $ty_{t}>1$ for all $t\in [t_1,t_2]$ and
		
		\begin{itemize}
			\item [(i)] At $t=t_1$ we have
			
			$$\sup_{z \in Q(s,3C(t_1)/t_1)} \bigg|E(t_1,z)-\frac{\alpha_{t_1}\gamma(t_1y_{t_1})}{\sqrt{w(s)}} \sin[t_1(z-(x_{t_1}-iy_{t_1}))]\bigg|<\psi(t_1),$$
			
			$$\sup_{z \in Q(s,3C(t_1)/t_1)} \bigg|\tilde{E}(t_1,z)-\frac{\alpha_{t_1}\gamma(t_1y_{t_1})}{\sqrt{\tilde{w}(s)}} \cos[t_1(z-(\tilde{x}_{t_1}-iy_{t_1}))]\bigg|<\psi(t_1).$$
			
			\item[(ii)] For all $t \in [t_1,t_2]$ we have 
			
			$$ -\frac{\pi}{2}<t(\tilde{x}_{t}-x_{t})<\frac{\pi}{2}, \quad \cos[t_k(\tilde{x}_{t}-x_{t})]=\sqrt{w(s)\tilde{w}(s)}.$$
			
			\item[(iii)] For $k=1,2$ and large enough $t_1$ we have 
			
			$$E(t_k,s)=\frac{\alpha_{t_k}\gamma(t_ky_{t_k})}{\sqrt{w(s)}} \sin[t_k(s-(x_{t_k}-iy_{t_k}))],$$
			
			$$\tilde{E}(t_k,s)=\frac{\alpha_{t_k}\gamma(t_ky_{t_k})}{\sqrt{\tilde{w}(s)}} \cos[t_k(s-(\tilde{x}_{t_k}-iy_{t_k}))].$$
			
			\item[(iv)] As $t$ changes from $t_1$ to $t_2$ the zeros of the approximating functions change similarly to the zeros of $E(t,z)$ and $\tilde{E}(t,z),$
			
			$$ |[(x_{t_2}-iy_{t_2})-(x_{t_1}-iy_{t_1})]-[\xi_2-\xi_1]|\leq \psi(t_1)|\xi_2-\xi_1|,$$
			
			$$|[(\tilde{x}_{t_2}-iy_{t_2})-(\tilde{x}_{t_1}-iy_{t_1})]-[\tilde{\xi_2}-\tilde{\xi_1}]| \leq \psi(t_1)|\xi_2-\xi_1|.$$
		\end{itemize}
	\end{theorem}

	As already mentioned we will prove Theorem \ref{Lemma_10} in the next section. We are now going to use the stronger approximation to deduce asymptotics for $a_{t_1\rightarrow t_2}.$ For the remainder of this section we will assume that $s$ is such that Theroem \ref{Lemma_10} holds and we fix the notation for the time interval $(t_1,t_2),$ the size of the box $A,$ the zeros $\xi_1$ and $\tilde{\xi}_t$ of $E(t,z)$ and $\tilde{E}(t,z),$ respectively, as well as the approximating zeros $x_t-iy_t$ and $\tilde{x}_t-iy_t$.	We introduce two quantities $\varepsilon_1$ and $\varepsilon_2$ which measure the horizontal increment and vertical increment of the blown up approximating zeros as they travel from time $t_1$ to $t_2,$ respectively. More precisely, we set
	
	$$\varepsilon_1=t_2(x_{t_2}-s)-t_1(x_{t_1}-s), \quad \varepsilon_2=t_2y_2-t_1y_1.$$
	
	Directly from the definition one sees that $E(t,0)=1$ and $\tilde E(t,0)=-i$ for all $t>0.$ Hence, $a_{t_1\rightarrow t_2}(0)=1$ for all intervals $(t_1,t_2)$ and we denote by $\arg E$ and $\arg a_{t_1\rightarrow t_2}$ continuous branches of the argument on $\R$ and $\arg E(t,0)=\arg a_{t_1\rightarrow t_2}(0)=0.$ We now have the following asymptotics.
	
	\begin{lemma} 
		$$a_{t_1\rightarrow t_2}(s)=e^{i(t_2-t_1)s}\alpha_{t_2}\overline\alpha_{t_1}\bigg(1+i\varepsilon_1\coth[2t_1y_1]+O(\varepsilon_1^2+\varepsilon_2^2)\bigg),$$
		
		where the implicit constant depends only on $A.$
	\end{lemma}

	In the proof we are going to apply Theorem \ref{Lemma_10} to both functions $E(t,z)$ and $\tilde{E}(t,z)$ at the same time. This requires that both functions have a zero in the box $Q(s,A/t).$ In this section we supposed only that $E(t,z)$ has a zero, so Theorem\ref{Lemma_10} seems not applicable at a first glance. However, by Lemma \ref{Lemma_8}, which we will prove in the next section and is essential for the proof of Theorem \ref{Lemma_10}, $E(t,z)$ is close to $\sin z$ and $\tilde{E}(t,z)$ is close to $\cos z.$ Therefore, by Lemma \ref{open_mapping_theorem} if one of the functions has a zero in the box, the other one has a zero in the box $Q(s,(A+\varepsilon)/t)$ (we may have to consider a larger box, because the zero could be located at the boundary. If the zero is in the centre, we do not have to consider a bigger box). For larger $t$ we can choose $\varepsilon$ smaller. Hence, one should think of the zeros of $E(t,z)$ and $\tilde{E}(t,z)$ to be similar to the zeros of $\sin z$ and $\cos z,$ i.e. they arrange in a straight line the more $t$ increases. Since we consider a large box with $A> 8\pi$ it is really enough to look at times $t,$ where the function $E(t,z)$ has a zero in the box $Q(s,C/t)$

	\begin{proof}
		For $k=1,2$ we have by Theorem \ref{Lemma_10} for large enough $t_1$
		
		$$E(t_k,s)=\frac{\alpha_{t_k}\gamma(t_ky_{t_k})}{\sqrt{w(s)}} \sin[t_k(s-(x_{t_k}-iy_{t_k}))],$$
		
		$$\tilde{E}(t_k,s)=\frac{\alpha_{t_k}\gamma(t_ky_{t_k})}{\sqrt{\tilde{w}(s)}} \cos[t_k(s-(\tilde{x}_{t_k}-iy_{t_k}))].$$
		
		By definition of the transfer matrix $M_{t_1\rightarrow t_2}$ and Corollary \ref{determinant_2} we find
		
		$$M_{t_1\rightarrow t_2}(z)=M(t_2,1)M^{-1}(t_1,z)$$
		
		$$=\frac{1}{2} \begin{pmatrix}
		1 && 1 \\ i && -i
		\end{pmatrix} \begin{pmatrix}
		E(t_2,z) && \tilde{E}(t_2,z) \\ E^{\sharp}(t_2,z) && \tilde{E}^{\sharp}(t_2,z)
		\end{pmatrix} \begin{pmatrix}
		E(t_1,z) && \tilde{E}(t_1,z) \\ E^{\sharp}(t_1,z) && \tilde{E}^{\sharp}(t_1,z)
		\end{pmatrix}^{-1}2\begin{pmatrix}
		1 && -i \\ 1 && i
		\end{pmatrix}^{-1}.$$
		
		\noindent
		Again by Corollary \ref{determinant_2}
		
		$$\begin{pmatrix}
		E_{t_1\rightarrow t_2} && \tilde{E}_{t_1 \rightarrow t_2} \\ E_{t_1 \rightarrow t_2}^{\sharp} && \tilde{E}_{t_1 \rightarrow t_2}^{\sharp}
		\end{pmatrix} = 2 \begin{pmatrix}
		1 && 1  \\ i && -i
		\end{pmatrix}^{-1}
		M_{t_1\rightarrow t_2}(z).$$
		
		$$=\begin{pmatrix}
		E(t_2,z) && \tilde{E}(t_2,z) \\ E^{\sharp}(t_2,z) && \tilde{E}^{\sharp}(t_2,z)
		\end{pmatrix} \frac{1}{2i}
		\begin{pmatrix}
		\tilde{E}^{\sharp}(t_1,z) && -\tilde{E}(t_1,z) \\
		-E^{\sharp}(t_1,z) && E(t_1,z)
		\end{pmatrix} \begin{pmatrix}
		1 && -i \\ 1 && i
		\end{pmatrix}.$$
		
		By definition of $a_{t_1 \rightarrow t_2}$ we obtain 
		
		$$a_{t_1 \rightarrow t_2}(s)=\frac{e^{i(t_2-t_1)z}}{2}(E_{t_1\rightarrow t_2}(z)+i\tilde{E}_{t_1 \rightarrow t_2}(z))$$
		$$=\frac{e^{i(t_2-t_1)z}}{4i}\bigg(\big[(E(t_2,z)\tilde{E}^{\sharp}(t_1,z)-\tilde{E}(t_2,z)E^{\sharp}(t_1,z))+(-E(t_2,z)\tilde{E}(t_1,z)+\tilde{E}(t_2,z)E(t_1,z))\big]$$
		$$+i\big[-i(E(t_2,z)\tilde{E}^{\sharp}(t_1,z)-\tilde{E}(t_2,z)E^{\sharp}(t_1,z))+i(-E(t_2,z)\tilde{E}(t_1,z)+\tilde{E}(t_2,z)E(t_1,z))\big]\bigg)$$
		$$=\frac{e^{i(t_2-t_1)z}}{2i}(E(t_2,z)\tilde{E}^{\sharp}(t_1,z)-\tilde{E}(t_2,z)E^{\sharp}(t_1,z)).$$
		
		We will simplify notation by putting $x_k=x_{t_k}-s, \tilde{x}_k=\tilde{x}_{t_k}-s$ and $y_k=y_{t_k}$ for $k=1,2.$ Thus, $t_k(s-(x_{t_k}-iy_{t_k}))=-t_k(x_k-iy_k)$ and $t_k(s-(\tilde x_{t_k}-iy_{t_k}))=-t_k(\tilde x_k-iy_k)$.
		
		Hence, by using Theorem \ref{Lemma_10}
		
		$$a_{t_1 \rightarrow t_2}(s)=-\frac{e^{i(t_2-t_1)s}}{2i}\bigg(\frac{\alpha_{t_2}\gamma(t_2y_2)}{\sqrt{w(s)}} \sin[t_2(x_2-iy_2)] \frac{\overline{\alpha}_{t_1}\gamma(t_1y_1)}{\sqrt{\tilde{w}(s)}} \cos[t_1(\tilde{x}_1+iy_1)]$$
		$$-\frac{\alpha_{t_2}\gamma(t_2y_2)}{\sqrt{\tilde{w}(s)}} \cos[t_2(\tilde{x}_2-iy_2)] \frac{\overline{\alpha}_{t_1}\gamma(t_1y_1)}{\sqrt{w(s)}} \sin[t_1(x_1+iy_1)] \bigg)$$
		
		\begin{equation}\label{misc18}
		=-\frac{e^{i(t_2-t_1)s}}{2i}\frac{\gamma(t_1y_1)\gamma(t_2y_2)}{\sqrt{w(s)\tilde{w}(s)}}\big| \sin[t_2(x_2-iy_2)]\cos[t_1(\tilde{x}_1+iy_1)]-\cos[t_2(\tilde{x}_2-iy_2)]\sin[t_1(x_1+iy_1)]\big|.
		\end{equation}
		
		After using that 
		
		$$\cos z = \frac{e^{iz}+e^{-iz}}{2}, \quad \sin z=\frac{e^{iz}-e^{-iz}}{2i}$$
		
		\noindent
		and rearranging the terms, the expression inside the bracket in \eqref{misc18} is equal to

		$$\frac{1}{2}\bigg(\sin[t_2(x_2-iy_2)+t_1(\tilde{x}_1+iy_1)]+\sin[t_2(x_2-iy_2)-t_1(\tilde{x}_1+iy_1)]$$
		$$-\sin[t_1(x_1+iy_1)+t_2(\tilde{x}_2-iy_2)]-\sin[t_1(x_1+iy_1)-t_2(\tilde{x}_2-iy_2)]\bigg)$$
		
		Recall that for any two complex numbers $z,w\in \C,$ 
		
		$$\sin z - \sin w = 2 \sin \bigg(\frac{z-w}{2}\bigg)\cos\bigg(\frac{z+w}{2}\bigg).$$
		
		\noindent
		Hence, the last expression can be written as
		
		$$=\cos\bigg[\frac{1}{2}\big([t_2(x_2-iy_2)+t_1(\tilde{x}_1+iy_1)]+[t_1(x_1+iy_1)+t_2(\tilde{x}_2-iy_2)]\big)\bigg]$$
		$$\times \sin\bigg[\frac{1}{2}\big([t_2(x_2-iy_2)+t_1(\tilde{x}_1+iy_1)]-[t_1(x_1+iy_1)+t_2(\tilde{x}_2-iy_2)]\big)\bigg]$$
		$$+\cos\bigg[\frac{1}{2}\big([t_2(x_2-iy_2)-t_1(\tilde{x}_1+iy_1)]+[t_1(x_1+iy_1)-t_2(\tilde{x}_2-iy_2)]\big)\bigg]$$
		$$\times \sin\bigg[\frac{1}{2}\big([t_2(x_2-iy_2)-t_1(\tilde{x}_1+iy_1)]-[t_1(x_1+iy_1)-t_2(\tilde{x}_2-iy_2)]\big)\bigg]$$
		$$=\cos\bigg[\frac{1}{2}(t_2(x_2+\tilde{x}_2)+t_1(x_1+\tilde{x}_1))-i(t_2y_2-t_1y_1)\bigg]\times \sin\bigg[\frac{1}{2}(t_2(x_2-\tilde{x}_2)-t_1(x_1-\tilde{x}_1))\bigg]$$
		$$+\cos\bigg[\frac{1}{2}(t_2(x_2-\tilde{x}_2)+t_1(x_1-\tilde{x}_1))\bigg]\times \sin\bigg[\frac{1}{2}(t_2(x_2+\tilde{x}_2)-t_1(x_1+\tilde{x}_1))-i(t_2y_2+t_1y_1)\bigg].$$
		
		By Theorem \ref{Lemma_10} (ii) we have $\cos[t(\tilde{x}_{t}-x_{t})]=\sqrt{w(s)\tilde w(s)}$ for $t\in[t_1,t_2]$ and moreover since $x_{t_1}$ moves continuously to $x_{t_2}$ we find $t_2(x_2-\tilde{x}_2)=t_1(x_1-\tilde{x}_1).$ Thus, the last expression simplifies further to
		
		$$\sqrt{w(s)\tilde{w}(s)}\sin\bigg[\frac{1}{2}(t_2(x_2+\tilde{x}_2)-t_1(x_1+\tilde{x}_1))-i(t_2y_2+t_1y_1)\bigg].$$
		\noindent
		Plugging this back into \eqref{misc18} and using $\sin(ix)=i\sinh(x)$ we obtain
		
		$$a_{t_1 \rightarrow t_2}(s)=-\frac{e^{i(t_2-t_1)s}\alpha_{t_2}\overline{\alpha}_{t_1}}{i\sqrt{\sin[2it_1y_1]\sin[2it_2y_2]}}\big(\sin[(t_2x_2-t_1x_1)-i(t_2y_2+t_1y_1)]\big).$$
		
		\noindent
		We now substitute $t_1x_1=u, t_2x_2=u+\varepsilon_1, t_1y_1=v, t_2y_2=v+\varepsilon_2.$ Our last equation can then be written as 
		
		$$a_{t_1\rightarrow t_2}(s)=-\frac{e^{i(t_2-t_1)s}\alpha_{t_2}\overline\alpha_{t_1}}{i\sqrt{\sin[2iv]\sin[2i(v+\varepsilon_2)]}}\sin[\varepsilon_1-2iv\color{red}-\color{black}i\varepsilon_2].$$
		
		We are going to simply this further to
		
		$$a_{t_1\rightarrow t_2}(s)=e^{i(t_2-t_1)s}\alpha_{t_2}\overline\alpha_{t_1}\bigg[1+i\varepsilon_1\coth[2v]+O(\varepsilon_1^2+\varepsilon_2^2)\bigg],$$
		
		\noindent
		where the implicit constant depends only on $s$ and $A.$ Recall first that $v=t_1y_1$ and $t_2y_2$ satisfy the bounds $2<v<A,$ so the terms with $\sin[2iv]$ and $\sin[2i(v+\varepsilon_2)]$ are bounded away from zero. Further recall the identities
		
		\begin{equation}\label{misc41}
		\sin(ix)=i\sinh(x), \quad \cos(ix)=\cosh(x), \text{ for all } x \in \R,
		\end{equation}
		
		\begin{equation}\label{misc42}
		\sin(z \pm w)=\sin z \cos w \pm \cos z \sin w, \text{ for all } z,w \in \C,		
		\end{equation}
		
		\begin{equation}\label{misc43}
		\sinh(z \pm w)=\sinh z \cosh w \pm \sinh w \cosh z, \text{ for all } z,w \in \C,
		\end{equation}
		
		\begin{equation}\label{misc44}
		\cosh(z \pm w)= \cosh z \cosh w \pm \sinh z \sinh w, \text{ for all } z,w \in \C.
		\end{equation}
		
		We decompose using \eqref{misc42}
		
		$$a_{t_1 \rightarrow t_2}(s)=\frac{1}{i\sqrt{\sin[2iv]\sin[2i(v+\varepsilon_2)]}} \bigg[\sin \varepsilon_1 \cos(2iv+i\varepsilon_2)-\cos \varepsilon_1 \sin(2iv+i\varepsilon_2)\bigg].$$
		
		We are going to approximate both summands seperately. For the first one we decompose further using \eqref{misc41} and \eqref{misc44} to arrive at 
		
		$$\sin \varepsilon_1 \cosh(2v+\varepsilon_2)=\sin \varepsilon_1 \cosh 2v \cosh \varepsilon_2+ \sin \varepsilon_1 \sinh 2v \sinh \varepsilon_2.$$
		\noindent
		We thus have to approximate the following terms for the first summand.
		
		\begin{equation}\label{term1}
		\sqrt{\sinh 2v} \coth[2v] \sin \varepsilon_1 \frac{\cosh \varepsilon_2}{\sinh[2(v+\varepsilon_2)]},
		\end{equation}
		
		\begin{equation}\label{term2}
		\sqrt{\sinh 2v} \sin \varepsilon_1 \frac{\sinh \varepsilon_2}{\sqrt{\sinh[2(v+\varepsilon_2)]}}.
		\end{equation}
		\noindent
		By Taylor's Theorem we have that \eqref{term1} and \eqref{term2} are respectively
		
		$$\sqrt{\sinh 2v} \coth[2v] \sin \varepsilon_1 \frac{\cosh \varepsilon_2}{\sinh[2(v+\varepsilon_2)]}$$
		$$= \sqrt{\sinh 2v} \coth[2v](\varepsilon_1+O(\varepsilon_1^2))\bigg(\frac{1}{\sqrt{\sinh 2v}}+O(\varepsilon_2)\bigg)=\varepsilon_1\coth[2v]+O(\varepsilon_1^2+\varepsilon_2^2),$$
		
		$$\sqrt{\sinh 2v} \sin \varepsilon_1 \frac{\sinh \varepsilon_2}{\sqrt{\sinh[2(v+\varepsilon_2)]}}$$
		$$=\sqrt{\sinh 2v}(\varepsilon_1+O(\varepsilon_1^2))\bigg(\frac{\varepsilon_2}{\sqrt{\sinh 2v}}+O(\varepsilon_2^2)\bigg)=O(\varepsilon_1^2+\varepsilon_2^2).$$
		\noindent
		Thus, the first summand equals
		
		\begin{equation}\label{summand1}
		\varepsilon_1\coth[2v]+O(\varepsilon_1^2+\varepsilon_2^2).
		\end{equation}
		
		For the second summand we obtain using \eqref{misc41} and \eqref{misc43} 
		
		$$\cos \varepsilon_1 \sin(2iv+i\varepsilon_2)=i(\cos \varepsilon_1 \sinh 2v \cosh \varepsilon_2 + \cos \varepsilon_1 \sinh \varepsilon_2 \cosh 2v).$$
		
		\noindent
		The first term equals using Taylor's Theorem 
		
		$$\sqrt{\sinh 2v}\cos \varepsilon_1 \frac{\cosh \varepsilon_2}{\sqrt{\sinh[2(v+\varepsilon_2)]}}$$
		
		$$=\sqrt{\sinh 2v} (1+ O(\varepsilon_1^2))\bigg(\frac{1}{\sqrt{\sinh2v}}-\varepsilon_2 \frac{\cosh2v}{(\sinh2v)^{\frac{3}{2}}}+O(\varepsilon_2^2)\bigg)=1-\varepsilon_2\coth[2v]+O(\varepsilon_1^2+\varepsilon_2^2).$$
		
		\noindent
		By another application of Taylor's Theorem the second term equals
		
		$$\sqrt{\sinh 2v} \cos \varepsilon_1 \coth[2v] \frac{\sin \varepsilon_2}{\sqrt{\sinh[2(v+\varepsilon_2)]}}=$$
		$$\sqrt{\sinh 2v}\coth[2v] (1+O(\varepsilon_2^2))\bigg(\frac{\varepsilon_2}{\sqrt{\sinh 2v}}+O(\varepsilon_2^2)\bigg)=\varepsilon_2\coth[2v] + O(\varepsilon_1^2+\varepsilon_2^2).$$
		
		\noindent
		Hence, the second summand is
		
		\begin{equation}\label{summand2}
		i+O(\varepsilon_1^2+\varepsilon_2^2).
		\end{equation}
		
		Thus, by combining \eqref{summand1} and \eqref{summand2} we find as claimed
		
		$$a_{t_1\rightarrow t_2}(s)=-\frac{1}{i}e^{i(t_2-t_1)s}\alpha_{t_2}\overline\alpha_{t_1}\bigg[\varepsilon_1\coth[2v]-i+O(\varepsilon_1^2+\varepsilon_2^2)\bigg]$$
		$$=e^{i(t_2-t_1)s}\alpha_{t_2}\overline\alpha_{t_1}\bigg[1+i\varepsilon_1\coth[2v]+O(\varepsilon_1^2+\varepsilon_2^2)\bigg].$$
	\end{proof}

	One can show that on $H_s-$ and $\sigma-$ intervals we have the bound
	
	$$\int_{t_1}^{t_2} |f(t)|dt \ll \varepsilon_1.$$
	
	\noindent
	It was thus desirable that $\log a_{t_1\rightarrow t_2}$ was linear in $\varepsilon_1$ to otbain by linearity of the integral that 
	
	$$||\log a_{t_1 \rightarrow t_2}||_{L^1(\R)} \gg \varepsilon_1 \gg \int_{t_1}^{t_2} |f(t)|dt $$
	
	\noindent
	and deduce our result by a standard covering argument. However, this does not quite work as the asymptotics from the previous Lemma show. Indeed, because we have a non-trivial imaginary part the norm is giving us a quadratic contribution in $\varepsilon_2$ and thus we can not exploit linearity of the integral anymore. This was a mistake in older versions of \cite{Poltoratski} which was fixed by using the following idea. While it turns out that $|\log a_{t_1\rightarrow t_2}(s)|$ is quadratic in $\varepsilon_1$, we will show that $\arg a_{t_1 \rightarrow t_2}$ is barely enough linear on some subset in the base of the box, so we can show on the interval $(t_1,t_2)$ an estimate of the form
	
	$$|\arg a_{t_1\rightarrow t_2}(s)| \gg \int_{t_1}^{t_2} |f(t)|dt.$$

	We now note that $\arg a_{t_1 \rightarrow t_2}$ is a harmonic conjugate $\log|a_{t_1 \rightarrow t_2}|$ and thus may differ only by a constant from the Hilbert transform of $\log|a_{t_1 \rightarrow t_2}|.$ Since $\arg a_{t_1 \rightarrow t_2}(0)=0$ the constant is zero. Let $R$ denote the subset where the above inequality holds. Thus, we can estimate the weak $L^1-$norm of $\arg a_{t_1\rightarrow t_2}(s)$ from below by 
	
	$$||\arg a_{t_1\rightarrow t_2}(s)||_{L^{1,\infty}(\R)}>|R|\int_{t_1}^{t_2} |f(t)|dt.$$
	
	\noindent
	Since the Hilbert transform is of weak type $(1,1)$ we find that 
	
	\begin{equation}\label{misc55}
	||\log|a_{t_1 \rightarrow t_2}||_{L^1(\R)} \gg ||\arg a_{t_1\rightarrow t_2}(s)||_{L^{1,\infty}(\R)}>|R|\int_{t_1}^{t_2} |f(t)|dt.
	\end{equation}
	
	This saves the idea of our argument to exploit the linearity of the integral and the non-linear Parseval identity via a covering type argument. We are now going to make these ideas rigorous.

	\begin{lemma}\label{Lemma_11}
		Suppose that $t_ky_k> A/2$ for $k=1,2$ and that $(t_1,t_2)$ is a $H_s-$ and $\sigma-$interval. Then, on some subset $R$ of the interval $I_{t_2}=(s-\frac{10}{t_2},s+\frac{10}{t_2})$ of measure $|R|>D/t_2,$ we have the lower bound
		
		$$|\arg a_{t_1 \rightarrow t_2}| \geq D |\varepsilon_1| + o(|\varepsilon_1|+|\varepsilon_2|),$$
		
		\noindent
		where $D>0$ is an absolute constant depending only on $A.$
	\end{lemma}

	\begin{proof}
		By the last lemma we have the asymptotics
		
		$$a_{t_1\rightarrow t_2}(s)=e^{i(t_2-t_1)s}\alpha_{t_2}\overline\alpha_{t_1}\bigg(1+i\varepsilon_1\coth[2t_1y_1]+O(\varepsilon_1^2+\varepsilon_2^2)\bigg).$$ To estimate $\arg a_{t_1\rightarrow t_2}$ it is enough to work out the argument of each term seprerately by additivity. We have trivially that
		
		$$\arg e^{i(t_2-t_1)s}=(t_2-t_1)s.$$
		
		By Theorem \ref{Lemma_10} we have that 
		
		\begin{equation}\label{misc56}
		E(t_k,s)=\alpha_{t_k}\sin[t_k(iy_k-x_k)]=\alpha_{t_k}\sin[t_k(s-(x_{t_k}-iy_{t_k}))].
		\end{equation}
		
		\noindent
		We set $\phi(t,s)=\arg E(t,s).$ We know that $E(t,z)$ satisfies the differential equation 
		
		$$\frac{d}{dt} E(t,z)=-izE(t,z)+f(t)E^{\sharp}(t,z).$$
		
		\noindent
		Thus, we can use the same strategy as in the proof of Lemma \ref{Lemma_7} to calculate
		
		$$\phi(t,s)=-ts - \int_{0}^{t} f(t)\sin(2\phi(t,s)) dt.$$
		
		\noindent
		Hence, by \eqref{misc56}
		
		$$\arg \alpha_{t_k}= \phi(t_k,s)-\arg\sin[t_k(iy_k-x_k)]$$
		
		\noindent
		By direct computations one shows that
		
		$$\arg\sin[t_k(iy_k-x_k)]=\arg\sin[t_k(s-(x_{t_k}-iy_{t_k}))]=t_k(s-x_k)-\frac{\pi}{2}+\eta(t_k,s),$$
		
		\noindent
		where $\eta(t_k,\cdot)$ is a $\pi/t_k-$periodic function whose integral over each period vanishes. Since $(t_1,t_2)$ is an $H_s-$interval we have the bound following bound on $\R$
		for the difference function $\eta(s)=\eta(t_2,s)-\eta(t_1,s)$
		
		$$|\eta(s)|\ll |\varepsilon_1|+|\varepsilon_2| \ll |\varepsilon_1,$$
		
		\noindent
		where the implicit constant depends only on $A.$ Combining everything, we obtain
		
		$$\arg a_{t_1 \rightarrow t_2}(s)=(t_2-t_1)s+(\phi(t_2,s)-\phi(t_1,s))-\eta(s)-\varepsilon_1+\varepsilon_1\coth[2v]+O(\varepsilon_1^2+\varepsilon_2^2)$$
		$$=(\coth[2v]-1)\varepsilon_1-\int_{t_1}^{t_2} f(t) \sin(2\phi(t,s))dt-\eta(s) + O(\varepsilon_1^2+\varepsilon_2^2).$$ 
		
		We now need to make sure that the oscillating term does not cancel the linear term in $\varepsilon_1$ completely. Since by assumption $2v>A>0,$ we have $\coth[2v]-1>0.$ Suppose that $\varepsilon_1,\varepsilon_2$ and $\int_{t_1}^{t_2} f(t) dt$ are all positive. We can treat other cases similarly. Now as in the upcoming proof of Lemma \ref{claim_2}, using that $(t_1,t_2)$ is an $H_s-$interval and a $\sigma-$interval we can show that 
		
		\begin{equation}\label{misc58}
		\int_{t_1}^{t_2}|f(t)|dt \ll \int_{t_1}^{t_2} f(t)dt \ll \varepsilon_1.
		\end{equation}
		
		\noindent
		Together with the previously established bound $|\eta(s)|\ll \varepsilon_1$ we find that
		
		\begin{equation}\label{misc57}
		\bigg|\int_{t_1}^{t_2} f(t)\sin(2\phi(t,s))dt+\eta(s)\bigg| \ll \varepsilon_1.
		\end{equation}
		
		We note that Theorem \ref{Lemma_10} can be applied to any point $u$ in $I_{t_2}$ with the same zeros $\xi_k$ and $\tilde{\xi}_k$ of $E(t,z)$ and $\tilde{E}(t,z),$ respectively. We denote by $x_{t_k}^u, \tilde{x}_{t_k}^u$ and $y_{t_k}^u$ the corresponding approximating zeros from the Theorem at the point $u.$ Thus, by the very same calculations with the point $u$ instead of $s$ we find again that 
		
		$$\arg a_{t_1 \rightarrow t_2}(u)=(\coth[2v]-1)\varepsilon_1^u-\int_{t_1}^{t_2} f(t) \sin(2\phi(t,u))dt-\eta(u) + O((\varepsilon_1^u)^2+(\varepsilon_2^u)^2).$$ 
		
		Our goal is now to show that the left hand side of \eqref{misc57} is less than $(\coth[2v]-1)\varepsilon_1^u/2$ on a large subset of $I_{t_2}(s).$ Due to Corollary $\ref{corollary_5}$ (i) we have for each fixed $t\in[t_1,t_2]$ that 
		
		$$\sin(2\phi(t,u))=\sin(2\arg E(t,u))$$
		
		\noindent
		differs from a $\pi/t_k-$periodic function of $u$ with vanishing integral over each period by at most $o(1).$ Recall that the integral of $\eta(t_k,\cdot)$ vanishes as well over each period. Hence,
		
		$$\bigg|\int_{s-\frac{\pi}{t_2}}^{s+\frac{\pi}{t_2}}\int_{x-\frac{\pi}{t_1}}^{x+\frac{\pi}{t_1}}\bigg[\int_{t_1}^{t_2} f(t)\sin(2\phi(t,u))dt+\eta(u)\bigg]dudx\bigg|$$
		$$=\bigg|\int_{s-\frac{\pi}{t_2}}^{s+\frac{\pi}{t_2}}\int_{x-\frac{\pi}{t_1}}^{x+\frac{\pi}{t_1}}\bigg[\int_{t_1}^{t_2} f(t)\sin(2\phi(t,u))dt+\eta(t_2,u)-\eta(t_1,u)\bigg]dudx\bigg|$$
		$$=\bigg|\int_{s-\frac{\pi}{t_2}}^{s+\frac{\pi}{t_2}}\int_{x-\frac{\pi}{t_1}}^{x+\frac{\pi}{t_1}}\bigg[\int_{t_1}^{t_2} f(t)\sin(2\phi(t,u))dt\bigg]dudx\bigg|$$
		
		\noindent
		By Fubini's Theorem we can change the order of integration to arrive at
		
		$$=\bigg|\int_{s-\frac{\pi}{t_2}}^{s+\frac{\pi}{t_2}}\int_{t_1}^{t_2}f(t)\bigg[\int_{(x-\frac{\pi}{t},x+\frac{\pi}{t})}\sin(2\phi(t,u))du+\int_{(x-\frac{\pi}{t_1},x+\frac{\pi}{t_1})\setminus (x-\frac{\pi}{t},x+\frac{\pi}{t})}\sin(2\phi(t,u))du\bigg]dtdx\bigg|$$
		
		\noindent
		Since $\sin(2\phi(t,u))=\sin(2\arg E(t,u))$differs from a $\pi/t_k-$periodic function of $u$ with vanishing integral over each period by at most $o(1)$ and due to \eqref{misc58} we arrive at
		
		$$\ll \varepsilon_1 \frac{o(1)}{t_1t_2}+\varepsilon_1 \frac{t_2-t_1}{t_1t_2^2}.$$
		
		Thus, the first integral is for large enough $t_1$ less than $\varepsilon_1 (\coth[2v]-1)/t_2^2.$ Since by \eqref{misc57}, which is true for any point $u$ in the interval $I_{t_2}(s)$ we get
		
		$$\bigg|\int_{t_1}^{t_2} f(t)\sin(2\phi(t,u))dt+\eta(u)\bigg| \ll \varepsilon_1^u,$$
		
		\noindent
		combining both estimates shows that the left hand side must be less than $\varepsilon_1^u(\coth[2v]-1)/2$ on a set 
		
		$$R \subset \bigg(s-\frac{\pi(t_1+t_2)}{t_1t_2},s+\frac{\pi(t_1+t_2)}{t_1t_2}\bigg), \quad |R| \asymp \frac{1}{t_2}.$$
		
		\noindent
		Therefore, on $R$ we have the lower bound
		
		\begin{equation}\label{misc59}
		\arg a_{t_1 \rightarrow t_2}(u) \geq \frac{\coth[2v]-1}{2}\varepsilon_1^u+O((\varepsilon_1^u)^2+(\varepsilon_2^u)^2).
		\end{equation}
		
		In a last step we will have to make this lower bound uniform over $R,$ so the constant does only depend on $A$ as desired. To that end, note that for large enough $A$ the box $Q(s,A/t)$ contains multiple zeros of the approximating functions and, switching to the next zero if necessary we can assume with out loss of generality that $x_{t_k}^u-u>x_{t_k}-s.$ Indeed, this just boils down to the fact that $\sin(tz)$ whose zeros differ by exactly $k\pi/t,$ for $k\in \Z.$ If now $A$ is large, a lot of these zeros are in the box and you can switch to any zeros with out changing the approximating function. By Theorem \ref{Lemma_10}
		
		$$(x_{t_2}^u-iy_{t_2}^u)-(x_{t_1}^u-iy_{t_1}^u)=\xi_2-\xi_1+o(\xi_2-\xi_1),$$
		$$(\tilde{x}_{t_2}^u-iy_{t_2}^u)-(\tilde{x}_{t_1}^u-iy_{t_1}^u)=\tilde{\xi}_2-\tilde{\xi}_1+o(\xi_2-\xi_1),$$
		
		\noindent
		Again these relations do hold for all choices of the zeros of the approximating functions in the box $Q(s,A/t).$ This follows from the proof of Theorem \ref{Lemma_10}, more precisely \eqref{misc51}. Therefore, there exists $c>0$ such that 
		
		$$\varepsilon_1^u=t_2(x_{t_2}^u-u)-t_1(x_{t_1}^u-u)=$$
		$$t_2(x_{t_2}-s+c)-t_1{x_{t_1}-s+c}+o(\xi_2-\xi_1)>\varepsilon_1+o(|\varepsilon_1|+|\varepsilon_2|).$$
		
		Combining this estimate with \eqref{misc59} gives uniform on $R$
		
		$$|\arg a_{t_1\rightarrow t_2}(u)| > D\varepsilon_1+o(|\varepsilon_1|+|\varepsilon_2|),$$
		
		\noindent
		with a constant $D>0,$ depending only on $A.$
	\end{proof}

	We can further refine this bound by using Theorem \ref{Lemma_10} which tells us that the zeros of the approximating functions move like the zeros of $E(t,z).$ Indeed, if $\xi_{t_2}$ and $\xi_{t_1}$ are the zeros like before, then by \eqref{derivative_zero} we have

	$$\xi_2-\xi_1=\int_{t_1}^{t_2} \frac{f(t)}{\theta_z(t,\overline{\xi}(t))} dt.$$
	
	Moreover, by Lemma \ref{formulas_mif} (i) we know that on $(t_1,t_2),$ we have $|\theta_z(t,\overline{\xi}_t)|\asymp t.$ Hence, we expect a lower bound in terms of $||f||_{L^1(t_1,t_2)}.$
	
	\begin{lemma}\label{claim_2}
	 	In the situation of Lemma \ref{Lemma_11} set $J=(t_1,t_2)$ and suppose that $J$ is an $H_s-$interval. In particular, by assumption of Lemma \ref{Lemma_11} $J$ is a $\sigma-$interval. Then there exists a constant $D>0,$ depending only on $A,$ such that for large enough $t_1,$ we have on $R$ the estimate
	 	
	 	$$|\arg a_{t_1\rightarrow t_2}(s)|\geq D\int_{t_1}^{t_2} |f(t)|dt$$ 
	\end{lemma}

	\begin{proof}
		Let $\xi_t$ be the zero $E(t,z)$ from Theorem \ref{Lemma_10}. By Lemma \ref{formulas_mif} we have $|\theta_z(t,\overline{\xi}_t)|\asymp t$ for all $t\in J.$ Furthermore, by Lemma \ref{Lemma_7}
		
		$$\frac{\pi}{2}+2\arg \alpha(s,t)=-\arg\theta_z(t,\overline{\xi}(t)).$$
		
		Moreover, since $J$ is a $H_s-$interval, $|\Im \alpha^2(s,t)|>\frac{1}{200}.$ Since $|\alpha(s,t)|=1$ for all $t\in J,$ we have $|\arg \alpha^2(s,t)|\in [\frac{1}{200},\pi-\frac{1}{200}].$ Hence, $|\Im e^{-i\arg \alpha^2(s,t)}|\geq \frac{1}{300}.$ Thus, for large enough $t_1,$
		
		$$|\Re \theta_z(t,\overline{\xi}_t)| = |\theta_z(t,\overline{\xi}_t)||\Re e^{i\arg\theta_z(t,\overline{\xi}_t)}|=|\theta_z(t,\overline{\xi}_t)||\Re(ie^{i\arg \alpha^2(s,t)})|$$
		
		$$=|\theta_z(t,\overline{\xi}_t)||\Im e^{i\arg \alpha^2(s,t)}| \geq \frac{|\theta_z(t,\overline{\xi}_t)|}{300}.$$
		
		By continuity of $\theta_z$ this implies that for all $t\in J$ either 
		
		\begin{equation}\label{misc46}
		\Re \theta_z(t,\overline{\xi}_t) \geq \frac{|\theta_z(t,\overline{\xi}_t)|}{300} \text{ or } \Re \theta_z(t,\overline{\xi}_t) \leq -\frac{|\theta_z(t,\overline{\xi}_t)|}{300}.
		\end{equation}
		
		Assume that 
		
		\begin{equation}\label{misc37}
		\Re (\xi_2-\xi_1)= \Re \int_{J} \frac{f(t)}{\theta_z(t,\overline{\xi}_t)} dt \geq 0.
		\end{equation}
		
		The case where $\Re (\xi_{t_2}-\xi_{t_1})<0$ can be proved similarly. Since $J$ is a $H_s-$interval and $\sigma-$interval we find that
		
		\begin{equation}\label{misc36}
		|\xi_2-\xi_1| \ll |\Re (\xi_2-\xi_1)|.
		\end{equation}
		
		\noindent
		Indeed, 
		
		$$|\xi_2-\xi_1|=\bigg|\int_{J} \frac{f(t)}{\theta_z(t,\overline{\xi}_t)} dt\bigg|=\bigg|\int_{J} \frac{f(t)\overline{\theta_{z}(t,\overline{\xi}_t)}}{|\theta_z(t,\overline{\xi}_t)|^2} dt\bigg|=\bigg|\int_{J} \frac{f(t) \Re \theta_{z}(t,\overline{\xi}_t)}{|\theta_{z}(t,\overline{\xi}_t)|}dt - i \int_{J} \frac{f(t) \Im \theta_{z}(t,\overline{\xi}_t)}{|\theta_{z}(t,\overline{\xi}_t)|^2}dt\bigg|$$
		
		$$=|\Re(\xi_2-\xi_1)-i \Im(\xi_2-\xi_1)|\leq |\Re(\xi_2-\xi_1)|+|\Im(\xi_2-\xi_1)|.$$
		
		\noindent
		We thus have to show that 
		
		$$|\Im(\xi_2-\xi_1)| \ll |\Re(\xi_2-\xi_1)|.$$
		
		\noindent
		To see this, recall the following facts. By \eqref{misc46} we have that $\Re \theta_z(t,\overline{\xi}_t)$ has constant sign and $|\Im \theta_z(t,\overline{\xi}_t)|\leq |\theta_z(t,\overline{\xi}_t)|\leq 300 |\Re \theta_z(t,\overline{\xi}_t)|$. Furthermore, by Lemma \ref{formulas_mif} we have that $|\theta_z(t,\overline{\xi}_t)|\asymp t$ on $J$ and from the assumption $t_2-t_1\leq 1,$ we see immediately that $\frac{t_2}{t_1}\leq 2.$ Thus,
		
		$$|\Im(\xi_2-\xi_1)|=\bigg|\int_{J} \frac{f(t) \Im \theta_z(t,\overline{\xi}_t)}{|\theta_z(t,\overline{\xi}_t)|^2}dt\bigg|\leq \int_{J} \frac{|f(t)||\Im \theta_z(t,\overline{\xi}_t)|}{|\theta_z(t,\overline{\xi}_t)|^2}dt\leq \int_{J} \frac{|f(t)|}{|\theta_z(t,\overline{\xi}_t)|}dt$$
		
		$$\ll \frac{1}{t_1} \int_{J} |f(t)| dt \ll \frac{2(1-\sigma)}{t_2} \bigg|\int_{J} f(t) dt\bigg|\ll 2(1-\sigma) \bigg|\int_{J} \frac{f(t)}{|\theta_z(t,\overline{\xi}_t)|} dt\bigg|$$
		
		$$= 2(1-\sigma) \bigg|\int_{J} \frac{f(t)|\theta_z(t,\overline{\xi}_t)|}{|\theta_z(t,\overline{\xi}_t)|^2} dt\bigg|=600(1-\sigma) \bigg|\int_{J} \frac{f(t) \Re \theta_z(t,\overline{\xi}_t)}{|\theta_z(t,\overline{\xi}_t)|} dt\bigg| \ll |\Re(\xi_2-\xi_1)|.$$
		
		\noindent
		Thus, as claimed		
		
		$$|\xi_2-\xi_1| \ll |\Re (\xi_2-\xi_1)|.$$
				
		By Theorem \ref{Lemma_10} (iv), \eqref{misc37} and \eqref{misc36}
		
		$$x_{t_2}-x_{t_1}=\Re(\xi_2-\xi_1)+(x_{t_2}-x_{t_1}-\Re(\xi_2-\xi_1))\leq |\Re(\xi_2-\xi_1)|-|(x_{t_2}-x_{t_1}-\Re(\xi_2-\xi_1))|$$
		\begin{equation}\label{misc19}
		\geq |\Re(\xi_2-\xi_1)|-\psi(t_1)|\xi_2-\xi_1|\geq |\Re(\xi_2-\xi_1)|-D\psi(t_1)|\Re(\xi_2-\xi_1)|\geq D_1 |\Re(\xi_2-\xi_1)|
		\end{equation}
		
		\noindent
		for some positive constants $D$ and $D_1.$
		
		Since in Theorem \ref{Lemma_10} we have chosen $A>8\pi$ the box $Q(s,A/t_1)$ contains at least one zero $x_{t_2}$ of the approximating function with $x_{t_2}>s+\frac{1}{t_1}.$ Recall aswell that $t_2-t_1\leq \frac{1}{|s|+1}\leq 1.$Thus, using that $\Re \theta_z(t,\overline{\xi}_t)$ has constant sign we find
		
		$$(t_2(x_{t_2}-s)-t_1(x_{t_1}-s))=(t_2-t_1)(x_{t_2}-s)+t_1(x_{t_2}-x_{t_1})$$
		$$\geq \frac{t_2-t_1}{t_1}+t_1(x_{t_2}-x_{t_1})\overset{\eqref{misc19}}{\geq} \frac{t_2-t_1}{t_1} D_1|\Re t_1(\xi_2-\xi_1)|$$
		$$=\frac{t_2-t_1}{t_1} + D_1t_1\bigg|\Re\int_{J} \frac{f(t)}{\theta_z(t,\overline{\xi}_t)}dt\bigg|\geq \frac{t_2-t_1}{t_1} + \frac{D_1t_1}{300} \bigg|\int_{J} \frac{f(t)}{|\theta_z(t,\overline{\xi}_t)|}dt\bigg|$$
		$$\geq \frac{t_2-t_1}{t_1} + \frac{D_1D_2t_1}{300t_2}\bigg|\int_{J} f(t)dt\bigg|\geq \frac{t_2-t_1}{t_1} + \frac{D_1D_2(1-\sigma)}{600}\int_{J}|f(t)|dt.$$
		
		Furthermore, we now show that there is an absolute constant $D_3>0$ such that
		
		\begin{equation}\label{misc49}
		|t_2y_{t_2}-t_1y_{t_1}| < D_3\bigg(\frac{t_2-t_1}{t_1} \int_J |f(t)| dt\bigg).
		\end{equation}

		\noindent
		We separate two cases. First assume $t_2y_{t_2}\leq t_1y_{t_1}.$ Then, since $t_2>t_1$
		
		$$|t_2y_{t_2}-t_1y_{t_1}|=t_1y_{t_1}-t_2y_{t_2} \leq t_2y_{t_1}-t_2y_{t_2}= t_2 \int_{t_1}^{t_2} \frac{f(t)}{\theta_{z}(t,\xi_t)} dt \leq D_3 \int_{J} |f(t)|dt.$$
		
		\noindent
		Now assume that $t_2y_{t_2}\geq t_1y_{t_1}.$ Then,
		
		$$|t_2y_{t_2}-t_1y_{t_1}|=t_2y_{t_2}-t_1y_{t_1}=t_2y_{t_2}-{t_2-(t_2-t_1)}y_{t_1}\leq D_3\bigg(\frac{t_2-t_1}{t_1}+\int_{J} |f(t)|dt\bigg).$$
		
		\noindent
		Combining everything, gives as claimed
		
		$$|\arg a_{t_1\rightarrow t_2}(s)|\geq D\int_{t_1}^{t_2} |f(t)|dt$$ 
		
		on $R$ for an absolute constant $D>0,$ depending only on $A.$

	\end{proof}

	Since $|S_H^n|\leq |S|<\infty$ we can consider a finite collection of intverals $I_1,\dots, I_N$ of size $|I_k|=4C2^{-n},$ centered at $s_1,\dots , s_N\in S_H^n$ such that each point is covered by at most two intervals and $I_1,\dots, I_N$ cover at least one half of $S_H^n.$ We consider the intervals $L_{s_k}^n$ for $k=1,\dots, N$ as defined before. 
	
	\begin{lemma}\label{collection_horizontal_intervals}
		There exists a finite collection of pairwise disjoint intervals $\mathcal{T}_1,\dots \mathcal{T}_M, \mathcal{T}_l=(\tau_1^l,\tau_2^l)$ satisfying the following properties.
		
		\begin{itemize}
			\item [(i)] Each $\mathcal{T}_l$ belongs to $HL_{s_k}$ for at least one $k=1,\dots, N.$
			\item [(ii)] All $\mathcal{T}_l$ are $\sigma-$intervals.
			\item [(iii)] For each $k=1,\dots,N$ 
			
			$$\int_{\bigcup_{\mathcal{T}_l \in HL_{s_k}} \mathcal{T}_l} |f(t)| dt > \frac{\Delta}{500}.$$
			
			\item[(iv)] For each $\mathcal{T}_l$ we have 
		
			$$\int_{\mathcal{T}_l} |f(t)| dt < \frac{1}{100}.$$
		\end{itemize}
	\end{lemma}

\begin{proof}
	We can assume without loss of generality that $s_1,\dots,s_n$ are Lebesgue points of $f.$ Then, the claim follows immediately from the fact that all small enough intervals centered around $s_1,\dots, s_n$ are $\sigma-$intervals by Lemma \ref{sigma-intervals} and the fact that subintervals of $H_s-$intervals are $H_s-$intervals by definition.
\end{proof}
	
	Let $\mathcal{T}_1,\dots \mathcal{T}_M,\mathcal{T}_l=(\tau_1^l,\tau_2^l)$ be the finite collection of pairwise disjoint intervals from Lemma \ref{collection_horizontal_intervals}. Then,
	
	$$\sum_{l=1}^{M} ||\log |a_{\tau_1^l\rightarrow \tau_2^l}|||_{L^1(\R)} \overset{(1)}{\geq} \sum_{l=1}^{M} \frac{1}{2} \sum_{\substack{1\leq k \leq N :  \\ \mathcal{T}_l \in HL_{s_k}}} ||\log |a_{\tau_1^l\rightarrow \tau_2^l}|||_{L^1(I_{t_2}(s_k))}\overset{(2)}{\geq} \sum_{l=1}^{M} \frac{1}{2} \sum_{\substack{1\leq k \leq N :  \\ \mathcal{T}_l \in HL_{s_k}}} \frac{D^2}{\tau_2^l}\int_{\tau_1^l}^{\tau_2^l} |f(t)|dt.$$

	$$\overset{(3)}{\geq} \sum_{l=1}^{M} \sum_{\substack{1\leq k \leq N :  \\ \mathcal{T}_l \in HL_{s_k}}} \frac{D^2|I_k|}{32C} \int_{\mathcal{T}_l} |f(t)|dt\overset{(4)}{=} \sum_{k=1}^{N} \sum_{\substack{1\leq l \leq M :  \\ \mathcal{T}_l \in HL_{s_k}}} \frac{D^2|I_k|}{32C} \int_{\mathcal{T}_l} |f(t)|dt$$
	$$ \overset{(5)}{=}D^2\sum_{k=1}^{N} \frac{|I_k|}{32C} \int_{\bigcup_{\mathcal{T}_l \in HL_{s_k}} \mathcal{T}_l} |f(t)| dt \overset{(6)}{>} \frac{D^2\Delta}{16000C} \sum_{k=1}^{N} |I_k| \overset{(7)}{\geq} D^2\frac{\Delta}{32000C} |S_H^n|,$$
	
	\noindent
	where we used the following facts
	
	\begin{itemize}
		\item [(1)] By property $(i)$ of Lemma \ref{collection_horizontal_intervals} each $\mathcal{T}_l$ is a $H_{s_k}-$interval for at least one $1\leq k\leq N.$ Since $\mathcal{T}_l \subset L_s^n,$ and $|I_k|=4C2^{-n},$ we have $I_{t_2}(s_k)\subset I_k$ and since each point is covered by at most two of those intervals, this gives the factor $\frac{1}{2}.$ 
		
		\item[(2)] By property $(ii)$ of Lemma \ref{collection_horizontal_intervals} all $\mathcal{T}_l$ are $\sigma-$intervals. Hence, the conditions for Lemma \ref{claim_2} are satisfied and the inequality follows by applying that Lemma. Moreover, we used \eqref{misc51} and that $|R|\geq D'/\tau_2$ with $D'$ depending only on $A.$ Additionally, we assume without loss of generality that $D'>D,$ so we can write $DD'>D^2.$
		
		\item [(3)] Since by property $(i)$ of Lemma \ref{collection_horizontal_intervals} each $\mathcal{T}_l$ belongs to $HL_{s_k}$ for at least one $1\leq k \leq N.$ Thus, by definition $\mathcal{T}_l\subset L_{s_k}^n \subset [2^n,2^{n+2}].$ In particular, $\tau_2^l\leq2^{n+2}$ and thus
		
		$$\frac{D}{\tau_2^l} \geq \frac{D}{2^{n+2}}=\frac{D|I_k|}{16C}.$$
		
		\item [(4)] By changing the order of summation.
		
		\item [(5)] Since the intervals $\mathcal{T}_1 \dots, \mathcal{T}_M$ are pairwise disjoint.
		
		\item[(6)] By property (iii) of Lemma \ref{collection_horizontal_intervals}.
		
		\item [(7)] Since the intervals $I_1,\dots I_N$ cover at least one half of $S_H^n.$
		
	\end{itemize}

	Applying now the non-linear Parseval identity \eqref{non_linear_parseval_2} to $a_{\tau_1^l\rightarrow \tau_2^l}$ for all $l=1,\dots,M$ and using the property that the intervals $\mathcal{T}_l$ are pairwise disjoint and all contained in $[2^n,2^{n+2}]$ gives
	
	$$||f||_{L^2([2^n,2^{n+2}])}^2\geq \sum_{l=1}^{M} ||f||_{L^2(\mathcal{T}_l)}^2 = \sum_{l=1}^{M} ||\log|a_{\tau_1^l\rightarrow \tau_2^l}|||_{L^1(\R)}.$$
	
	Together with the last computation, we obtain the desired
	
	\begin{corollary}\label{claim_3}
		There exists a constant $D>0$ such that 
		
		$$D|S_H^n|<||f||_{L^2([2^n,2^{n+2}])}^2.$$
	\end{corollary}

	\subsection{Stronger approximations of the Hermite-Biehler functions}
	In this section, we are going to prove Theorem \ref{Lemma_10}. Since the statement is quite rich, we are going to construct the stronger approximation in several steps. First, we will prove that we can approximate $E(t,z)$ and $\tilde{E}(t,z)$ with the same imaginary part $y(t)$ and have the real parts satisfy $\cos[t(\tilde{x}(t)-x(t))]=\sqrt{w(s)\tilde{w}(s)}.$ Since we now approximate $E(t,z)$ and $\tilde{E}(t,z)$ at the same time, we will have to introduce analogous notation as in the case of $E(t,z).$ We define for a positive function $C=C(t)$ the sets
	
	$$\tilde{T}_0(s,C)=\{t >0 \mid \text{ there is a zero } z(t) \text{ of } \tilde{E}(t,z) \text{ inside } Q(s,C/t)\},$$
	
	$$\tilde{T}_1(s,C)=\{t> 0 \mid \text{ all zeros } z(t) \text{ of } \tilde{E}(t,z) \text{ inside } Q(s,C/t), \text{ satisfy } \Im z(t) < -1/t\}.$$
	
	Let us once again mention that all results out of the previous section that hold for $E(t,z)$ hold aswell for $\tilde{E}(t,z)$ and recall that by definition if there is no zero of $E(t,z),$ respectively $\tilde{E}(t,z),$ inside $Q(s,C/t),$ then $t \in T_1(s,C),$ respectively $t \in \tilde{T}_1(s,C).$ In particular it will be convenient for us to state Lemma \ref{Lemma_6} for $\tilde{E}(t,z)$ with $\cos(z)$ as the approximating function instead of $\sin(z)$ as in the case of $E(t,z).$ This just corresponds to the translation $\tilde{z}(t) \mapsto \tilde{z}(t)+\frac{\pi}{2t}.$ 
	 
	\begin{lemma}\label{Lemma_8}
		For almost all $s \in \R$ there exists a increasing function $C(t)>0, C(t) \rightarrow \infty$ as $t\rightarrow \infty$ with the following properties.
		
		For all $t \in T_1(s,C)\cap \tilde{T}_1(s,C)$ there exist $z(t)=u(t)-ip(t), \tilde{z}(t)=\tilde{u}(t)-ip(t)$ and $\alpha(t)=\alpha(s,t)$ such that $p(t)>0,$ $|\alpha(t)|=1$ and 
		
		$$	\sup_{z \in Q(s,C(t)/t)} \bigg| E(t,z) - \frac{\alpha(t)\gamma(ty(t))}{\sqrt{w(s)}} \sin[t(z-z(t))]\bigg|=o(1),$$
		
		$$	\sup_{z \in Q(s,C(t)/t)} \bigg| \tilde{E}(t,z) - \frac{\alpha(t)\gamma(ty(t))}{\sqrt{w(s)}} \cos[t(z-\tilde{z}(t))]\bigg|=o(1),$$
		
		\noindent
		as $t \rightarrow \infty.$ Furthermore, $p(t)>C(t)$ for $t \notin T_0(s,C)$ and $p(t)\leq C(t)$ for $t \in T_0(s,C).$ The functions $u(t)$ and $\tilde{u}(t)$ satisfy
		
		$$\cos[t(\tilde{u}(t)-u(t))]=\sqrt{w(s)\tilde{w}(s)}.$$ 
	\end{lemma}

	\begin{proof}
		We apply Corollary \ref{corollary_4} to both $E$ and $\tilde{E}$ to obtain for any $D>4\pi$
		
		\begin{equation}\label{misc20}
		\sup_{z \in Q(s,D/t)} \bigg| E(t,z) - \frac{\beta(t)\gamma(ty(t))}{\sqrt{w(s)}} \sin[t(z-\xi(t))]\bigg|=o(\gamma(ty(t))),
		\end{equation}
		
		\begin{equation}\label{misc21}
		\sup_{z \in Q(s,D/t)} \bigg| E(t,z) - \frac{\delta(t)\gamma(t\tilde{y}(t))}{\sqrt{\tilde{w}(s)}} \sin[t(z-\tilde{\xi}(t))]\bigg|=o(\gamma(ty(t))),
		\end{equation}
		
		\noindent
		for some $\xi(t)=x(t)-iy(t),\tilde{\xi}(t)=\tilde{x}(t)-i\tilde{y}(t)$ and unimodular functions $\beta(t),\delta(t).$ Hence,
		
		$$\det\begin{pmatrix}
		E(t,z) && \tilde{E}(t,z) \\
		E^{\sharp}(t,z) && \tilde{E}^{\sharp}(t,z)
		\end{pmatrix}=\frac{\gamma(ty(t))\gamma(t\tilde{y}(t))}{\sqrt{w(s)\tilde{w}(s)}} \big(\beta(t)\overline{\delta}(t)\sin[t(z-\xi(t))]\cos[t(z-\overline{\tilde{\xi}}(t))]$$
		
		\begin{equation}\label{misc22}
		-\overline{\beta}(t)\delta(t)\sin[t(z-\overline{\xi}(t))]\cos[t(z-\tilde{\xi}(t))]\big)+o(\gamma(ty(t))\gamma(t\tilde{y}(t))).
		\end{equation}

	\noindent
	If we plug in $z=x \in \R,$ \eqref{misc22} simplifies to 
	
	$$=2i\Im\bigg(\frac{\gamma(ty(t))\gamma(t\tilde{y}(t))}{\sqrt{w(s)\tilde{w}(s)}}\beta(t)\overline{\delta(t)}\sin[t(x-\xi(t))]\cos[t(x-\overline{\tilde{\xi}}(t))]\bigg)$$
	
	\begin{equation}\label{misc23}
	=2i\Im\bigg(\frac{\gamma(ty(t))\gamma(t\tilde{y}(t))}{\sqrt{w(s)\tilde{w}(s)}} \beta(t)\overline{\delta}(t) \frac{1}{2}\bigg[\sin[t(\overline{\tilde{\xi}}(t)-\xi(t))]+\sin[t(2x-(\xi(t)+\overline{\tilde{\xi}}(t)))]\bigg]\bigg).
	\end{equation}
	
	We know that by Corollary \ref{determinant_2} for all $z\in \C$ and any $t\geq 0$
	
	$$\det\begin{pmatrix}
	E(t,z) && \tilde{E}(t,z) \\
	E^{\sharp}(t,z) && \tilde{E}^{\sharp}(t,z)
	\end{pmatrix}=2i.$$
	
	\noindent
	Hence, by \eqref{misc22} the expression in \eqref{misc23} must be within $o(\gamma(ty(t))\gamma(t\tilde{y}(t))$ from $2i$ on the interval $I_t=Q(s,D/t)\cap \R.$ Suppose that for $t\in T_1(s,C(t))\cap\tilde{T}_1(s,C(t))$ we had the lower bound $\limsup_{t \rightarrow \infty} t|\tilde{y}(t)-y(t)|>\Delta>0.$ In that case for any fixed $t$ the first sine in \eqref{misc23} is constant, while the second has absolute value $\geq \sinh \Delta$ and its argument grows by more than $2\pi$ on $I_t.$ This immediately contradicts that the expression in \eqref{misc23} is within $o(\gamma(ty(t))\gamma(t\tilde{y}(t))$ from $2i$ on $I_t.$
	
	\noindent
	Thus, $y(t)=\tilde{y}(t)+o(1/t)$ and we can change either of $y(t)$ and $\tilde{y}(t)$ into the other with \eqref{misc20} and \eqref{misc21} remaining still true, just with a different error term of the same order. So we assume from now on that $\tilde{y}(t)=y(t)$ and set $p(t)=y(t)=\tilde{y}(t).$ By choice of the approximation from Corollary \ref{corollary_4} we have $p(t)>C(t)$ for $t \notin T_0(s,C(t))$ and $p(t)\leq C(t)$ for $t\in T_0(s,C(t)).$ 
	
	Suppose now that $\xi(t)=x(t)-ip(t) \in Q(s,D/t)$ and $tp(t)>1.$ By plugging $z=\xi(t)$ and $z=\overline{\xi}(t)$ into \eqref{misc22} we obtain the equations
	
	$$\frac{2i\overline{\beta}(t)\delta(t)}{\sqrt{w(s)\tilde{w}(s)}}\cos[t(\tilde{x}(t)-x(t))]=2i+o(1),$$
	
	$$\frac{2i\beta(t)\overline{\delta}(t)}{\sqrt{w(s)\tilde{w}(s)}}\cos[t(\tilde{x}(t)-x(t))]=2i+o(1).$$
	
	\noindent
	We immediately conclude from both equations that $\arg\beta(t)=\arg\delta(t)=o(1)$ and since $|\beta(t)|=|\delta(t)|=1$ we find that $\beta(t)=\delta+o(1)$ and thus set $\alpha(t)=\beta(t)=\delta+o(1)$ for $t \in T_0(s,D)\subset T_0(s,C(t)).$ 
	
	Furthermore, we conclude from these equations that
	
	$$\cos[t(\tilde{x}(t)-x(t))]=\sqrt{w(s)\tilde{w}(s)}+o(1).$$ 
	
	\noindent
	Hence, $\tilde{x}(t)=\tilde{u}(t)+o(1/t)$ and $x(t)=u(t)+o(1/t),$ where $u(t)$ and $\tilde u(t)$ are chosen continuously to satisfy the equation 
	
	$$\cos[t(\tilde{u}(t)-u(t))]=\sqrt{w(s)\tilde{w}(s)}.$$ 
	
	\noindent
	Thus, by replacing the $x(t)$ with $u(t)$ and $\tilde{x}(t)$ with $\tilde{u}(t)$ the approximations \eqref{misc20} and \eqref{misc21} remain valid with a different error term of the same order.
	
	Finally, by our standard argument, we can improve from fixed $D>4\pi$ to $C(t)$ by making $C(t)$ grow a bit slower if necessary.
	\end{proof}

	In the next step, we are going to change the approximation, such that the approximating function and $E(t,z)$ have the same value at the point $s\in \R$ instead of sharing a common zero. The following simple lemma will play an important role in this process.
	
	\begin{lemma}\label{determinant_3}
		Let $a,b,c,d \in \C\setminus\{0\}$ be such that $\frac{a}{b}=\frac{c}{d}$ and
		
		$$\det\begin{pmatrix}
		a && b \\
		
		\overline{a} && \overline{b}
		\end{pmatrix}
		=
		\det\begin{pmatrix}
		c && d \\
		\overline{c} && \overline{d}
		\end{pmatrix}\neq 0.$$
		
		\noindent
		Then, $|a|=|c|$ and $|b|=|d|.$
	\end{lemma}

	\begin{proof}
		We write  $b=\frac{ad}{c}$ and thus obtain 
		
		$$a \overline b - \overline a b = a \frac{\overline{ad}}{\overline c} - \overline a \frac{ad}{c}=|a|^2\bigg(\frac{c \overline d- d \overline c}{|c|^2}\bigg)=c \overline d -d \overline c.$$
		
		\noindent
		We immediately see $|a|^2=|c|^2.$ We get $|b|^2=|d|^2$ similarly by writing $a=\frac{bc}{d}.$
	\end{proof}
	
	Thus, our strategy will be to change the approximating functions slightly such that they satisfy the assumptions of the previous Lemma.
	
	\begin{lemma}\label{Lemma_9}
		For almost all $s\in \R$ and for all $C>1$ there exists $\varepsilon_0>0$ such that for all $\varepsilon<\varepsilon_0$ the following holds. If 
		
		$$\sup_{z \in Q(s,3C/t)} \bigg| E(t,z) - \frac{\alpha\gamma(ty)}{\sqrt{w(s)}} \sin[t(z-(x-iy))]\bigg|< \varepsilon,$$
		
		$$\sup_{z \in Q(s,3C/t)} \bigg| \tilde E(t,z) - \frac{\alpha\gamma(ty)}{\sqrt{\tilde w(s)}} \cos[t(z-(\tilde x-i\tilde y))]\bigg|< \varepsilon,$$
		
		\noindent
		for some $t>0,x,\tilde x, y \in \R, \alpha \in \C,$ satisfying $1<ty<2C, |\alpha|=1$ and
		
		\begin{equation}\label{misc24}
		-\frac{\pi}{2} < t(\tilde x-x)< \frac{\pi}{2}, \quad \cos[t(\tilde x - x)]=\sqrt{w(s)\tilde w(s)},
		\end{equation}
		
		\noindent
		then there exist $y',x',\tilde x' \in \R$ and $\alpha'\in \C, |\alpha'|=1,$ such that 
		
		$$|tx'-tx|+|t\tilde x'-t\tilde x|+|ty'-ty| + |\alpha'-\alpha|< D\varepsilon, \quad \tilde x-x=\tilde x'-x'$$
		
		\noindent
		for some constant $D=D(C,s)$ and furthermore 
		
		\begin{equation}\label{misc25}
		E(t,s)=\alpha'\frac{\gamma(ty)}{\sqrt{w(s)}}\sin[t(s-(x'-iy'))], \quad \tilde E(t,s)=\alpha'\frac{\gamma(ty)}{\sqrt{\tilde w(s)}}\cos[t(s-(\tilde x'-iy'))].
		\end{equation}
	\end{lemma}

	\begin{proof} Recall that for almost all $s \in \R$ we have that $w(s)\neq0$ and $\tilde w(s)\neq 0.$ So assume without loss of generality that $w(s)\tilde w(s)\neq 0.$
		We set for all $z \in \C$ where it is well-defined
		
		$$f(z)=\sqrt{\frac{w(s)}{\tilde w(s)}} \cdot \frac{\sin[t(z-(x-iy))]}{\cos[t(z-(\tilde x -i y))]}.$$
		
		Denote by $J$ the middle-third of $I=Q(s,3C/t)\cap \R.$ and note that by the same calculation as in the proof of Lemma \ref{formulas_mif} $(i)$ we find a constant $D_1$ such that for all $s \in J,$ we have if $\varepsilon_0$ is small enough
		
		$$\bigg|\frac{E(t,s)}{\tilde{E}(t,s)}-f(s)\bigg|< D_1\varepsilon.$$
		
		\noindent
		Since $1<ty<2C,$ $f$ is holomorphic in a $\frac{1}{2t}-$neighborhood of $J.$ Hence, for small enough $\varepsilon$ there exists by Lemma \ref{open_mapping_theorem} a $D_2>0$ such that in the disk $B(s,D_2\varepsilon/t),$ f takes all values from $B(f(s),D_1\varepsilon).$ We choose $a \in B(s,D_2\varepsilon/t)$ such that 
		
		$$f(a)=\frac{E(t,s)}{\tilde E(t,s)}.$$
		
		If we now set $\tilde x'= \tilde x+ \Re(s-a), x'=x+ \Re(s-a), y'= y + \Im a,$ then 
		
		$$\frac{E(t,s)}{\tilde{E}(t,s)}= \sqrt{\frac{\tilde w(s)}{w(s)}} \cdot \frac{\sin[t(s-(x'-iy'))]}{\cos[t(s-(\tilde x'-iy'))]}.$$
		
		\noindent
		Using trigonometric identities and \eqref{misc24} one can check that
		
		$$\det \begin{pmatrix}
		\frac{\gamma(ty)}{w(s)} \sin[t(s-(x'-iy'))] && \frac{\gamma(ty)}{\tilde w(s)} \cos[t(s-(\tilde x'-iy'))] \\
		\frac{\gamma(ty)}{w(s)} \sin[t(s-(x'+iy'))] && \frac{\gamma(ty)}{\tilde w(s)} \cos[t(s-(\tilde x'+iy'))]
		\end{pmatrix}=2i.$$
		
		\noindent
		Together with Corollary \ref{determinant_2} and Lemma \ref{determinant_3} this implies

		$$|E(t,s)|=\bigg|\frac{\gamma(ty)}{w(s)} \sin[t(s-(x'-iy'))]\bigg|, \quad |\tilde E(t,s)|=\bigg|\frac{\gamma(ty)}{\tilde w(s)} \cos[t(s-(\tilde x'-iy'))]\bigg|.$$
		
		Since we have additionally 
		
		$$\arg \frac{E(t,s)}{\tilde E(t,s)} = \arg \frac{\sin[t(s-(x'-iy'))]}{\cos[t(s-(\tilde x'-iy'))]},$$
		
		\noindent
		we find $\alpha' \in \C, |\alpha'|=1$ such that $\eqref{misc25}$ holds. By the inequalities from the statement we have $|\alpha-\alpha‘|\ll \varepsilon.$ Indeed,
		
		$$|\alpha'-\alpha|\cdot \bigg|\frac{\alpha\gamma(ty)}{\sqrt{w(s)}} \sin[t(s-(x'-iy'))]\bigg|=\bigg|\alpha'\frac{\gamma(ty)}{\sqrt{w(s)}} \sin[t(s-(x'-iy'))]-\alpha\frac{\gamma(ty)}{\sqrt{w(s)}} \sin[t(z-(x'-iy'))]\bigg|$$
		
		$$=\bigg|E(t,s)-\frac{\alpha\gamma(ty)}{\sqrt{w(s)}} \sin[t(z-(x'-iy'))]\bigg|\leq \bigg|E(t,s)-\alpha\frac{\gamma(ty)}{\sqrt{w(s)}} \sin[t(z-(x-iy))]\bigg|$$
		$$+\bigg|\alpha\frac{\gamma(ty)}{\sqrt{w(s)}} \sin[t(z-(x-iy))]-\alpha\frac{\gamma(ty)}{\sqrt{w(s)}} \sin[t(z-(x'-iy'))]\bigg|$$
		
		\noindent
		By assumption we have
		
		$$\bigg|E(t,s)-\alpha\frac{\gamma(ty)}{\sqrt{w(s)}} \sin[t(z-(x-iy))]\bigg|< \varepsilon$$
		
		\noindent
		and by the mean value Theorem for complex functions, using that the box is convex and compact, as well as the fact that $\gamma(ty)$ is bounded due to $1<ty<2C$ we have 
		
		$$\bigg|\alpha\frac{\gamma(ty)}{\sqrt{w(s)}} \sin[t(z-(x-iy))]-\alpha\frac{\gamma(ty)}{\sqrt{w(s)}} \sin[t(z-(x'-iy'))]\bigg|\leq D_3|(tx'-tx)+(ty'-ty)|<D_3\varepsilon,$$
		
		\noindent
		where $D_3$ is a constant depending only on $s$ and $C.$ Thus, $|\alpha'-\alpha|\ll \varepsilon$ by using again that $\gamma(ty)$ is bounded from above and below since $1<ty<2C.$
		
	\end{proof}

	Let us briefly summarize what we have shown so far and restate Theorem \ref{Lemma_10} to see what is still left to prove.
	
	\begin{corollary}\label{corollary_5}
	For almost all $s \in \R$ there exist positive functions $C(t), C(t)\rightarrow \infty, \psi(t), \psi(t) \rightarrow 0$ as $t\rightarrow \infty, t \in T_1(s,C(t))\cap \tilde T_1(s,C(t)),$ real functions $x(t), \tilde x(t), y(t),$ with $ty(t)>1$ and a complex function $\alpha(t) \in \C, |\alpha(t)|=1$ such that for all $t \in T_1(s,C(t))\cap \tilde T_1(s,C(t))$ we have 
	
	\begin{itemize}
		\item [(i)]
		
		$$\sup_{z \in Q(s,3C(t)/t)} \bigg|E(t,z)-\frac{\alpha(t)\gamma(ty(t))}{\sqrt{w(s)}} \sin[t(z-(x(t)-iy(t)))]\bigg|<\psi(t),$$
		
		$$\sup_{z \in Q(s,3C(t)/t)} \bigg|\tilde{E}(t,z)-\frac{\alpha(t)\gamma(ty(t))}{\sqrt{\tilde{w}(s)}} \cos[t(z-(\tilde{x}(t)-iy(t)))]\bigg|<\psi(t).$$
		
		\item[(ii)]
		
		$$ -\frac{\pi}{2}<t(\tilde{x}(t)-x(t))<\frac{\pi}{2}, \quad \cos[t(\tilde{x}(t)-x(t))]=\sqrt{w(s)\tilde{w}(s)}.$$
		
		\item[(iii)] For large enough $t$ for which $1<ty(t)<2C(t),$
		
		$$E(t,s)=\frac{\alpha(t)\gamma(ty(t))}{\sqrt{w(s)}} \sin[t(s-(x(t)-iy(t)))],$$
		
		$$\tilde{E}(t,s)=\frac{\alpha(t)\gamma(ty(t))}{\sqrt{\tilde{w}(s)}} \cos[t(s-(\tilde{x}(t)-iy(t)))].$$
		
			\end{itemize}
	
 \end{corollary}

	Note in the following that we established all claims about the approximation itself. We are left to establish the claims about the movement of the zeros of $E, \tilde E$ and the approximating functions.
	
	\begin{theorem*}[Strong approximation]
		For almost all $s\in \R$ there exists a increasing functions $C(t)>0, C(t) \rightarrow \infty$ and a function $\psi(t)>0, \psi(t)\rightarrow 0$ as $t\rightarrow \infty$ such that the following holds.
		
		Let $(t_1,t_2) \subset T_0(s,C(t))\cap \tilde T_0(s,C(t)), t_2-t_1 \leq \frac{1}{|s|+1}$ be a $\sigma-$interval and $H_s-$interval such that 
		
		$$\int_{t_1}^{t_2} |f(t)|dt < \frac{1}{100\cosh[2A]},$$
		
		\noindent
		and $C(t)>10\pi$ for $t>t_1,$ where $A$ is a constant satisfying $10\pi< A < C(t)$ for all $t>t_1.$ Let $\xi_1$ a zero of $E(t_1,z)$ in $Q(s,A/t_1)$ which moves continuously inside $Q(s,A/t)$ to a zero $\xi_2$ of $E(t_2,z)$ as $t$ changes from $t_1$ to $t_2.$ Let $\tilde{\xi_1}, \tilde{\xi_2}$ be similar zeros of $\tilde{E}$ inside $Q(s,A/t).$ Suppose that $t_k\Im \xi_k>2, t_k \Im \tilde{\xi_2}>2$ for $k=1,2.$
		
		Then, the zeros of $E(t,z)$ and $\tilde{E}(t,z)$ change in similar ways as $t$ changes from $t_1$ to $t_2,$ 
		
		$$|(\xi_2-\xi_1)-(\tilde{\xi_2}-\tilde{\xi_1})|\leq \psi(t_1)|\xi_2-\xi_1|,$$
		
		\noindent
		and there exist real continuous functions $y_{t}, x_{t}, \tilde{x}_{t}$ and a unimodular continuous function $\alpha_{t}$ on the interval $[t_1,t_2]$ such that $ty_{t}>1$ for all $t\in [t_1,t_2]$ and
		
		\begin{itemize}
			\item [(i)] At $t=t_1$ we have
			
			$$\sup_{z \in Q(s,3C(t_1)/t_1)} \bigg|E(t_1,z)-\frac{\alpha_{t_1}\gamma(t_1y_{t_1})}{\sqrt{w(s)}} \sin[t_1(z-(x_{t_1}-iy_{t_1}))]\bigg|<\psi(t_1),$$
			
			$$\sup_{z \in Q(s,3C(t_1)/t_1)} \bigg|\tilde{E}(t_1,z)-\frac{\alpha_{t_1}\gamma(t_1y_{t_1})}{\sqrt{\tilde{w}(s)}} \cos[t_1(z-(\tilde{x}_{t_1}-iy_{t_1}))]\bigg|<\psi(t_1).$$
			
			\item[(ii)] For all $t \in [t_1,t_2]$ we have 
			
			$$ -\frac{\pi}{2}<t(\tilde{x}_{t}-x_{t})<\frac{\pi}{2}, \quad \cos[t(\tilde{x}_{t}-x_{t})]=\sqrt{w(s)\tilde{w}(s)}.$$
			
			\item[(iii)] For $k=1,2$ and large enough $t_1$ we have 
			
			$$E(t_k,s)=\frac{\alpha_{t_k}\gamma(t_ky_{t_k})}{\sqrt{w(s)}} \sin[t_k(s-(x_{t_k}-iy_{t_k}))],$$
			
			$$\tilde{E}(t_k,s)=\frac{\alpha_{t_k}\gamma(t_ky_{t_k})}{\sqrt{\tilde{w}(s)}} \cos[t_k(s-(\tilde{x}_{t_k}-iy_{t_k}))].$$
			
			\item[(iv)] As $t$ changes from $t_1$ to $t_2$ the zeros of the approximating functions change similarly to the zeros of $E(t,z)$ and $\tilde{E}(t,z),$
			
			$$ |[(x_{t_2}-iy_{t_2})-(x_{t_1}-iy_{t_1})]-[\xi_2-\xi_1]|\leq \psi(t_1)|\xi_2-\xi_1|,$$
			
			$$|[(\tilde{x}_{t_2}-iy_{t_2})-(\tilde{x}_{t_1}-iy_{t_1})]-[\tilde{\xi_2}-\tilde{\xi_1}]| \leq \psi(t_1)|\xi_2-\xi_1|.$$
		\end{itemize}
	\end{theorem*}

	\begin{proof}
		Assume $s \in \R$ is like in Corollary \ref{corollary_5}. Then, $y_{t_1},x_{t_1},\tilde x_{t_1}$ and $\alpha_{t_1}$ satisfying $(i)-(iii)$ do exist. We will have to find $y_{t_2},x_{t_2},\tilde x_{t_2}$ and $\alpha_{t_2}$ such that $(iv)$ holds and $(i)$ and $(iii)$ remain true as well, for $t=t_2.$
		
		Recall that rewriting the real Dirac system \eqref{dirac_system} gave us the differential equation
		
		\begin{equation}\label{misc26}
		\frac{\partial}{\partial t} E(t,z)= -izE(t,z)+f(t)E^{\sharp}(t,z)
		\end{equation}
		
		\noindent
		 for the Hermite-Biehler function $E(t,z)=A(t,z)-iC(t,z).$ Let $S(t,z)$ be the solution to \eqref{misc26} for $t\in [t_1,t_2]$ satisfying the initial condition 
	 
	$$S(t_1,z)=\alpha_{t_1}\frac{\gamma(t_1y_{t_1})}{\sqrt{w(s)}} \sin[t_1(z-(x_{t_1}-iy_{t_1}))].$$
	
	Note that Lemma \ref{differential_equation} and Lemma \ref{differential_equation_2} can be formualted aswell for the Hermite-Biehler function $S(t,z)$ and hence the associated scattering function $\mathfrak{S}(t,z)=e^{itz}S(t,z)$ satisfies the differential equation
	
	\begin{equation}\label{misc35}
	\frac{\partial}{\partial t} \mathfrak{S}(t,z)=iz\mathfrak{S}(t,z)+e^{izt}\frac{\partial}{\partial t} E(t,z)=f(t)e^{2izt}\mathfrak{S}^{\sharp}(t,z).
	\end{equation}
	
	Hence, we have for all $x \in \R$ 
	
	$$\frac{\partial}{\partial t} |S(t,x)| \leq |S(t,x)||f(t)|.$$
	
	\noindent
	Indeed a simple calculation recalling that $\mathfrak{S}(t,z)=e^{itz}S(t,z)$ yields 
	
	$$\frac{d}{dt}|S(t,x)|=\frac{d}{dt}|\mathfrak{S}(t,x)|=\Re\bigg(\frac{S(t,x)\overline{\frac{d}{dt}S(t,x)}}{|S(t,x)|}\bigg)\overset{\eqref{misc35}}{=}\frac{\Re(e^{ixt}S(t,x)f(t)e^{-ixt}S(t,x))}{|S(t,x)|}$$
	
	$$=\frac{\Re(S^2(t,x)f(t))}{|S(t,x)|}\leq |S(t,x)||f(t)|.$$
	
	\noindent
	Integrating this inequality yields the pointwise bound
	
	\begin{equation}\label{misc27}
	|S(t,x)|\leq |S(t_1,x)|e^{\int_{t_1}^{t} |f(u)|du}.
	\end{equation}
	
	By assumption we have that 
	
	$$\int_{t_1}^{t_2} |f(t)|dt < \frac{1}{100}.$$
	
	\noindent
	Using this together with the inequality $e^{x}\leq 4(1+x)$ for $0\leq x \leq \frac{1}{100}$ gives 
	
	$$e^{\int_{t_1}^{t} |f(u)|du} \leq 4\bigg(1+\int_{t_1}^{t} |f(u)|du\bigg)\leq 8.$$
	
	Moreover, we have chosen the initial condition such that 
	
	$$|S(t_1,x)|=\bigg|\alpha_{t_1}\frac{\gamma(t_1y_{t_1})}{\sqrt{w(s)}} \sin[t_1(x-(x_{t_1}-iy_{t_1}))]\bigg|\leq \bigg|\frac{1}{\sqrt{w(s)}} \frac{\sqrt{2}}{\sqrt{\sinh[2t_1y_{t_1}]}}\sin[t_1(x-(x_{t_1}-iy_{t_1}))]\bigg|$$
	
	$$\leq \bigg|\frac{\sqrt{2}}{\sqrt{w(s)}} \frac{\sin[t_1(x-x_{t_1})]\cosh(t_1y_{t_1})+i\cos[t_1(x-x_{t_1})]\sinh(t_1y_{t_1})}{\sqrt{\sinh(2t_1y_{t_1})}}\bigg|$$
	
	\noindent
 	Recall the following limits
 	
 	$$\lim_{t \rightarrow \infty} \frac{\cosh t}{\sqrt{\sinh t}}=\frac{1}{\sqrt{2}}, \quad \lim_{t \rightarrow \infty} \frac{\sinh t}{\sqrt{\sinh t}}=\frac{1}{\sqrt{2}}.$$
 	
 	\noindent
 	Inserting these limits in the last formula together with the fact that $t_1y_{t_1}^>1$ gives us for large $t_1$ that
 	
 	$$|S(t_1,x)|\leq \frac{\sqrt{2}}{\sqrt{w(s)}} \sqrt{2\bigg(\frac{1}{\sinh 2}\bigg)^2}< \frac{1}{\sqrt{w(s)}}.$$
 	
 	\noindent
 	Combining these estimates with \eqref{misc27} yields the bound
 	
 	$$|S(t,x)|\leq |S(t_1,x)|e^{\int_{t_1}^{t} |f(u)|du}\leq \frac{8}{\sqrt{w(s)}}.$$
 	
 	By Lemma \ref{differential_equation_2} we have for all $t\in (t_1,t_2)$ and $x\in \R$ that
 	
 	$$|\mathfrak{S}(t_1,x)-\mathfrak{S}(t,x)|=\bigg|e^{it_1x}S(t_1,x)-e^{itx}S(t,x)\bigg|=\int_{t_1}^{t}\frac{d}{dt}|\mathfrak{S}(t,x)|=\int_{t_1}^{t}\frac{d}{dt}|S(t,x)|$$
 	$$\leq \int_{t_1}^{t} |S(t,x)||f(t)|\leq \frac{8}{\sqrt{w(s)}}\int_{t_1}^{t} |f(u)|du.$$
 	
 	Since $\mathfrak{S}(t,z)$ is outer and together with the maximum principle for the upper half plane this implies that for all $z\in \C_+$ we have
 	
 	$$|\mathfrak{S}(t_1,z)-\mathfrak{S}(t,z)|\leq \frac{8}{\sqrt{w(s)}}\int_{t_1}^{t} |f(u)|du.$$
 	
 	\noindent
 	We are going to use this bound, to bound the distance of two zeros of $S(t,z).$ We denote by $\zeta_{t_1}=\zeta_1$ the zero of $S(t_1,z)$ inside $Q(s,A/t_1)$ that is closest to $\xi_1.$ If there are more than one at the same distance, we pick one of them. Denote by $\xi_t$ and $\zeta_t$ the zeros of $E(t,z)$ and $S(t,z)$ evolving from $\xi_1$ and $\zeta_1,$ respectively. 
 	
 	From the previous bound we find
 	
 	$$|\mathfrak{S}(t_1,\zeta_t)-\mathfrak{S}(t,\zeta_t)|=|\mathfrak{S}(t_1,\zeta_t)|\leq \frac{8}{\sqrt{w(s)}}\int_{t_1}^{t} |f(u)|du.$$
 	
 	Hence, we find a constant $C_1,$ depending only on $s$ and $A$ such that
 	
 	\begin{equation}\label{misc28}
 	|\zeta_1-\zeta_t|\leq \frac{C_1}{t_1} \int_{t_1}^{t_2} |f(u)|du.
 	\end{equation}
 	
 	\noindent
 	This inequality holds in fact for any zero $\nu_1$ of $S(t_1,z)$ and its evolution $\nu_t$ as the same calculation shows. Let now $z_{t_1}$ be any zero of $E(t_1,z)$ in $Q(s,A/t_1)$ and $z_t$ its evolution. This includes the choice $z_t=\xi_t.$ Using \eqref{derivative_zero} and Lemma \ref{formulas_mif} we have that $|\theta_z(t,z_t)|\asymp t$ for $t\in [t_1,t_2]$ and 
 	
 	\begin{equation}\label{misc29}
 	|z_1-z_t|=\bigg|\int_{t_1}^{t} \frac{f(t)}{\theta_z(t,z_t)}\bigg|\leq \frac{C_2}{t_1} \int_{t_1}^{t_2} |f(u)|du
 	\end{equation}
 	
 	\noindent
 	for some constant $C_2,$ depending only on $s$ and $A.$ Furthermore, we can relate $\xi_1$ and $\zeta_1$ with the help of $(i).$ From
 	
 	$$\sup_{z \in Q(s,3C(t_1)/t_1)} \bigg|E(t_1,z)-\frac{\alpha_{t_1}\gamma(t_1y_{t_1})}{\sqrt{w(s)}} \sin[t_1(z-(x_{t_1}-iy_{t_1}))]\bigg|<\psi(t_1),$$
 	
 	\noindent
 	we can first of all assume without loss of generality that $\zeta_1=x_{t_1}-iy_{t_1},$ if not we can just change to $\zeta_1,$ since all zeros of the approximating function have the same imaginary part and the real parts differ by $2\pi.$ By periodicity of $\sin(z)$ this does not change the approximation. Then, by plugging in $z=\xi_1$ we find
 	
 	$$\bigg|E(t_1,\xi_1)-\frac{\alpha_{t_1}\gamma(t_1y_{t_1})}{\sqrt{w(s)}} \sin[t_1(\xi_1-(x_{t_1}-iy_{t_1}))]\bigg|=\bigg|\frac{\gamma(t_1y_{t_1})}{\sqrt{w(s)}} \sin[t_1(\xi_1-\zeta_1)]\bigg|<\psi(t_1)$$
 	
 	\noindent
	Since $\psi(t_1) \rightarrow 0,$ as $t_1\rightarrow \infty,$ we must have because of $1<t_1y_{t_1}<A$ and the fact that $\zeta_1$ was chosen as the zero of $S(t_1,z)$ closest to $\xi_1$ that
	
	\begin{equation}\label{misc30}
	\xi_1=\zeta_1+o(1/t_1).
	\end{equation}
	
	Now the equations \eqref{misc28}, \eqref{misc29}, and \eqref{misc30} imply immediately 
	
	\begin{equation}\label{misc60}
	|\xi_t-\zeta_t|< \frac{C_3}{t_1} \int_{t_1}^{t_2} |f(u)| du + o(1/t_1)
	\end{equation}
	
	\noindent
	for some constant $C_3$ depending only on $s$ and $A.$
	
	Recall that we can associate to any Hermite-Biehler function $E$ a meromorphic inner function by 
	
	$$\theta(t,z)=\frac{E^{\sharp}(t,z)}{E(t,z)}.$$
	
	\noindent
	We will denote by $I(t,z)$ the meromorphic inner function associated to $S(t,z).$ We have he estimate
	
	$$|\theta_z(t,\overline{\xi}_1)-I_z(t_1,\overline{\zeta}_1)|=o(t_1).$$
	
	\noindent
	Indeed, looking at the initial condition of $S(t_1,z)$ this is nothing, but Lemma \ref{formulas_mif} $(i).$
	
	Furthermore, since $E$ satisfies $(i)$ at $t=t_1$ we can apply Lemma \ref{formulas_mif} and get
	
	\begin{equation}\label{misc61}
	|\theta_z(t,\overline{\xi}_t)| \asymp t, \quad \bigg|\frac{\theta_{zz}(t,\overline{\xi}_t)}{\theta_z(t,\overline{\xi}_t)}\bigg|\ll t.
	\end{equation}
	
	By \eqref{misc60} the zeros of $E(t,z)$ and $S(t,z)$ are close and thus one can show by looking at the product representation of $\theta(t,z)$ and $I(t,z)$ that
	
	\begin{equation}\label{misc31}
	|I_z(t,\overline{\zeta}_t)| \asymp t, \quad \bigg|\frac{I_{zz}(t,\overline{\zeta}_t)}{I_z(t,\overline{\zeta}_t)}\bigg|\ll t.
	\end{equation}
	
	We now find by \eqref{riccati} that for all $t\in (t_1,t_2)$
	
	\begin{equation}\label{misc39}
	|I_z(t,\overline{\zeta}_t)-\theta_z(t,\overline{\xi}_t)|\ll t\bigg(\int_{t_1}^{t_2} |f(u)|du +o(1)\bigg).
	\end{equation}
	
	Indeed, by equation \eqref{riccati}
	
	$$\frac{d}{dt}\theta_z(t,\overline{\xi}_t)=g(t)\theta_z(t,\overline{\xi}_t),$$
	
	where
	
	$$g(t)=2i\overline{\xi}_t-f(t)\frac{\theta_{zz}(t,\overline{\xi}_t)}{\theta_z^2(t,\overline{\xi}_t)}.$$
	
	We then write 
	
	$$\theta_z(t,\overline{\xi}_t)=\theta_z(t_1,\overline{\xi}_1)e^{\int_{t_1}^{t_2}g(t) dt}.$$
	
	Similarly,
	
	$$I_z(t,\overline{\zeta}_t)=I_z(t_1,\overline{\zeta}_1)e^{\int_{t_1}^{t_2} h(t) dt}, \quad h(t)=2i\overline{\zeta}_t+f(t) \frac{I_{zz}(t,\overline{\zeta}_t)}{I_z^2(t,\overline{\zeta}_t)}.$$
	
	Therefore, using the above established bounds \eqref{misc60}, \eqref{misc61}, and \eqref{misc31} we find
	
	$$|I_z(t,\overline{\zeta}_t)-\theta_z(t,\overline{\xi}_t)|\leq |\theta_z(t_1,\overline{\xi}_1)-I_z(t_1,\overline{\zeta}_1)|e^{\int_{t_1}^{t}|g| du}+ 
		|I_z(t_1,\overline{\zeta}_1)|\big|e^{\int_{t_1}^{t}g(u)du}-e^{\int_{t_1}^{t}h(u)du}\big|$$
		
	$$\ll o(1)t_1+ t \int_{t_1}^{t} |g-h|du \ll t \bigg( \int_{t_1}^{t_2} |f(t)| dt +o(1) \bigg).$$
	
	\noindent
	This estimate combined with \eqref{derivative_zero} immediately gives a bound on the difference of the velocities.
	
	\begin{equation}\label{misc47}
	|(\zeta_t)'(t)-(\xi_t)'(t)|=\bigg|\frac{f(t)}{I_z(t,\overline{\zeta}_t)}-\frac{f(t)}{\theta_z(t,\overline{\xi}_t)}\bigg|=|f(t)|\bigg|\frac{\theta_z(t,\overline{\xi}_t)-I_z(t,\overline{\zeta}_t)}{\theta_z(t,\overline{\xi}_t)I_z(t,\overline{\zeta}_t)}\bigg|$$
	$$\ll t\bigg(\int_{t_1}^{t_2}|f(u)| du +o(1)\bigg) \frac{|f(t)|}{|\theta_z(t,\overline{\xi}_t)||I_z(t,\overline{\zeta}_t)|}\ll \bigg(\int_{t_1}^{t_2}|f(u)| du +o(1)\bigg) |(\xi_t)'(t)|.
	\end{equation}
	
	From \eqref{riccati} we infer
	
	\begin{equation}\label{misc45}
	|\arg \theta_z(t,\overline{\xi}_t)-\arg \theta_z(t_1,\overline{\xi}_{t_1})| < \frac{\pi}{2}.
	\end{equation}
	
	Indeed, by the same trick as in the proof of Lemma \ref{Lemma_7}
	$$|\arg \theta_z(t,\overline{\xi}_t)-\arg \theta_z(t_1,\overline{\xi}_{t_1})|\leq\bigg|\int_{t_1}^{t_2} \frac{d}{dt} \arg \theta_{z}(t,\overline{\xi}_t) dt\bigg|\leq \bigg|\int_{t_1}^{t_2} s+(\Re\xi_t-s)+f(t)A(t) dt\bigg|$$
	
	$$\leq (t_2-t_1)|s|+A/t_1+\int_{t_1}^{t_2} |f(t)||A(t)|dt\leq1+o(1/t)+\frac{3\cosh[2A]}{100\cdot3\cosh[2A]}<\frac{\pi}{2}.$$
	
	Since $(t_1,t_2)$ is a $\sigma-$interval we deduce
	
	\begin{equation}\label{misc40}
	(\zeta_2-\zeta_1)-(\xi_2-\xi_1)=o(\xi_2-\xi_2).
	\end{equation}
 
 	Let now $\nu_{t_1}$ be any zero of $S(t_1,z)$ and $\nu_t$ its evolution. Then, by \eqref{misc28} and the property that $(t_1,t_2)$ is a $\sigma-$interval and $H_s-$interval, and $t_2-t_1\leq 1$ we find 
 	
 	\begin{equation}\label{misc32}
 	|\nu_{t_2}-\nu_{t_1}|\leq \frac{C_1}{t_1} \int_{t_1}^{t_2} |f(u)|du \leq \frac{2C_1}{t_2} \int_{t_1}^{t_2} |f(u)|du \ll \bigg| \int_{t_1}^{t_2} \frac{f(u)}{|\theta_z(u,\tilde{\xi}_u)|} du \bigg|\leq|\xi_2-\xi_1|.
 	\end{equation}
 	
 	Without loss of generality we assume $C(t)<\sqrt{t}.$ Then, all zeros of 
 	
 	$$I(t_1,z)=\overline{\alpha_{t_1}}^2\frac{\sin[t_1(z-(x_{t_1}+iy_{t_1}))]}{\sin[t_1(z-(x_{t_1}-iy_{t_1}))]}$$
 	
 	\noindent
 	have the same imaginary part and their real parts differ by at most $C(t_1)/t_1<1/\sqrt{t_1}.$ Hence, we find that their velocities differ at most by 
 	
 	\begin{equation}\label{misc50}
 	\ll t\bigg(\frac{1}{\sqrt{t}}+\int_{t_1}^{t_2} |f(u)|du\bigg) =o(|(\zeta_t)'(t)|).
 	\end{equation}

	Again, we find that their arguments $\arg[(\nu_t)'(t)]$ change by less than $\frac{\pi}{2}$  and thus we get 
	
	\begin{equation}\label{misc51}
	(\nu_{t_2}-\nu_{t_1})-(\xi_2-\xi_1)=o(\xi_2-\xi_1).
	\end{equation}
	
	This estimate together with \eqref{misc32} implies that there is a $D_1\in \C$ such that for all $z\in Q(s,A/t)$ we have 
	
	\begin{equation}\label{misc52}
	|S(t_2,z)-D_1\sin[t_2(z-(\zeta_1+(\xi_2-\xi_1)))]=o(t_2(\xi_2-\xi_1)).
	\end{equation}
	
	\noindent
	Indeed, one can show this by a standard argument comparing the two product representations and using that the zeros are close to each other. By the usual trick we can improve this to 
	
	$$\sup_{z \in Q(s,C(t)/t)} |S(t_2,z)-D_1\sin[t_2(z-(\zeta_1+(\xi_2-\xi_1)))]=o(t_2(\xi_2-\xi_1))$$
	
	\noindent
	by making $C(t)$ grow slower if necessary. We repeat the same argument for $\tilde{E}$ by defining $\tilde S(t,z)$ to be the solution of \eqref{misc26} with initial condition
	
	$$\tilde S(t_1,z)= \alpha_1 \frac{\gamma(t_1y_{t_1})}{\tilde{w}(s)} \cos[t_1(z-(\tilde x_{t_1}-iy_{t_1}))].$$
	
	Once again, we define $\tilde \zeta_1$ to be the zero of $\tilde E(t_1,z)$ closest to $\tilde x_1.$ By the exact same argument we obtain, making $C(t)$ grow slower if necessary that
	
	$$\sup_{z \in Q(s,C(t)/t)} |\tilde S(t_2,z)-D_2\cos[t_2(z-(\tilde \zeta_1+(\tilde\xi_2-\tilde\xi_1)))]=o(t_2(\tilde\xi_2-\tilde\xi_1)).$$

	Since $E$ and $\tilde E$ satisfy $(i)$ at $t=t_1$ we find as above with $\theta_z(t_1,\xi_t)$ and $I_z(t_1,\zeta_1)$ that
	
	$$|\theta_z(t_1,\overline{\xi}_1)-\tilde \theta_z(t_1,\overline{\tilde\xi}_1)|=o(t).$$
	
	Since $|\xi_1-\tilde\xi_1|\ll \frac{1}{t}$ and $\theta_z(t,\xi_t), \tilde \theta_z(t,\tilde\xi_t)$ satisfy \eqref{riccati} we find by using the same argument we used to show \eqref{misc39}that for all $t\in[t_1,t_2]$
	
	$$|\theta_z(t,\overline{\xi}_1)-\tilde \theta_z(t,\overline{\tilde\xi}_1)|=o(t).$$
	
	Hence, using \eqref{derivative_zero} one more time we have by the same calculation we used to prove \eqref{misc40} that
	
	$$(\xi_2-\xi_1)-(\tilde \xi_2-\tilde \xi_1)=o(\xi_2-\xi_1).$$
	
	Thus, we can replace $(\tilde \xi_2-\tilde \xi_1)$ by $(\xi_2-\xi_1)$ in the approximation above to obtain
	
	$$\sup_{z \in Q(s,C(t)/t)} |\tilde S(t_2,z)-D_2\cos[t_2(z-(\tilde \zeta_1+(\xi_2-\xi_1)))]=o(t_2(\xi_2-\xi_1)).$$
	
	The solution to \eqref{misc26} is unique and thus
	
	$$S(t_2,s)=E(t_2,s), \quad \tilde S(t_2,s)=\tilde E(t_2,s).$$
	
	Combining this with earlier estimates gives by the same calculation as in the proof of Lemma \ref{formulas_mif} $(i)$ that
	
	$$\frac{E(t_2,s)}{\tilde E(t_2,s)}=\frac{\sin[t_2(s-(\zeta_1+(\xi_2-\xi_1)))]}{\cos[t_2(s-(\tilde \zeta_1+(\xi_2-\xi_1)))]}+o(t_2(\xi_2-\xi_1)).$$

	\noindent
	Note that in particular we have at least $o(t_2(\xi_2-\xi_1))=o(1).$ Hoewever, generally $o(t_2(\xi_2-\xi_1))$ might be much smaller than $o(1).$ By the same way as in the proof of Lemma \ref{Lemma_9}  we find a constant $\Delta=o(1)(\xi_2-\xi_1)$ such that 
	 
	 \begin{equation}\label{misc33}
	 \frac{E(t_2,s)}{\tilde E(t_2,s)}=\frac{\sin[t_2(s-(\zeta_1+(\xi_2-\xi_1)+\Delta))]}{\cos[t_2(s-(\tilde \zeta_1+(\xi_2-\xi_1)+\Delta))]}.
	 \end{equation}
	 
	 Note that we still have
	 
	 $$\sup_{z \in Q(s,A/t)} |S(t_2,z)-D_1\sin[t_2(z-(\zeta_1+(\xi_2-\xi_1)+\Delta))]=o(t_2(\xi_2-\xi_1)),$$
	 
	 $$\sup_{z \in Q(s,A/t)} |\tilde S(t_2,z)-D_2\cos[t_2(z-(\tilde \zeta_1+(\xi_2-\xi_1)+\Delta))]=o(t_2(\xi_2-\xi_1)).$$
	 
	 We set
	 
	 $$x_{t_2}=\Re[\zeta_1+(\xi_2-\xi_1)+\Delta], \quad \tilde x_{t_2}=\Re[\tilde \zeta_1+(\xi_2-\xi_1)+\Delta]$$
	 
	 $$y_2=-\Im[\zeta_1+(\xi_2-\xi_1)+\Delta]=-\Im[\tilde \zeta_1+(\xi_2-\xi_1)+\Delta].$$
	 
	 We still need to show that $(ii)$ and $(iii)$ hold. Property $(ii)$ follows from the construction. We will now prove $(iii).$ Note that since $S$ and $\tilde S$ must satisfy a version of Corollary \ref{determinant_2} we can choose $D_1$ and $D_2$ so that
	 
	 \begin{equation}\label{misc34}
	 D_1\overline{D}_2>0, \quad |D_1D_2|=\frac{\gamma^2(t_2y_2)}{\sqrt{w(s)\tilde w(s)}}.
	 \end{equation}
 
	Indeed, if we want to have that $E(t_2,s)=S(t_2,s)$ and $\tilde E(t_2,s)=\tilde S(t_2,s)$ we must have
	
	$$\det \begin{pmatrix}
	D_1\sin[t_2(z-(\zeta_1+(\xi_2-\xi_1)+\Delta))] && D_2 \cos[t_2(z-(\tilde \zeta_1+(\xi_2-\xi_1)+\Delta))] \\
	\overline D_1 \sin[t_2(z-(\overline\zeta_1+(\overline\xi_2-\overline\xi_1)+\Delta))] && \overline D_2 \cos[t_2(z-(\overline{\tilde \zeta}_1+(\overline\xi_2-\overline\xi_1)+\Delta))]
	\end{pmatrix}=2i.$$

	\noindent
	Let $z,w \in \C.$ Then, we have 
	
	$$\sin(z)\cos(\overline{w})-\sin(\overline{z})\cos(w)=\frac{\sin(z-\overline{w})}{2}+\frac{\sin(z+\overline{w})}{2}-\frac{\sin(\overline{z}+w)}{2}-\frac{\sin(\overline{z}-w)}{2}$$
	
	$$=i(\Im \sin(z-\overline{w})+\Im \sin(z+\overline{w}))$$
	
	\noindent
	Now we set
	
	$$z=t_2(z-(\zeta_1+(\xi_2-\xi_1)+\Delta)), \quad w=t_2(z-(\tilde \zeta_1+(\xi_2-\xi_1)+\Delta)).$$
	
	\noindent
	Then,
	
	$$\Im \sin (z-\overline{w})=\Im\sin(t_2(\overline{\tilde{ \zeta_1}}-\zeta_1-2i\Im(\xi_2-\xi_1)+2i\Im \Delta))$$
	$$=\cos(t_2\Re(\tilde \zeta_1-\zeta_1))\sinh(2y_1-2\Im(\xi_2-\xi_1)+2\Im \Delta)$$
	$$=\sqrt{w(s)\tilde w(s)}\sinh[2t_2y_2],$$
	
	\noindent
	where we used property $(ii)$ and the definition of $y_2$ in the last step. Furthermore
	
	$$\Im(\sin(z+\overline{w}))=\Im \sin(t_2(2s-(\xi_1+\tilde \xi_1)-2\Re(\xi_2-\xi_1))-2\Re \Delta)=0.$$
	
	\noindent
	Hence, 	
	
		$$\det \begin{pmatrix}
	D_1\sin[t_2(z-(\zeta_1+(\xi_2-\xi_1)+\Delta))] && D_2 \cos[t_2(z-(\tilde \zeta_1+(\xi_2-\xi_1)+\Delta))] \\
	\overline D_1 \sin[t_2(z-(\overline\zeta_1+(\overline\xi_2-\overline\xi_1)+\Delta))] && \overline D_2 \cos[t_2(z-(\overline{\tilde \zeta}_1+(\overline\xi_2-\overline\xi_1)+\Delta))]
	\end{pmatrix}$$
	$$=i\sqrt{w(s)\tilde w(s)}\sinh[2t_2y_2].$$
	
	Thus, $D_1$ and $D_2$ have to satisfy as claimed
	
	$$ D_1\overline{D}_2>0, \quad |D_1D_2|=\frac{\gamma^2(t_2y_2)}{\sqrt{w(s)\tilde w(s)}}.$$
	 
	 If the point $s_1$ is satisifies $t_1s_1=t_1(s+(\tilde x_1-x_1))-\frac{\pi}{2}$ we get that 
	 
	 \begin{equation}\label{misc38}
	 \tilde S(t_1,s_1)=\sqrt{\frac{w(s)}{\tilde{w}(s)}}S(t_1,s).
	 \end{equation}
	 
	 By property $(ii)$ we have $|s-s_1|\leq |\tilde{x}_1-x_1|+\frac{\pi}{2t_1}<\frac{\pi}{t_1}\leq \frac{2\pi}{t_2}$ and for all $t\in (t_1,t_2)$
	 
	 $$\arg \frac{\tilde S(t,s_1)}{S(t,s)}<|s-s_1|(t_2-t_1)+2\int_{t_1}^{t_2}|f(u)|du< \frac{2\pi}{t_2}+2\int_{t_1}^{t_2} |f(u)|du.$$
	 
	 Indeed, as in the proof of Lemma \ref{Lemma_7} we find that 
	 
	 $$\frac{d}{dt} \arg\bigg(\frac{\tilde S(t,s_1)}{S(t,s)}\bigg)$$
	 
	 \noindent
	 can be calculated using $\proj_{iS(t,s)} \frac{d}{dt} S(t,s).$ By Lemma \ref{differential_equation} we find that 
	 
	 $$\proj_{iS(t,s)} \frac{d}{dt} S(t,s)=-\proj_{iS(t,s)} isS(t,s) +f(t) \proj_{iS(t,s)} \overline{S(t,s)}$$
	 $$=\bigg(-s+\frac{2A(t,s)C(t,s)}{|S(t,s)|^2}\bigg)iS(t,s).$$
	 
	 \noindent
	 Now by the same argument as in the proof of Lemma \ref{Lemma_7} we have
	 
	 $$\frac{d}{dt} \arg S(t,s)=-s+\frac{2F(t,s)G(t,s)}{|S(t,s)|^2},$$
	 
	 \noindent
	 where we wrote $S(t,z)=F(t,z)-iG(t,z).$ Note as well that $2F(t,s)G(t,s)\leq 4(F(t,s)^2+G(t,s)^2)=4|S(t,s)|^2.$
	 
	 By \eqref{misc38} we find $\arg(\tilde S(t,_1s_1))=\arg(S(t,s))$ and thus, as desired 
	 
	 $$\bigg|\arg \frac{\tilde S(t,s_1)}{S(t,s)}\bigg|=\bigg|\int_{t_1}^{t_2} \frac{d}{dt} \arg\bigg(\frac{\tilde S(t,s_1)}{S(t,s)}\bigg) dt \bigg| \leq |s-s_1|(t_2-t_1)+4\int_{t_1}^{t_2} |f(t)|dt.$$
	 
	 For the absolute values, taking into account the initial condition 
	 
	 $$|\tilde S(t_1,s_1)|=\sqrt{\frac{\tilde{w}(s)}{w(s)}}|S(t_1,s)|=D,$$
	 
	 \noindent
	 we have by applying twice the mean value Theorem for all $t\in[t_1,t_2]$
	 
	 $$\bigg||\tilde{S}(t,s_1)|-\sqrt{\frac{\tilde w(s)}{w(s)}}|S(t,s)|\bigg|=D\bigg|e^{\int_{t_1}^{t} f(u)\cos[2\arg \tilde S(u,s_1)]du}-e^{\int_{t_1}^{t} f(u)\cos[2\arg S(u,s)] du}\bigg|$$
	 
	 $$\ll \int_{t_1}^{t} |f(u)| \cdot \big|\cos[2\arg \tilde S(u,s_1)]-\cos[2\arg S(u,s)]\big| du$$
	 $$\ll \int_{t_1}^{t} |f(u)|\cdot|\arg(\tilde S(u,s_1))-\arg(S(u,s))|du$$
	 $$\ll \bigg(\frac{1}{t_2}+\int_{t_1}^{t_2}|f(u)|du\bigg)\int_{t_1}^{t_2}|f(u)|du.$$
	 
	 Note that we used in the first application of the mean value Theorem that 
	 
	 $$\int_{t_1}^{t_2} |f(t)|dt < \frac{1}{100\cosh[2A]},$$
	 
	 \noindent
	 so that the bound is not dependent on the interval $(t_1,t_2).$ In particular we showed
	 
	 \begin{equation}\label{misc54}
	 \bigg||\tilde{S}(t_2,s_1)|-\sqrt{\frac{\tilde w(s)}{w(s)}}|S(t_2,s)|\bigg| \ll \bigg(\frac{1}{t_2}+\int_{t_1}^{t_2}|f(u)|du\bigg)\int_{t_1}^{t_2}|f(u)|du.
	 \end{equation}
	 
	 Combining this estimate with \eqref{misc34} now implies that the constants $D_1,D_2$ can be chosen so that 
	 
	 \begin{equation}\label{misc53}
	 D_1=\alpha\frac{\gamma(t_2y_2)}{\sqrt{w(s)}}, \quad D_2=\alpha\frac{\gamma(t_2y_2)}{\tilde{w}(s)}
	 \end{equation}
	  
	 \noindent
	 for some unimodular constant $\alpha.$ Using trigonometric identities we find as above
	 
	 $$\det \begin{pmatrix}
	 \alpha\frac{\gamma(t_2y_2)}{\sqrt{w(s)}}\sin[t_2(z-(\zeta_1+(\xi_2-\xi_1)+\Delta))] && \alpha\frac{\gamma(t_2y_2)}{\tilde{w}(s)} \cos[t_2(z-(\tilde \zeta_1+(\xi_2-\xi_1)+\Delta))] \\
	\alpha\frac{\gamma(t_2y_2)}{\sqrt{w(s)}} \sin[t_2(z-(\overline\zeta_1+(\overline\xi_2-\overline\xi_1)+\Delta))] && \alpha\frac{\gamma(t_2y_2)}{\tilde{w}(s)} \cos[t_2(z-(\overline{\tilde \zeta}_1+(\overline\xi_2-\overline\xi_1)+\Delta))]
	 \end{pmatrix}=2i.$$
	 
	 Combining Corollary \ref{determinant_2} and Lemma \ref{determinant_3} we find
	 
	 $$|E(t_2,s)|=\bigg|\alpha\frac{\gamma(t_2y_2)}{\sqrt{w(s)}}\sin[t_2(z-(\zeta_1+(\xi_2-\xi_1)+\Delta))]\bigg|,$$
	 
	 $$|\tilde E(t_2,s)|=\bigg|\alpha\frac{\gamma(t_2y_2)}{\sqrt{w(s)}} \sin[t_2(z-(\overline\zeta_1+(\overline\xi_2-\overline\xi_1)+\Delta))]\bigg|.$$
	 
	 We can now modify $\alpha$ such that $(iii)$ holds and define it to be $\alpha_{t_2}.$
	\end{proof}

	\subsection{Pointwise convergence of the non-linear Fourier transform}
	In this section, we will finally prove Theorem \ref{non-linear carleson}. We begin with the following result.
	\begin{theorem}\label{measure4}
		$|S|=0$
	\end{theorem}

	\begin{proof} By definition we have for all $k\geq 1$ that
		$$S\subset \bigcup_{n=k}^{\infty} S^n \subset \bigcup_{n=k}^{\infty} S_V^n \cup \bigcup_{n=k}^{\infty} S_H^n.$$
		
		\noindent
		Combining this with Corollary \ref{claim_1} and Corollary \ref{claim_3} we find
		
		$$|S|\leq\sum_{n=k}^{\infty} |S^n|=\sum_{n=k}^{\infty} |S_V^n| + \sum_{n=k}^{\infty} |S_H^n| \leq 2D \sum_{n=1}^{\infty} ||f||_{L^2([2^n,2^{n+2}])}^2=2D||f||_{L^2(\R)}^2<\infty.$$
		
		\noindent
		Since each $s \in S$ belongs to infinitely many $|S^n|,$ we can estimate $|S|$ against the tail of the first series, which tends to zero since it converges.
	\end{proof}
	
	We state once again the non-linear version of Carleson's Theorem.
	
	\begin{theorem*}[Non-linear Carleson's theorem]
		Let $f$ be in $L^2(\R_+).$ Then, for almost all $s\in \R$ 
		
		$$\lim_{T\rightarrow \infty} f_{T}^{\dagger}(s) = f^{\dagger}(s).$$  
	\end{theorem*}

	\begin{proof}
		By Theorem \ref{measure4} and Corollary \ref{corollary_4} we find that for almost all $s \in \R$ and all $D>0$ there is a unimodular function $\alpha(s,t)$ and a constant $\beta \in \C$ such that 
		
		$$\sup_{z \in Q(s,D/t)} \bigg| E(t,z) - \frac{-i\alpha(s,t)}{\sqrt{w(s)}} e^{itz}\bigg|=o(1),$$
		
		$$\sup_{z \in Q(s,D/t)} \bigg| \tilde E(t,z) - \frac{-i\beta\alpha(s,t)}{\sqrt{\tilde w(s)}} e^{itz}\bigg|=o(1),$$
		
		\noindent
		as $t\rightarrow \infty.$ By Corollary \ref{determinant_2} we can choose $\beta=e^{-i\varphi}$ with $\varphi=\arcsin\sqrt{w(s)\tilde w(s)}.$ Indeed, initially we have by Corollary \ref{corollary_4} that there exists a function $\alpha'(s,t)$ such that 
		
		$$\sup_{z \in Q(s,D/t)} \bigg| \tilde E(t,z) - \frac{-i\alpha'(s,t)}{\sqrt{\tilde w(s)}} e^{itz}\bigg|=o(1).$$
		
		We now define $\beta(t), |\beta(t)|=1$ such that $\beta(t)\alpha(s,t)=\alpha'(s,t).$ Plugging in this approximation in the determinant equation from Corollary \ref{determinant_2} we find
		
		$$\det\begin{pmatrix}
		\frac{-i\alpha(s,t)}{\sqrt{w(s)}} e^{its} && \frac{-i \beta(t)\alpha(s,t)}{\sqrt{\tilde w(s)}} e^{its} \\
		\frac{i\overline{\alpha}(s,t)}{\sqrt{w(s)}} e^{-its} && \frac{i \overline\beta(t)\overline\alpha(s,t)}{\sqrt{\tilde w(s)}} e^{-its}
		\end{pmatrix}=\frac{1}{\sqrt{w(s)\tilde{w}(s)}}\big(\overline{\beta}(t)-\beta(t)\big)=2i+o(1).$$
		
		\noindent
		Hence, in the limit
		
		$$\lim_{t\rightarrow \infty} \Im \beta(t)= \lim_{t\rightarrow \infty} \frac{\beta(t)-\overline\beta (t)}{2i}=-\sqrt{w(s)\tilde w(s)}.$$
		
		\noindent
		Together with the fact that $|\beta(t)|=1$ it follows that we can choose $\beta=e^{-i\varphi}$ with $\varphi=\arcsin\sqrt{w(s)\tilde w(s)}$ in the approximation above.
		
		By definiton we have
		
		$$|a(t,s)|=\bigg|\frac{e^{its}E(t,s)+ie^{its}\tilde E(t,s)}{2}\bigg|=\frac{1}{2}|E(t,s)+i\tilde E(t,s)|.$$
		
		\noindent
		Combining the above estimates gives for almost all $s\in \R$ that 
		
		$$\bigg|E(t,s)+i\tilde E(t,s)-\bigg(\frac{-i\alpha(s,t)}{\sqrt{w(s)}} + i\frac{-i\beta\alpha(s,t)}{\sqrt{\tilde w(s)}} \bigg)e^{its}\bigg|=o(1).$$
		
		\noindent
		Doing the same calculations for $b(t,s)$ and applying the inverse triangle inequality now yields that 
		
		$$|a(t,s)|=	\frac{1}{2}|E(t,s)+i\tilde E(t,s)|\overset{t\rightarrow\infty}{\longrightarrow}	\frac{1}{2}\bigg|\frac{1}{\sqrt{w(s)}}+\frac{i\beta}{\sqrt{\tilde{w}(s)}}\bigg|$$
		
		$$|b(t,s)|=\frac{1}{2}|E(t,s)-i\tilde E(t,s)|\overset{t\rightarrow\infty}{\longrightarrow}	\frac{1}{2}\bigg|\frac{1}{\sqrt{w(s)}}-\frac{i\beta}{\sqrt{\tilde{w}(s)}}\bigg|.$$
		
		Furthermore, we compute
		
		$$\arg\bigg(\frac{b}{a}\bigg)=\arg b(t,s)- \arg a(t,s)=\arg\big[E(t,s)-i\tilde E(t,s)\big]-\arg\big[E(t,s)+i\tilde E(t,s)\big]$$
		
		$$=\arg(-i\alpha(s,t))+\arg\bigg(\frac{1}{\sqrt{w(s)}}-\frac{i\beta}{\sqrt{\tilde{w}(s)}}\bigg)+tx-\arg(-i\alpha(s,t))-\arg\bigg(\frac{1}{\sqrt{w(s)}}+\frac{i\beta}{\sqrt{\tilde{w}(s)}}\bigg)-tx+o(1).$$
		
		$$=\arg\bigg(\frac{\frac{1}{\sqrt{w(s)}}-\frac{i\beta}{\sqrt{\tilde{w}(s)}}}{\frac{1}{\sqrt{w(s)}}+\frac{i\beta}{\sqrt{\tilde{w}(s)}}}\bigg)=o(1).$$
		
		Thus, for almost all $s\in \R$ we have shown that
		
		$$\bigg|\frac{b}{a}\bigg|\overset{t\rightarrow \infty}{\longrightarrow} \bigg|\frac{\frac{1}{\sqrt{w(s)}}-\frac{i\beta}{\sqrt{\tilde{w}(s)}}}{\frac{1}{\sqrt{w(s)}}+\frac{i\beta}{\sqrt{\tilde{w}(s)}}}\bigg|, \quad \arg\bigg(\frac{b}{a}\bigg)\overset{t\rightarrow \infty}{\longrightarrow} \arg \bigg(\frac{\frac{1}{\sqrt{w(s)}}-\frac{i\beta}{\sqrt{\tilde{w}(s)}}}{\frac{1}{\sqrt{w(s)}}+\frac{i\beta}{\sqrt{\tilde{w}(s)}}}\bigg).$$
		
		Finally, by combining the convergence of the absolute value and the argument we find that for almost all $s\in \R$
		
		$$\lim_{T\rightarrow \infty} f_{T}^{\dagger}(s)=\lim_{T\rightarrow \infty} \frac{b(T,s)}{a(T,s)}=\frac{\frac{1}{\sqrt{w(s)}}-\frac{i\beta}{\sqrt{\tilde{w}(s)}}}{\frac{1}{\sqrt{w(s)}}+\frac{i\beta}{\sqrt{\tilde{w}(s)}}}=f^{\dagger}(s).$$
		
		\noindent
		proving that the non-linear Fourier transform converges pointwise almost everywhere.
	
	\end{proof}

	\section{Further results}
	In the last chapter we have shown that $|a(t,s)|$ and $|b(t,s)|$ converge for almost all $s\in \R$ as $t \rightarrow \infty.$ It is natural to ask whether $a(t,s)$ and $b(t,s)$ converge too. We are going to investigate this problem by tracing it back to the question of convergence of an ordinary Fourier integral. We will frequently abuse the notation by using the letter $t$ in the integrand as well as in the domain of integration.
	
	\begin{lemma}
		$$\frac{d}{dt} |a(t,s)|=2f(t)\frac{|E(t,s)|^2\cos[2\arg E(t,s)]+|\tilde E(t,s)|^2\cos[2\arg \tilde{E}(t,s)]}{|a(t,s)|^2}.$$
	\end{lemma}
	
	\begin{proof}
		We have the equation 
		
		$$\frac{d}{dt} |a(t,s)|=\frac{\Re(a(t,s)\overline{\frac{d}{dt}a(t,s)})}{|a(t,s)|^2}.$$
		
		To make the next calculations more clear, we will supress from time to time the arguments of the appearing functions. Since $|e^{its}|=1$ we can ignore the exponential term and it is sufficient to calculate  with the same formula as above the derivative
		
		$$\frac{d}{dt} |E(t,s)+i\tilde{E}(t,s)|.$$
		
		By Lemma \ref{differential_equation} we have that 
		
		$$\frac{d}{dt} E(t,s)+i\tilde E(t,s) = -isE(t,s)+f(t)E^{\sharp}(t,s)+i[-is\tilde{E}(t,s)+f(t)\tilde{E}^{\sharp}(t,s)]$$
		$$=-isE(t,s)+s\tilde{E}(t,s)+f(t)E^{\sharp}(t,s)+if(t)\tilde{E}^{\sharp}(t,s).$$
		
		Thus, setting $\tilde a(t,s)=E(t,s)+iE(t,s)$ we find that 
		
		$$\Re \bigg(\tilde a(t,s)\overline{\frac{d}{dt}\tilde{a}(t,s)}\bigg)=(E+i\tilde{E})(isE^{\sharp}+s\tilde{E}^{\sharp}+f(t)E-if(t)\tilde{E})$$
		
		$$=s(AB+CD)+f(t)(A^2-C^2)+ f(t)(B^2-D^2)-s(AB+CD)=f(t)[A^2-C^2]+f(t)[B^2-D^2].$$
		
		$$=f(t)|E(t,s)|^2\cos[2\arg E(t,s)]+f(t)|\tilde{E}(t,s)|\cos[2\arg \tilde E(t,s)].$$
		
		Hence,
		
		$$\frac{d}{dt}|a(t,s)|=2\frac{f(t)|E(t,s)|^2\cos[2\arg E(t,s)]+f(t)|\tilde{E}(t,s)|^2\cos[2\arg \tilde{E}(t,s)]}{|a(t,s)|^2}$$
	\end{proof}
	
	We proceed by estimating $\arg a.$ First note that 
	
	$$\arg a(t,s)=\arg\bigg(\frac{e^{its}}{2}[E(t,s)+i\tilde E(t,s)]\bigg)=ts+ \arg \tilde{a}.$$
	
	Thus, 
	
	$$|a(t,s)|=1+ 2\int_{0}^{t} f(t) \frac{|E(t,s)|^2\cos[2\arg E(t,s)]+|\tilde E(t,s)|^2\cos[2\arg \tilde{E}(t,s)]}{|a(t,s)|^2} dt.$$
	
	\begin{lemma}
		$$\frac{d}{dt} \tilde{a}(t,s) = -s -2f(t)\bigg[|E(t,s)|^2\sin[2\arg E(t,s)]+|\tilde{E}(t,s)|^2\sin[2\arg\tilde E(t,s)]\bigg].$$
	\end{lemma}
	
	\begin{proof}
		$$\frac{d}{dt} \arg \tilde{a} = \proj_{i(E(t,s)+i\tilde{E}(t,s))} \frac{d}{dt} E(t,s)+ i \tilde{E}(t,s)$$
		
		$$=\proj_{-\tilde{E}(t,s)+iE(t,s)} -isE(t,s)+s\tilde{E}(t,s) +f(t) E^{\sharp}(t,s)+if(t)\tilde{E}^{\sharp}(t,s)$$
		
		$$=\proj_{-\tilde{E}(t,s)+iE(t,s)} -isE(t,s)+s\tilde{E}(t,s) +f(t)\proj_{-\tilde{E}(t,s)+iE(t,s)} E^{\sharp}(t,s)+i\tilde{E}^{\sharp}(t,s).$$
		
		$$=-s +f(t) \frac{\bigg\langle \begin{pmatrix}
			A-D \\
			B+C
			\end{pmatrix}, \begin{pmatrix}
			-B+C \\ A+D
			\end{pmatrix} \bigg\rangle}{|a(t,s)|^2}=-s+\frac{2[AC+BD]}{|a(t,s)|^2}.$$
		
		$$=-s -2 \frac{|E(t,s)|^2\sin[2\arg E(t,s)]+|\tilde{E}(t,s)|^2\sin[2\arg \tilde{E}(t,s)]}{|a(t,s)|^2}.$$
	\end{proof}
	
	Thus, 
	
	$$\arg a(t,s)= -2\int_{0}^{t} f(t) \frac{|E(t,s)|^2\sin[2\arg E(t,s)]+|\tilde{E}(t,s)|^2\sin[2\arg \tilde{E}(t,s)]}{|a(t,s)|^2} dt.$$\begin{lemma}
		$$\frac{d}{dt} |a(t,s)|=2f(t)\frac{|E(t,s)|^2\cos[2\arg E(t,s)]+|\tilde E(t,s)|^2\cos[2\arg \tilde{E}(t,s)]}{|a(t,s)|^2}.$$
	\end{lemma}
	
	\begin{proof}
		We have the equation 
		
		$$\frac{d}{dt} |a(t,s)|=\frac{\Re(a(t,s)\overline{\frac{d}{dt}a(t,s)})}{|a(t,s)|^2}.$$
		
		To make the next calculations more clear, we will supress from time to time the arguments of the appearing functions. Since $|e^{its}|=1$ we can supress the exponential term and it is sufficient to calculate with the same formula as above
		
		$$\frac{d}{dt} |E(t,s)+i\tilde{E}(t,s)|.$$
		
		\noindent
		By Lemma \ref{differential_equation} we have that 
		
		$$\frac{d}{dt} E(t,s)+i\tilde E(t,s) = -isE(t,s)+f(t)E^{\sharp}(t,s)+i[-is\tilde{E}(t,s)+f(t)\tilde{E}^{\sharp}(t,s)]$$
		$$=-isE(t,s)+s\tilde{E}(t,s)+f(t)E^{\sharp}(t,s)+if(t)\tilde{E}^{\sharp}(t,s).$$
		
		\noindent
		Thus, setting $\tilde a(t,s)=E(t,s)+iE(t,s)$ we find that 
		
		$$\Re \bigg(\tilde a(t,s)\overline{\frac{d}{dt}\tilde{a}(t,s)}\bigg)=(E+i\tilde{E})(isE^{\sharp}+s\tilde{E}^{\sharp}+f(t)E-if(t)\tilde{E})$$
		
		$$=s(AB+CD)+f(t)(A^2-C^2)+ f(t)(B^2-D^2)-s(AB+CD)=f(t)[A^2-C^2]+f(t)[B^2-D^2].$$
		
		$$=f(t)|E(t,s)|^2\cos[2\arg E(t,s)]+f(t)|\tilde{E}(t,s)|\cos[2\arg \tilde E(t,s)].$$
		
		\noindent
		Hence,
		
		$$\frac{d}{dt}|a(t,s)|=2\frac{f(t)|E(t,s)|^2\cos[2\arg E(t,s)]+f(t)|\tilde{E}(t,s)|^2\cos[2\arg \tilde{E}(t,s)]}{|a(t,s)|^2}$$
	\end{proof}

	\noindent
	Thus, we can write
	
	$$|a(t,s)|=1+ 2\int_{0}^{t} f(t) \frac{|E(t,s)|^2\cos[2\arg E(t,s)]+|\tilde E(t,s)|^2\cos[2\arg \tilde{E}(t,s)]}{|a(t,s)|^2} dt.$$

	We proceed by estimating $\arg a.$ First note that 
	
	$$\arg a(t,s)=\arg\bigg(\frac{e^{its}}{2}[E(t,s)+i\tilde E(t,s)]\bigg)=ts+ \arg \tilde{a}.$$

	\begin{lemma}
		$$\frac{d}{dt} \tilde{a}(t,s) = -s -2f(t)\bigg[|E(t,s)|^2\sin[2\arg E(t,s)]+|\tilde{E}(t,s)|^2\sin[2\arg\tilde E(t,s)]\bigg].$$
	\end{lemma}
	
	\begin{proof}
		$$\frac{d}{dt} \arg \tilde{a} = \proj_{i(E(t,s)+i\tilde{E}(t,s))} \frac{d}{dt} E(t,s)+ i \tilde{E}(t,s)$$
		
		$$=\proj_{-\tilde{E}(t,s)+iE(t,s)} -isE(t,s)+s\tilde{E}(t,s) +f(t) E^{\sharp}(t,s)+if(t)\tilde{E}^{\sharp}(t,s)$$
		
		$$=\proj_{-\tilde{E}(t,s)+iE(t,s)} -isE(t,s)+s\tilde{E}(t,s) +f(t)\proj_{-\tilde{E}(t,s)+iE(t,s)} E^{\sharp}(t,s)+i\tilde{E}^{\sharp}(t,s).$$
		
		$$=-s +f(t) \frac{\bigg\langle \begin{pmatrix}
			A-D \\
			B+C
			\end{pmatrix}, \begin{pmatrix}
			-B+C \\ A+D
			\end{pmatrix} \bigg\rangle}{|a(t,s)|^2}=-s+\frac{2[AC+BD]}{|a(t,s)|^2}.$$
		
		$$=-s -2 \frac{|E(t,s)|^2\sin[2\arg E(t,s)]+|\tilde{E}(t,s)|^2\sin[2\arg \tilde{E}(t,s)]}{|a(t,s)|^2}.$$
	\end{proof}
	
	Thus, we have established 
	
	$$\arg a(t,s)= -2\int_{0}^{t} f(t) \frac{|E(t,s)|^2\sin[2\arg E(t,s)]+|\tilde{E}(t,s)|^2\sin[2\arg \tilde{E}(t,s)]}{|a(t,s)|^2} dt.$$

	\begin{theorem}
		Let $s\in \R.$ We have the following equivalence. As $t \rightarrow \infty$
		
		$$ a(t,s) \text{ converges if and only if the Fourier integral} \int_{0}^{t} f(t)\frac{b(t,s)}{a^\sharp(t,s)}e^{-2its}dt \text{ converges}.$$
		
	\end{theorem}

	\begin{proof}
		$$|a(t,s)|-1-i\arg a(t,s) = 2 \int_{0}^{t} f(t)\frac{\bigg[|E(t,s)|^2e^{2i\arg E(t,s)}+|\tilde E(t,s)|^2e^{2i\arg \tilde E(t,s)}\bigg]}{|a(t,s)|^2} dt$$
		$$ =2\int_{0}^{t} f(t) \frac{E(t,s)^2+\tilde E(t,s)^2}{|a(t,s)|^2} dt= 2\int_{0}^{t} f(t) \frac{(E(t,s)+i\tilde E(t,s))(E(t,s)-i\tilde E(t,s))}{|a(t,s)|^2} dt$$
		$$ = 8 \int_{0}^{t} f(t) \frac{a(t,s)b(t,s)}{|a(t,s)|^2} e^{-2its} dt=8\int_{0}^{t}f(t) \frac{b(t,s)}{\overline{a(t,s)}} e^{-2its} dt.$$
	\end{proof}

	Notice that $b(t,s)/a^\sharp(t,s)\leq 1$ for all $s\in \R$ and thus we obtain that if $f(t) \in L^p(\R)$ for $1\leq p < \infty,$ then as well for all $s\in \R$
	
	$$f(t)\frac{b(t,s)}{a^\sharp(t,s)} \in L^p(\R).$$
	
	We obtain immediately the following 
	
	\begin{corollary}
		If $f(t) \in L^1(\R),$ then $a(t,s)$ converges for all $s \in \R.$
	\end{corollary}

	Unfortunately we can not extend this with out furhter work to the case $1<p<\infty.$ To that end we introduce for $(s,y) \in \R^2$ and $t>0$ the function 
	
	$$I(t,s,y)=\int_{0}^{t} \int_{0}^{t} f(t)\frac{b(t,s)}{a^\sharp(t,s)}e^{-2ity}dt.$$
	
	Then, by the linear Carleson Theorem we find that for almost all $(s,y) \in \R^2$ the following limit exists
	
	$$\lim_{t\rightarrow \infty} I(t,s,y).$$
	
	The statement that $a(t,s)$ converges almost everywhere on $\R$ is equivalent to the statement that the above integral convergeges almost everyhwere on the main diagonal in $\R^2.$ Therefore, we can think about the question, whether $a(t,s)$ converges almost everywhere, being as difficult as proving that the following Fourier integral exists 
	
	$$\int_{0}^{\infty} f(t) f_{t}^{\dagger}(s)e^{-its} dt.$$

\end{document}